\documentclass[12pt]{amsart}
\usepackage{amsfonts,amssymb, amscd, latexsym, graphicx, psfrag, color, float}

\usepackage[all,2cell]{xy}
\UseTwocells

\usepackage{mathrsfs}
\usepackage{eucal}

\usepackage[colorlinks,
             linkcolor=black,
             citecolor=blue,
             bookmarksnumbered,hyperindex,hyperfigures,pdfpagelabels]{hyperref}
\usepackage{todonotes}
\usepackage{tikz-cd}
\usepackage{graphicx}
\usepackage{xfrac}

\newtheorem{dummy}{dummy}[section]
\newtheorem{lem}[dummy]{Lemma}
\newtheorem{thm}[dummy]{Theorem}
\newtheorem*{theorem*}{Theorem}

\newtheorem{cor}[dummy]{Corollary}
\newtheorem{prop}[dummy]{Proposition}

\theoremstyle{definition}
\newtheorem{dfn}[dummy]{Definition}
\newtheorem{ex}[dummy]{Example}
\newtheorem{rmk}[dummy]{Remark}
\newtheorem{warning}[dummy]{Warning}

\makeatletter
\newcommand\incircbin
{%
  \mathpalette\@incircbin
}
\newcommand\@incircbin[2]
{%
  \mathbin%
  {%
    \ooalign{\hidewidth$#1#2$\hidewidth\crcr$#1\bigcirc$}%
  }%
}
\newcommand{\oinfty}{\incircbin{\mkern-13.73mu\smalli}\!}
\makeatother

\newcommand{\smalli}{\resizebox{.7cm}{0.29cm}{$\phantom{x}_{\phantom{.}^\i}$}}

\newcommand{\bA}{\mathbb{A}}

\newcommand{\bR}{\mathbb{R}}
\newcommand{\bZ}{\mathbb{Z}}
\newcommand{\bT}{\mathbb{T}}

\newcommand{\iCat}{\mathrm{Cat}_{\i}}

\newcommand{\cB}{\mathcal{B}}

\newcommand{\cE}{\mathcal{E}}
\newcommand{\cF}{\mathcal{F}}

\newcommand{\cM}{\mathcal{M}}

\newcommand{\cO}{\mathcal{O}}

\renewcommand{\O}{\cO}

\newcommand{\cS}{\mathcal{S}}

\newcommand{\cZ}{\mathcal{Z}}

\newcommand{\sA}{\mathscr{A}}
\newcommand{\sB}{\mathscr{B}}
\newcommand{\sC}{\mathscr{C}}
\newcommand{\sD}{\mathscr{D}}

\newcommand{\Loc}{\operatorname{Loc}_{\Cart}}

\newcommand{\fp}{\mathbf{fp}}
\newcommand{\lex}{\mathbf{lex}}

\newcommand{\Spec}{\mathrm{Spec}\,}

\newcommand{\Sh}{\operatorname{Shv}}

\newcommand{\Hom}{\mathrm{Hom}}

\newcommand{\Spc}{\cS\!\operatorname{pc}}

\newcommand{\Map}{\mathbb{M}\mathrm{ap}}

\def\Top{\mathbf{Top}}

\def\Alg{\textbf{Alg}}
\def\Comk{\textbf{Com}_k}
\def\ComR{\textbf{Com}_\bR}

\renewcommand{\i}{\infty}

\def\Ind{\operatorname{Ind}}
\def\Fun{\operatorname{Fun}}
\DeclareMathOperator{\Psh}{Psh}

\def\colim{\underrightarrow{\mbox{colim}\vspace{0.5pt}}\mspace{4mu}}

\renewcommand{\lim}{\varprojlim\mspace{3mu}}
\def\blank{\mspace{3mu}\cdot\mspace{3mu}}

\def\Lan{\operatorname{Lan}}
\def\Ran{\operatorname{Ran}}

\def\Top{\mathbf{Top}}
\def\Set{\mathrm{Set}}

\def\Mfd{\mathbf{Mfd}}
\def\Cart{\mathbf{C}^\infty}
\def\DMfd{\mathbf{DMfd}}

\def\longlongrightarrow{-\!\!\!-\!\!\!-\!\!\!-\!\!\!-\!\!\!-\!\!\!\longrightarrow}
\def\longlonglongrightarrow{-\!\!\!-\!\!\!-\!\!\!-\!\!\!-\!\!\!-\!\!\!\longlongrightarrow}

\newcommand*{\longhookrightarrow}{\ensuremath{\lhook\joinrel\relbar\joinrel\rightarrow}}

\def\Speci{\Spec_{\Cart}}
\def\Mod{\mathbf{Mod}}

\begin{document}

\title{On the Universal Property of Derived Manifolds}
\author{David Carchedi}
\thanks{The first author was supported by the National Science Foundation under Grant No. 1811864}
\author{Pelle Steffens}
\thanks{The second author has received funding from the European Research Council (ERC) under the European Union’s Horizon 2020 research and innovation programme (grant agreement No 768679)}
\maketitle

\begin{abstract}
It is well known that any model for derived manifolds must form a higher category. In this paper, we propose a universal property for this higher category, classifying it up to equivalence. Namely, the $\i$-category $\DMfd$ of derived manifolds has finite limits, is idempotent complete, and receives a functor from the category of manifolds which preserves transverse pullbacks and the terminal object, and moreover is \emph{universal} with respect to these properties. We then show this universal property is equivalent to another one, intimately linking the $\i$-category of derived manifolds to the theory of $C^\i$-rings. More precisely, $\mathbb{R}$ is a $C^\i$-ring object in $\DMfd$, and the pair $\left(\DMfd,\mathbb{R}\right)$ is universal among idempotent complete $\i$-categories with finite limits and a $C^\i$-ring object. We then show that (a slight extension beyond the quasi-smooth setting of) Spivak's original model satisfies our universal property.
\end{abstract}

\tableofcontents

\section{Introduction}
The role of  transversality in differential topology has a rich history. It is well known that non-transverse pullbacks need not exist in the category of manifolds, and when they do, they do not have the same cohomological properties one would expect of a good intersection. Derived manifolds generalize the concept of smooth manifolds to allow arbitrary (iterative) intersections to exist as smooth objects, regardless of transversality. An argument going back to Spivak \cite{spivak} shows that in order to have the expected cohomological properties with respect to cobordism theory, derived manifolds must form a higher category. There exists to date several approaches to derived geometry in the $C^\i$-setting, all of them higher categorical in nature e.g. \cite{spivak,joyce,derivjustden,dg2}. Despite the wealth of different approaches, there has been as of yet little attention given to characterizing the resulting theory by a universal property, with the notable exception of recent work of Macpherson \cite{andrew}. In contrast, in Lurie's foundations on derived algebraic geometry (\cite{dagv}), the role of the universal property of derived geometry is very explicit in the passage from smooth affine spaces (of finite type) to `fully derived' spaces (of finite type). While differential geometry and the existence of a theory of derived manifolds are mentioned in \cite{dagv}, no tractable construction of the $\i$-category of derived manifolds is given in that paper.

In this paper, we rectify this situation and formulate a precise universal property for the $\i$-category of derived manifolds $\DMfd,$ thus characterizing it once and for all up to equivalence. Firstly, whatever the $\i$-category $\DMfd$ is, it certainly should receive a fully faithful  functor from the category of manifolds $$i:\Mfd \hookrightarrow \DMfd,$$ and this functor should preserve transverse pullback and the terminal object (this is part of Spivak's necessary conditions \cite{spivak} for an $\i$-category to be ``good for doing intersection theory.''). Moreover, since we will be concerned with general not necessarily quasi-smooth derived manifolds, $\DMfd$ should have finite limits and be idempotent complete (the former implies that later in the setting of $n$-categories for finite $n,$ but not for $\i$-categories). We propose the following as the universal property of the $\i$-category $\DMfd$:\\

\textbf{Universal Property 1}: \emph{For any idempotent complete $\i$-category $\sC$ with finite limits, composition with $i$ induces an equivalence of $\i$-categories
$$\Fun^{\lex}\left(\DMfd,\sC\right) \stackrel{\sim}{\longlongrightarrow} \Fun^{\pitchfork}\left(\Mfd,\sC\right)$$
between the $\i$-category of functors from derived manifolds to $\sC$ which preserve finite limits, and the $\i$-category of functors from manifolds to $\sC$ which preserve transverse pullbacks and the terminal object.}\\

In this paper, we argue that this is the ``correct'' universal property, since, in particular, it reproduces, almost for free, an extension of Spivak's model to beyond the quasi-smooth setting. In particular, Spivak's $\i$-category of quasi-smooth manifolds (of finite type) embed fully faithfully into any $\i$-category satisfying the above universal property. The proposed characterization of $\DMfd$ should also be compared with Lurie's notions of `pregeometries' and `geometric envelopes', introduced in \cite{dagv}. Roughly, a \emph{pregeometry} is an $\i$-category $\mathcal{T}$ equipped with a collection of distinguished pullbacks, which should be thought of as singling out the `geometrically correct' intersections. A \emph{geometric envelope} of a pregeometry is then an $\i$-category $\mathcal{G}$ obtained by freely adding finite limits and retracts to $\mathcal{T}$, while preserving the pullbacks we had deemed to be correct. This procedure defines a functor $\mathcal{T}\rightarrow \mathcal{G}$ which satisfies a universal property that is exactly analogous to the one above.    

We furthermore show that the above universal property is equivalent to the following one:\\

\textbf{Universal Property 2}: \emph{For any idempotent complete $\i$-category $\sC$ with finite limits, there is an equivalence of $\i$-categories
$$\Fun^{\lex}\left(\DMfd,\sC\right) \stackrel{\sim}{\longlongrightarrow} \Alg_{\Cart}\left(\sC\right)$$
between the $\i$-category of functors from derived manifolds to $\sC$ which preserve finite limits, and the $\i$-category of $C^\i$-ring objects in $\sC.$}\\

The equivalence of these two universal properties is the content of Theorem \ref{thm:univ2}.

As a corollary, we deduce the following:

\begin{cor}
There is a canonical equivalence of $\i$-categories $$\DMfd \simeq \left(\Alg_{\Cart}\left(\Spc\right)^{\fp}\right)^{op}$$
between the $\i$-category of derived manifolds, and the opposite of the $\i$-category of homotopically finitely presented simplicial $C^\i$-rings.
\end{cor}

This result is closely related to the fact that quasi-smooth derived manifolds of finite type are ``affine,'' which is the main result of \cite{derivjustden}. The above corollary shows that all finite type derived manifolds are affine, and they are precisely the $C^\i$-spectra of homotopically finitely presented simplicial $C^\i$-rings.

\subsection{Organization}
In Section \ref{sec:univ}, we describe in detail the universal property of the $\i$-category $\DMfd$ of derived manifolds, and deduce some of the basic consequences of this universal property.

In Section \ref{sec:algthy} we lay down the basic foundations for algebraic theories in the context of $\i$-categories. This includes some new results, for example, the theory of unramified transformations of algebraic theories, based on work of Lurie in \cite{dagix} in a more general setting. While the results we present on unramified transformations can be formally deduced from the results in \cite{dagix}, our focus on algebraic theories allows us to significantly streamline the arguments.  

In Section \ref{sec:prop}, we delve into the study of simplicial $C^\i$-rings, and establish their basic properties. The main results of this section are the following:

\begin{thm}
The functor $\sC^\i:\Mfd \hookrightarrow \Alg_{\Cart}\left(\Spc\right)^{op}$ from the category of smooth manifolds to the opposite $\i$-category of simplicial $C^\i$-rings, sending a manifold $M$ to the discrete simplicial $C^{\infty}$-ring of smooth functions on $M,$ is fully faithful and preserves transverse pullbacks and the terminal object. 
\end{thm}

\begin{thm}
There exists a fully faithful $C^\i$-spectrum functor
$$\Speci:\Alg_{\Cart}\left(\Spc\right)^{op} \hookrightarrow \Loc$$ from there $\i$-category of homotopically finitely presented simplicial $C^\i$-rings and the $\i$-category of spaces locally ringed in simplicial $C^\i$-rings.
\end{thm}

In Section \ref{sec:univ2}, we put the theory of simplicial $C^\i$-rings developed in the previous section to use and use it to explore the universal property of derived manifolds. The first main result is the following:

\begin{thm}
For all idempotent complete $\i$-categories $\sC$ with finite limits, composition with $$i\circ q:\Cart \to \DMfd$$ induces an equivalence of $\i$-categories
$$\Fun^{\lex}\left(\DMfd,\sC\right) \stackrel{\sim}{\longlongrightarrow} \Fun^{\pi}\left(\Cart,\sC\right)=\Alg_{\Cart}\left(\sC\right)$$
between the $\i$-category of left exact functors from derived manifolds to $\sC,$ and the $\i$-category of $C^\i$-rings in $\sC.$
\end{thm}

along with the following corollary:

\begin{cor}
There is a canonical equivalence of $\i$-categories $$\DMfd \simeq \left(\Alg_{\Cart}\left(\Spc\right)^{\fp}\right)^{op}.$$
\end{cor}

Along the way we prove the following interesting facts about functors out of the category of smooth manifolds:

\begin{lem}
Let $q:\Cart \hookrightarrow \Mfd$ be the inclusion of the full subcategory of smooth manifolds on those of the form $\bR^n,$ for some $n.$ Let $\sC$ be any idempotent complete $\i$-category with finite limits, and let $$F:\Cart \to \sC$$ be an arbitrary functor. Then the following are equivalent

\begin{itemize}
\item[1.] $F$ preserves finite products
\item[2.] The right Kan extension $\Ran_q F$ exists, and moreover preserves transverse pullbacks and the terminal object.
\end{itemize}
\end{lem}

\begin{cor}
Denote by $\mathbf{Dom} \subset \Mfd$ the full subcategory on the open domains. Let $\sC$ be any idempotent complete $\i$-category with finite limits, and let $$G:\Mfd \to \sC$$ be an arbitrary functor. Then the following are equivalent

\begin{itemize}
\item[1.] $G$ preserves transverse pullbacks and the terminal object
\item[2.] The restriction of $G$ to $\mathbf{Dom}$ preserves transverse pullbacks and the terminal object
\end{itemize}
\end{cor}

Finally, we end the paper by establishing the equivalence between (a minor extension beyond the quasi-smooth setting of) Spivak's model and our $\i$-category $\DMfd$ defined via its universal property:

\begin{thm}
There is a canonical equivalence of $\i$-categories between the homotopy coherent nerve of $\mathbf{dMan}_{\mathit{Spivak}}$ of (a minor extension of) Spivak's model of derived manifolds and $\DMfd.$
\end{thm}

\subsection{Conventions}
All manifolds will be assumed smooth, finite dimensional, $2^{nd}$-countable and Hausdorff. By an $\i$-category we mean a quasicategory or inner-Kan complex, which is one particular model for the theory of $\left(\i,1\right)$-categories. We will follow the notational conventions and terminology of \cite{htt}.\\
Our grading conventions are \emph{cohomological}, so that the differential on a complex \emph{raises} the degree.

\subsection*{Acknowledgments}

We are grateful for discussions with Damien Calaque, Andrew Macpherson and Joost Nuiten. We would also like to thank the Max Planck Institute for Mathematics for their hospitality, and D.C. would like to additionally thank them for support during his sabbatical, during which this research was conducted.

\section{The Universal Property}\label{sec:univ}

Spivak proposes properties that an $\i$-category $\sC$ should have to be ``good for doing intersection theory on manifolds.'' In particular, such  an $\i$-category should be come with additional data, such as a finite limit preserving functor to the category of topological spaces. Recent work of Macpherson \cite{andrew} has proposed a universal property for derived manifolds, in terms of $\i$-categories equipped with extra geometric structure, closely related to the concept of \emph{geometry} developed by Lurie in \cite{dagv}.
Indeed, Macpherson's ideas were highly influential on this paper. However, we propose a simpler universal property in the context of bare $\i$-categories (without need to appeal to the concept of a geometry). We then show that the extra structure, such as a natural underlying space functor, are direct consequences of the universal property. 
  
One of the axioms for an $\i$-category $\sC$ to be ``good for doing intersection theory on manifolds'' is that there should exist a fully faithful functor $$i:\Mfd \hookrightarrow \sC$$ from the category of smooth manifolds which preserves transverse pullbacks and the terminal object. We take the point of view that in fact, the $\i$-category of derived manifolds should be \emph{universal} with respect to this property. More precisely:

\begin{dfn}\label{dfn:DMfd}
Let $\sD$ denote the subcategory of the $\i$-category $\iCat$ of small $\i$-categories consisting of those which have finite limits and are idempotent complete, and left exact functors between them. Denote by $F:\sD \to \iCat$ the functor $$\sC \mapsto \Fun^{\pitchfork}\left(\Mfd,\sC\right),$$ where $\Fun^{\pitchfork}\left(\Mfd,\sC\right)$ is the full subcategory of the functor category $\Fun\left(\Mfd,\sC\right)$ on those functors preserving transverse pullbacks and the terminal object.
Denote by $\cE \to \sD$ the Grothendieck construction of this functor, i.e. the pullback of the universal coCartesian fibration $\cZ \to \iCat$ along $F$:
$$\xymatrix{\cE \ar[d] \ar[r] & \cZ \ar[d]\\ \sD \ar[r]^-{F} & \iCat,}$$
i.e., $\cE$ is the $\i$-category of finitely complete idempotent complete $\i$-categories equipped with a functor from the category of $\Mfd$ which preserves transverse pullbacks and the terminal object, and left exact functors between them respecting the functor from $\Mfd.$

The $\i$-category $\DMfd$ of \textbf{derived manifolds} is the initial object in $\cE.$ Unwinding this, this means that the $\i$-category $\DMfd$ has finite limits, is idempotent complete, and there is a functor $$i:\Mfd \to \DMfd$$ from the category of smooth manifolds to the $\i$-category of derived manifolds which preserves transverse pullbacks and the terminal object. Moreover, for any idempotent complete $\i$-category $\sC$ which has finite limits, composition with $i$ induces an equivalence of $\i$-categories 

$$\Fun^{\lex}\left(\DMfd,\sC\right) \stackrel{\sim}{\longlongrightarrow} \Fun^{\pitchfork}\left(\Mfd,\sC\right)$$
between functors from derived manifolds to $\sC$ which preserve finite limits, to functors from manifolds which preserve transverse pullbacks and the terminal object. 
\end{dfn}
As with any universal property, this pins down the $\i$-category $\DMfd$ up to equivalence. Moreover, such an $\i$-category can be shown to exist by standard and purely formal abstract reasoning. Rather than appealing to such an opaque construction, in this paper, we will investigate this universal property and find many equivalent characterizations of it, and ultimately arise at concrete presentations for this $\i$-category.
\begin{rmk}
As we shall see, the functor $i$ will moreover be fully faithful.
\end{rmk}

\subsection{$C^\infty$-rings}.
All existing models for derived manifolds use the concept of a \emph{$C^\infty$-ring} \cite{spivak,joyce,derivjustden,dg2}. As we shall see later, this is not a coincidence, as the universal property of the $\i$-category $\DMfd$ of derived manifolds is intimately linked with concept of a $C^\infty$-ring. Roughly speaking, a $C^\infty$-ring is a commutative $\mathbb{R}$-algebra $A$ with the extra structure of, for each smooth map $f:\mathbb{R}^n \to \mathbb{R}^m,$ an operation $$A\left(f\right):A^n \to A^m,$$ subject to natural compatibility axioms. The prototypical example of a $C^\i$-ring is the ring of smooth functions $\sC^\i\left(M\right)$ on a smooth manifold $M$, where
\begin{eqnarray*}
\sC^\i\left(M\right)\left(f\right):\sC^\i\left(M\right)^n &\to& \sC^\i\left(M\right)^m\\
\left(\varphi_1\left(x\right),\ldots,\varphi_n\left(x\right)\right) &\mapsto& \left( f_1\left(\varphi_1\left(x\right),\ldots,\varphi_n\left(x\right)\right),\ldots,f_m\left(\varphi_1\left(x\right),\ldots,\varphi_n\left(x\right)\right)\right).
\end{eqnarray*}
Denote by $\Cart$ the full subcategory of the category of smooth manifolds on those of the form $\mathbb{R}^n,$ for some $n.$ Then, the formal definition is as follows:
\begin{dfn}
A \emph{$C^\infty$-ring} is a finite product preserving functor $\mathcal{A}:\Cart \to \Set$ 
\end{dfn}
Since $\mathbb{R}$ is a commutative ring with addition and multiplication being smooth maps, and since $\mathcal{A}$ preserves finite products, $A:=\mathcal{A}\left(\mathbb{R}\right)$ is a commutative ring with extra structure maps as above. They key fact making
 $C^\i$-rings so useful for derived geometry in the smooth setting is that the functor
$$\sC^\i:\Mfd \to \Alg_{\Cart}\left(\Set\right)$$ is fully faithful and preserves transverse pullbacks \cite{MSIA}.

The theory of $C^\i$-rings is \emph{algebraic}. We present some generalities about the theory of algebraic theories in Section \ref{sec:algthy}, in the context of $\i$-categories. In particular, one may consider $\Cart$-algebras in any $\i$-category with finite products. If $\sC$ is the $\i$-category $\Spc$ of spaces, the $\i$-category $\Alg_{\Cart}\left(\Spc\right)$ of algebras in $\Spc$ is equivalent to the $\i$-category underlying the model category of simplicial $\Cart$-algebras (in sets), by \cite[Corollary 5.5.9.2]{htt}.

\subsection{Consequences of the Universal Property}
We note a couple immediate consequences of the universal property.

\begin{ex}
Let $u:\Mfd \to \Top$ be the forgetful functor to topological spaces. Then as $u$ preserves transverse pullbacks and the terminal object, there is a unique left exact functor $U$ making the diagram commute
$$\xymatrix{\Mfd \ar[r]^-{i} \ar[d]_-{u} & \DMfd. \ar[ld]^-{U}\\ \Top &}$$ I.e., $\DMfd$ is \emph{geometric} in the sense of \cite[Definition 1.3]{spivak}. We define $U$ to be the \textbf{underlying space functor}. 
\end{ex}

\begin{rmk}
Spivak proposes more axioms for an $\i$-category to be \emph{good for doing intersection theory on manifolds} (\cite[Definition 2.1]{spivak}--- phrased in terms of simplicial categories), and proves the model he constructs for quasi-smooth derived manifolds satisfies these axioms. Rather than proving directly that the subcategory of quasi-smooth objects of $\DMfd$ satisfies these axioms, we instead show that it is equivalent to Spivak's model (after removing the restriction of being quasi-smooth and adding the condition of being finite type), which is shown to satisfy these axioms in op. cit.
\end{rmk}
\begin{dfn}\label{dfn:JDMfd}
Using the underlying space functor $U,$ one can define a natural Grothendieck topology $J_{\DMfd}$ on $\DMfd$ by declaring a collection of maps $$\left(f_\alpha:U_\alpha \to \cM\right)_\alpha$$ to be a cover of a derived manifold $\cM$ if and only if $$\left(U\left(f_\alpha\right):U\left(U_\alpha\right) \to U\left(\cM\right)\right)_\alpha$$ is an open cover of topological spaces. 
\end{dfn}

\begin{rmk}
We will see later (Proposition \ref{prop:subcan}) that this Grothendieck topology is moreover subcanonical. 
\end{rmk}

\begin{ex}\label{ex:sheaf}
Let $\Cart$ be the full subcategory of the category of smooth manifolds on those of the form $\mathbb{R}^n,$ for some $n.$
Since $i:\Mfd \to \DMfd$ preserves transverse pullbacks and the terminal object, the composite $\Cart \to \DMfd$ preserves finite products. So, for $\cM$ any derived manifold, the functor
$$\Map_{\DMfd}\left(\cM,i\left(\blank\right)\right):\Cart \to \Spc$$ also preserves finite products, and hence defines a $\Cart$-algebra (in spaces) in the sense of Definition \ref{dfn:algebra}. Therefore there is a canonically induced functor
$$\cO_{\DMfd}:\DMfd^{op} \to \Alg_{\Cart}\left(\Spc\right).$$ We will see later that $\cO_{\DMfd}$ is a sheaf for the Grothendieck topology $J_{\DMfd}$ (Proposition \ref{prop:Osheaf}), and even more interestingly, the functor $\cO_{\DMfd}$ is fully faithful, and the essential image is precisely the (homotopically) finitely presented algebras. In particular, $$\DMfd \simeq \left(\Alg_{\Cart}\left(\Spc\right)^{\fp}\right)^{op}.$$ See Corollary \ref{cor:fparedms}.
\end{ex}

\section{Algebraic Theories}\label{sec:algthy}
In this section we outline the theory of \emph{algebraic theories} in the context of $\i$-categories. It will provide the correct language to study $C^\i$-rings in depth in Section \ref{sec:prop}.

\subsection{General Theory}
\begin{dfn}
An \textbf{algebraic theory}, or, \textbf{Lawvere theory}, is an $\infty$-category $\bT$ with finite products, together with a chosen object $r \in \bT_0,$ called the \textbf{generator} such that every object in $\bT$ is equivalent to $r^n,$ for some $n \ge 0.$ A morphism of algebraic theories $\left(\bT,r\right) \to \left(\bT',r'\right)$ is a finite product preserving functor $f: \bT \to \bT'$ carrying $r$ to $r'.$
\end{dfn}

\begin{ex}
The category $\Cart$ is an algebraic theory with $\bR$ as a generator.
\end{ex}

\begin{ex}
Let $k$ be a commutative ring and let $\Comk$ be the full subcategory of affine $k$-schemes on those of the form $\bA_k^n,$ for $n \ge 0.$ Then, $\Comk$ is an algebraic theory with $\bA_k^1$ as a generator.

When $k=\bR,$ $\ComR$ may also be seen as the category of manifolds of the form $\bR^n$ whose morphisms are given by polynomials. There is an evident functor $\ComR \to \Cart$ which is a morphism of algebraic theories.
\end{ex}

\begin{rmk}
The single object $r$ can be replaced with a set $S$ of objects, which leads to the definition of a multi-sorted algebraic theory. The theory of such algebraic theories is completely analogous.  Of particular relevance is the $2$-sorted algebraic theory $\mathbf{SC}^\infty,$ which is the full subcategory of supermanifolds on those of the form $\bR^{p|q}$, which has $\bR$ and $\bR^{0|1}$ as generators (see \cite{dg1}). Using this algebraic theory in place of $\Cart,$ one may extend the results of this paper to the setting of derived supermanifolds. We leave the details to the reader.
\end{rmk}

\begin{dfn}\label{dfn:algebra}
Let $\sC$ be an $\i$-category with finite products. The $\i$-category $\Alg_{\bT}\left(\sC\right)$ of \textbf{$\bT$-algebras in $\sC$} is the full subcategory of $\Fun\left(\bT,\sC\right)$ on those functors which preserve finite products.

Of particular importance is the cases when $\sC=\Set$ and when $\sC=\Spc$. We will call an object of $\Alg_{\bT}\left(\Spc\right)$ a $\bT$-algebra, and an object of $\Alg_{\bT}\left(\Set\right)$ a $0$-truncated $\bT$-algebra.
\end{dfn}

\begin{prop}
Let $\sC$ be a presentable $\i$-category. Then $\Alg_{\bT}\left(\sC\right)$ is also presentable.
\end{prop}

\begin{proof}
This follows immediately from \cite[Lemmas 5.5.4.17, 5.5.4.18, 5.5.4.19]{htt}.
\end{proof}

\begin{ex}
The category of $\Cart$-algebras in $\Set$ is the category of $C^\infty$-rings.
\end{ex}

\begin{ex}
The category of $\Comk$-algebras in $\Set$ is equivalent to the category of commutative $k$-algebras.
\end{ex}

When $\bT$ is a $1$-category, then the projective model structure on $\Fun\left(\bT,\Set^{\Delta^{op}}\right)$ restricts to the subcategory of simplicial $\bT$-algebras
$$\Alg_{\bT}\left(\Set\right)^{\Delta^{op}}\cong \Alg_{\bT}\left(\Set^{\Delta^{op}}\right).$$

We have the following result, originally due to Bergner:

\begin{thm}\cite[Corollary 5.5.9.2]{htt}
There is an equivalence of $\i$-categories $$N_{hc}\left(\Alg_{\bT}\left(\Set\right)^{\Delta^{op}}_{proj.}\right) \simeq \Alg_{\bT}\left(\Spc\right)$$
between the homotopy coherent nerve of the category of simplicial $\bT$-algebras, endowed with the projective model structure, and the $\i$-category of $\bT$-algebras in $\Spc.$
\end{thm}

In light of the above theorem, we refer to the $\i$-category $\Alg_{\Comk}\left(\Spc\right)$ as the $\i$-category of \textbf{simplicial commutative $k$-algebras}, and to $\Alg_{\Cart}\left(\Spc\right)$ as the $\i$-category of \textbf{simplicial $C^\infty$-rings}.

Let $\bT$ be an algebraic theory. Then as any co-representable functor preserves finite products, the Yoneda embedding restricts to a fully faithful $j:\bT \hookrightarrow \Alg_{\bT}\left(\Spc\right)^{op}.$ 
\begin{thm} \cite[Proposition 5.5.8.10]{htt} \label{thm:5.5.8.10}
\begin{itemize}
\item [a)] $\Alg_{\bT}\left(\Spc\right)$ is a localization of $\Fun\left(\bT,\Spc\right).$
 \item[b)] The $\i$-category $\Alg_{\bT}\left(\Spc\right)$ is compactly generated. In particular, $$\Ind\left(\Alg_{\bT}\left(\Spc\right)^{\fp}\right)\simeq \Alg_{\bT}\left(\Spc\right).$$
 \item[c)] The inclusion $j:\bT \hookrightarrow \Alg_{\bT}\left(\Spc\right)^{op}$ preserves finite products.
\end{itemize}
\end{thm}

The functor $j$ also has the following universal property:

\begin{thm} \cite[Proposition 5.5.8.15]{htt} \label{thm:5.5.8.15}
Let $\sC$ be an $\i$-category which admits sifted limits. Then, composition with $j$ induces an equivalence of $\i$-categories
$$\Fun^{sift}\left(\Alg_{\bT}\left(\Spc\right)^{op},\sC\right) \stackrel{\sim}{\longlongrightarrow} \Fun\left(\bT,\sC\right),$$
between the $\i$-category of functors $\Alg_{\bT}\left(\Spc\right)^{op} \to \sC$ which preserves sifted limits, to the $\i$-category $\Fun\left(\bT,\sC\right),$ and if $\sC$ is complete, it also induces an equivalence of $\i$-categories
$$\Fun^{R}\left(\Alg_{\bT}\left(\Spc\right)^{op},\sC\right) \stackrel{\sim}{\longlongrightarrow} \Fun^{\Pi\!\!}\left(\bT,\sC\right)=\Alg_{\bT}\left(\sC\right)$$
between the $\i$-category of functors $\Alg_{\bT}\left(\Spc\right)^{op} \to \sC$ which preserves small limits, to the $\i$-category of $\bT$-algebras in $\sC.$ The inverse of both equivalences is given by right Kan extension.
\end{thm}

In view of the above theorem, we call $\Alg_{\bT}\left(\Spc\right)^{op}$ the \textbf{small limit envelope} of the algebraic theory $\bT.$

\begin{rmk}
Note that by Theorem \ref{thm:5.5.8.10}, $c),$ $j$ is a $\bT$-algebra in $\Alg_{\bT}\left(\Spc\right)^{op}.$ Unwinding the definitions, this is the $\bT$-algebra corresponding to the identity functor under the equivalence
$$\Fun^{sift}\left(\Alg_{\bT}\left(\Spc\right)^{op},\Alg_{\bT}\left(\Spc\right)^{op}\right) \stackrel{\sim}{\longlongrightarrow} \Fun\left(\bT,\Alg_{\bT}\left(\Spc\right)^{op}\right).$$
\end{rmk}

Before going further, we introduce a categorical lemma which will be of crucial importance:

\begin{lem}\label{lem:crucial}
Let $\sA \stackrel{j}{\to} \sB \stackrel{\omega}{\longhookrightarrow} \sC$ be functors of $\i$-categories with $\omega$ fully faithful and $\sB$ small, and let $F:\sA \to \sD$ be an arbitrary functor with $\sD$ cocomplete. Then if the right Kan extension $\Ran_{\left(\omega \circ j\right)}F$ exists, so does $\Ran_j F$ and moreover
$$\Ran_j F \simeq \Ran_{\left(\omega \circ j\right)}F \circ \omega.$$
\end{lem}

\begin{proof}
The smallness of $\sB$ and the cocompleteness of $\sD$ guarantees the existence of a global left Kan extension functor $$\Lan_{\omega}:\Fun\left(\sB,\sD\right) \to \Fun\left(\sC,\sD\right)$$ which is left adjoint to the restriction functor $\omega^*,$ and the fully faithfulness of $\omega$ is equivalent to the assertion that the unit $$\eta:id \Rightarrow \omega^*\Lan_{\omega}$$ is an equivalence (which in turn is equivalent to $\Lan_{\omega}$ being fully faithful).

To establish the claim, we need to show that  $\left(\Ran_{\left(\omega \circ j\right)}F\right) \circ \omega$ satisfies the universal property for the right Kan extension $\Ran_j F,$ namely, for all $\psi:\sB \to \sD,$ 
$$\Map_{\Fun\left(\sB,\sD\right)}\left(\psi,\left(\Ran_{\left(\omega \circ j\right)}F \circ \omega\right)\right) \simeq \Map_{\Fun\left(\sA,\sD\right)}\left(j^*\psi,F\right).$$
Indeed, we have, by universal properties, the following string of natural equivalences:
\begin{eqnarray*}
\Map_{\Fun\left(\sB,\sD\right)}\left(\psi,\left(\Ran_{\left(\omega \circ j\right)}F \circ \omega\right)\right) &=& \Map_{\Fun\left(\sB,\sD\right)}\left(\psi,\omega^*\Ran_{\left(\omega \circ j\right)}F\right)\\
&\simeq& \Map_{\Fun\left(\sC,\sD\right)}\left(\Lan_\omega\psi,\Ran_{\left(\omega \circ j\right)}F \right)\\
&\simeq& \Map_{\Fun\left(\sA,\sD\right)}\left(\left(\omega \circ j\right)^*\Lan_\omega\psi,F \right)\\
&\simeq& \Map_{\Fun\left(\sA,\sD\right)}\left(j^*\left(\omega ^*\Lan_\omega\psi\right),F \right)\\
&\simeq& \Map_{\Fun\left(\sA,\sD\right)}\left(j^*\psi,F \right),
\end{eqnarray*}
the last equivalence coming from the unit $\eta$ being an equivalence.
\end{proof}
\begin{dfn}\label{defn:effectiveepi}
Let $\bT$ be an algebraic theory. We say that a morphism $f:A\rightarrow B$ is an \textbf{effective epimorphism} if $f$ is an effective epimorphism in the $\infty$-topos $\Psh\left(\bT^{op}\right)$, i.e. the induced map $\colim \text{\v{C}}\left(f\right)\rightarrow A$ from the colimit of the \v{C}ech nerve of $f$ to $A$ is an equivalence. 
\end{dfn}
Let $\bT$ be an algebraic theory. As limits and sifted colimits are computed in the $\infty$-topos $\Psh(\bT^{op})$ and geometric realizations are sifted colimits, we see that the $\infty$-category $\Alg_{\bT}\left(\Spc\right)$ inherits the following properties from $\Psh(\bT^{op})$.
\begin{prop}
Let $\bT$ be a algebraic theory.
\begin{enumerate}
    \item Sifted colimits are universal in $\Alg_{\bT}\left(\Spc\right)$.
    \item Every groupoid object in $\Alg_{\bT}\left(\Spc\right)$ is effective.
    \item $\Alg_{\bT}\left(\Spc\right)$ has an \emph{epi-mono factorization system}: There exists a factorization system $(S_L,S_R)$ on $\Alg_{\bT}\left(\Spc\right)$ such that $S_L$ consists of effective epimorphisms and $S_R$ consists of monomorphisms.
\end{enumerate}
\end{prop}
\begin{rmk}
It is observed in \cite[Remark 5.5.8.26]{htt} that the $n^{th}$ truncation $\tau_{\leq n}\Alg_{\bT}\left(\Spc\right)$ is precisely the full subcategory of functors $\bT\rightarrow \Spc$ taking values in $n$-truncated objects. In particular, if $\bT$ is an ordinary 1-category, the $1$-category $\tau_{\leq 0}\Alg_{\bT}\left(\Spc\right)$ can be identified with $\Alg_{\bT}\left(\Set\right)$, the category of ordinary $\bT$-algebras in sets and we have a fully faithful inclusion $\Alg_{\bT}\left(\Set\right)\hookrightarrow \Alg_{\bT}\left(\Spc\right)$.
\end{rmk}
\begin{rmk}
We may associate to any $\bT$-algebra $A$ a collection of homotopy sets as follows: let $r$ be the generator of $\bT$, then there is a functor $\theta_r:\Alg_{\bT}\left(\Spc\right)\rightarrow \Spc$ given by evaluating at $r$. It is customary to identify a $\bT$-algebra with the space $\theta_r\left(A\right) $. For each $n\geq 0$ we denote by $\pi_n\left(A\right)$ the $n^{th}$ homotopy set of $\theta_r\left(A\right)$ which is an abelian group for $n\geq 1$. By the previous remark, the homotopy set $\pi_0\left(A\right)$ can be identified with $\tau_{\leq 0}A\left(r\right)$ and if $\bT$ is a 1-category, the set $\pi_0\left(A\right)$ carries the structure of an ordinary $\bT$-algebra; we will use both notations in the sequel. 
\end{rmk}
\begin{rmk}\label{effepilawvere}
From the generating properties of the object $r\in \bT,$ we deduce immediately that the functor 
\[\theta_r:\Alg_{\bT}\left(\Spc\right){\longrightarrow} \Spc \]
is conservative. Combining this observation with the fact that $\theta_r$ preserves limits and geometric realizations of simplicial objects and \cite[Corollary 7.1.2.15]{htt}, we see that a morphism $A\rightarrow B$ of $\bT$-algebras is an effective epimorphism if and only if the induced map $\pi_0(A)\rightarrow \pi_0(B)$ on sets is surjective.
\end{rmk}

\begin{dfn}
\begin{enumerate}
     \item A $\bT$-algebra $X \in \Alg_{\bT}\left(\Spc\right)$ is \textbf{finitely generated} if the functor  $\Alg_{\bT}\left(\Spc\right)\rightarrow \Spc$ corepresented by $X$ preserves small filtered colimits consisting only of monomorphisms. Denote the full subcategory on the finitely generated algebras by $\Alg_{\bT}\left(\Spc\right)^{\mathbf{fg}}$.
    \item A $\bT$-algebra $X \in \Alg_{\bT}\left(\Spc\right)$ is \textbf{finitely presented} if the functor  $\Alg_{\bT}\left(\Spc\right)^{\fp}\rightarrow \Spc$ corepresented by $X$ preserves small filtered colimits; that is, if $X$ is a compact object. Denote the full subcategory on the finitely presented algebras by $\Alg_{\bT}\left(\Spc\right)^{\fp}$.
\end{enumerate}
\end{dfn}
 Any free $\bT$-algebra, i.e. one in the essential image of $j,$ is finitely presented. This follows by the Yoneda lemma and Theorem \ref{thm:5.5.8.10}. It follows, that there is a fully faithful functor $$j^{\fp}:\bT \hookrightarrow \left(\Alg_{\bT}\left(\Spc\right)^{\fp}\right)^{op}.$$ In the classical theory of $0$-truncated algebraic theories, a $\bT$-algebra $A$ is finitely generated if and only if there exists some free algebra $F$ on finitely many generators and a quotient map $F\rightarrow A$. The same principle holds for $\bT$-algebras, with the caveat that an effective equivalence relation must be replaced by an effective groupoid.   
\begin{prop}\label{prop:fgeffectiveepi}
Let $A$ be a $\bT$-algebra. The following are equivalent. 
\begin{enumerate}
    \item $A$ is finitely generated.
    \item There exists an object $t\in \bT$ and an effective epimorphism $q:j(t)\rightarrow A$.
\end{enumerate}
\end{prop}
\begin{proof}
We start by proving that $\left(1\right)\Rightarrow \left(2\right)$. Let $A$ be finitely generated, and let $\mathrm{Sub}\left(A\right)$ be the filtered poset of subobjects of $A$. Let $\mathrm{Sub}'\left(A\right)$ be the subposet of $\mathrm{Sub}\left(A\right)$ of subobjects of $A$ that satisfy condition $\left(2\right)$, which is nonempty (because every map $j\left(t\right)\rightarrow A$ factors as an epimorphism followed by a monomorphism) and is easily seen to be filtered. We claim that $A$ is the colimit of the diagram $\mathcal{J}:\mathrm{Sub}'\left(A\right)\rightarrow \Alg_{\bT}\left(\Spc\right)$: Since the $\infty$-categories $\tau_{\leq k}\Alg_{\bT}\left(\Spc\right)_{/_A}$ are stable under filtered colimits for $k\geq -2$, the map $$\underset{{A_i\in\mathrm{Sub}'\left(A\right)}}\colim A_i\rightarrow A$$ is a monomorphism, so the map of spaces $$\underset{{A_i\in\mathrm{Sub}'\left(A\right)}}\colim A_i\left(r\right)\rightarrow A\left(r\right)$$ is an inclusion of connected components (here we use that the evaluation functors preserve filtered colimits). The evaluation functor $$\theta_r:\Alg_{\bT}\left(\Spc\right)\rightarrow \Spc$$ of Remark \ref{effepilawvere} is conservative, so it suffices to show that the morphism $$\underset{{A_i\in\mathrm{Sub}'\left(A\right)}}\colim A_i\left(r\right)\rightarrow A\left(r\right)$$ is an equivalence. To see this, we only have to check that this morphism induces a surjection on connected components, meaning that every morphism $j\left(r\right)\rightarrow A$ factors through some $B\in \mathrm{Sub}'\left(A\right)$. This is the case as $j\left(r\right)\rightarrow A$ factors as an effective epimorphism followed by a monomorphism. Because $A$ is finitely generated, we have $$\Map_{\Alg_{\bT}\left(\Spc\right)}\left(A,A\right)\simeq \underset{{A_i\in \mathrm{Sub}'\left(A\right)}}\colim\left(A,A_i\right),$$ so the identity map $A\rightarrow A$ factors as $$A\overset{f}{\rightarrow} A_i\rightarrow A$$ for some $A_i\in \mathrm{Sub}'\left(A\right)$. The map $A_i\rightarrow A$ is a monomorphism for any $A_i$, so the map $f:A\overset{\simeq}{\longrightarrow}A_i$ is an equivalence.

Now we show that $\left(2\right)\Rightarrow \left(1\right).$ Let $$Y=\underset{i\in I}\colim Y_i$$ be a colimit of a filtered diagram consisting only of monomorphisms. A map $A\rightarrow Y$ induces a map $j\left(t\right)\rightarrow Y$ which must factor through one of the $Y_i$'s as $j\left(t\right)$ is a compact projective object in $\Alg_{\bT}\left(\Spc\right)$. Because $Y_i\rightarrow Y$ is a monomorphism and the class of effective epimorphisms is left orthogonal to the class of monomorphisms, we can find a dotted arrow that makes the diagram
\begin{equation*}
\begin{tikzcd}
j(t)\ar[d]\ar[r] & Y_i\ar[d]\\
A\ar[r]\ar[ur,dotted] & Y
\end{tikzcd}    
\end{equation*}
commute, which proves that $A$ is finitely generated. 
\end{proof}
Now we characterize the class of finitely presented $\bT$-algebras.

\begin{lem}\label{lem:retcol}
The subcategory $\Alg_{\bT}\left(\Spc\right)^{\fp}$ is the smallest subcategory of $\Alg_{\bT}\left(\Spc\right)$ containing the finitely generated free algebras and closed under finite colimits and retracts. Moreover, any finitely presented algebra is a retract of a finite colimit of finitely generated free algebras.
\end{lem}

\begin{proof}
Let $\sC$ be the smallest full subcategory of $\Alg_{\bT}\left(\Spc\right)$ that contains the free algebras and is stable under finite colimits and retracts. Since $\Alg_{\bT}\left(\Spc\right)^{\fp}$ is stable under finite colimits and retracts and contains the free algebras, we have $\sC$ is contained in $\Alg_{\bT}\left(\Spc\right)^{\fp}$. To establish the other inclusion, we show that every finitely presented $\bT$-algebra is a retract of a finite colimit of free algebras. By Theorem \ref{thm:5.5.8.10} a), it follows that any $\bT$-algebra is a small colimit of free $\bT$-algebras. By decomposing the index simplicial set $K$ of a small colimit into the partially ordered set of finite subsimplices of $K$, we may write the colimit of $K$ as a filtered colimit of
finite colimits \cite[Corollary 4.2.3.11]{htt}. Applying this to a finitely presented simplicial $\bT$-algebra $A$, we have a filtered colimit $$A = \underset{\alpha} \colim A _\alpha,$$ where each $A_\alpha$ is a finite colimit of free $\bT$-algebras. Because $A$ is finitely presented, the identity map $id:A \to A$ factors trough some $A_\alpha \to A,$ which shows that the desired retraction exists.
\end{proof}

\begin{thm}\label{thm:finenv}
Let $\sC$ be any idempotent complete $\i$-category with finite limits. Then composition with $j^{\fp}$ induces an equivalence of $\i$-categories

$$\Fun^{\lex}\left(\left(\Alg_{\bT}\left(\Spc\right)^{\fp}\right)^{op},\sC\right) \stackrel{\sim}{\longlongrightarrow} \Fun^{\Pi\!\!}\left(\bT,\sC\right)=\Alg_{\bT}\left(\sC\right)$$
whose inverse is given by right Kan extension.
\end{thm}
\begin{proof}
The proof starts similarly to that of \cite[Remark 3.4.6]{dagv}. Consider the fully faithful inclusions
$$\bT \stackrel{j^{\fp}}{\longhookrightarrow} \left(\Alg_{\bT}\left(\Spc\right)^{\fp}\right)^{op} \stackrel{\omega}{\longhookrightarrow} \left(\Alg_{\bT}\left(\Spc\right)\right)^{op}.$$
By Theorem \ref{thm:5.5.8.10} c), $j^{\fp}$ preserves finite products, and furthermore, $\omega$ preserves finite limits, so we have a commutative diagram of restriction functors
$$\xymatrix@R=2cm{& \Fun^{\lex}\left(\left(\Alg_{\bT}\left(\Spc\right)^{\fp}\right)^{op},\sC\right) \ar[dl]_-{\omega^*} \ar[dr]^-{\left(j^{\fp}\right)^*} &\\ \Fun^{R}\left(\Alg_{\bT}\left(\Spc\right)^{op},\sC\right) \ar[rr]^-{\left(\omega \circ j^{\fp}\right)^*=j^*} & & \Fun^{\Pi\!\!}\left(\bT,\sC\right).}$$
We wish to show that $\left(j^{\fp}\right)^*$ is an equivalence. First let us establish the claim in the case that $\sC$ has all small limits and colimits. Then $\left(\omega \circ j^{\fp}\right)^*=j^*$ is an equivalence by Theorem \ref{thm:5.5.8.15}. Also, by Theorem \ref{thm:5.5.8.10}, we have $$\Ind\left(\Alg_{\bT}\left(\Spc\right)^{\fp}\right)\simeq \Alg_{\bT}\left(\Spc\right),$$ and hence, it follows from  \cite[Propositions 5.3.5.10 and 5.5.1.9]{htt}, that $\omega^*$ is an equivalence, and hence $\left(j^{\fp}\right)^*$ is as well. We claim the inverse to $\left(j^{\fp}\right)^*$ is given by $\Ran_{j^{\fp}}.$ Since global right Kan extension $$\Ran_{j^{\fp}}:\Fun\left(\bT,\sC\right) \to \Fun\left(\left(\Alg_{\bT}\left(\Spc\right)^{\fp}\right)^{op},\sC\right)$$ is right adjoint to $\left(j^{\fp}\right)^*,$ and $\left(j^{\fp}\right)^*$ restricts to an equivalence as in the above diagram, it suffices to show that if $$F:\bT \to \sC$$ preserves finite products, $\Ran_{j^{\fp}} F$ is left exact. By Lemma \ref{lem:crucial}, $$\Ran_{j^{\fp}} F \simeq \left(\Ran_{\omega \circ j^{\fp}} F\right) \circ \omega$$ and by Theorem \ref{thm:5.5.8.15}, $\Ran_{\omega \circ j^{\fp}} F=\Ran_j F$ preserves small limits. Since $\omega$ preserves finite limits, we are done.

Now suppose that $\sC$ is idempotent complete and has finite limits, but not necessarily small limits or colimits. We claim that if $F:\bT \to \sC$ preserves finite products, then $\Ran_{j^{\fp}} F$ exists and is left exact. Indeed, denote by $y:\sC \hookrightarrow \Psh\left(\sC\right)$ the Yoneda embedding, and consider $y \circ F.$ Then $\Ran_{j^{\fp}}\left(y\circ F\right)$ and can be computed by the standard pointwise formula. Moreover, since $y\circ F$ preserves finite products, by the above paragraph, $\Ran_{j^{\fp}}\left(y\circ F\right)$ is left exact. Moreover, since $j^{\fp}$ is fully faithful, for any $t \in \bT,$ $$\Ran_{j^{\fp}}\left(y\circ F\right)\left(j^{\fp}\left(t\right)\right) \simeq F\left(t\right).$$ Now, since any $X \in \left(\Alg_{\bT}\left(\Spc\right)^{\fp}\right)^{op}$ is a retract of a finite limit of objects in the essential image of $j^{\fp}$ by Lemma \ref{lem:retcol}, and $\sC$ is stable under finite limits and retracts in $\Psh\left(\sC\right)$ (the latter since $\sC$ has retracts), it follows that $\Ran_{j^{\fp}}\left(y\circ F\right)\left(X\right)$ is representable for all $X.$ Since $y$ preserves all limits, we conclude that the pointwise right Kan extension $\Ran_{j^{\fp}} F$ exists $$y \circ \Ran_{j^{\fp}} F \simeq \Ran_{j^{\fp}}\left(y \circ F\right),$$ and $\Ran_{j^{\fp}} F$ is left exact. By the universal property of Kan extensions, it follows that the functor $$\Ran_{j^{\fp}}:\Fun\left(\bT,\sC\right) \to \Fun\left(\left(\Alg_{\bT}\left(\Spc\right)^{\fp}\right)^{op},\sC\right)$$ exists and is right adjoint to $\left(j^{\fp}\right)^*.$ We claim that the unit $\eta:id \Rightarrow \Ran_{j^{\fp}} \left(j^{\fp}\right)^*$ is an equivalence. Since $y$ is conservative, it suffices to check that for all $$H:\left(\Alg_{\bT}\left(\Spc\right)^{\fp}\right)^{op} \to \sC$$
the induced morphism $y \circ H \to y \circ \Ran_{j^{\fp}} \left(j^{\fp}\right)^*H$ is an equivalence. But we have that $$y \circ \Ran_{j^{\fp}} \left(j^{\fp}\right)^*H \simeq \Ran_{j^{\fp}} \left(y \circ \left(j^{\fp}\right)^*H\right)\simeq \Ran_{j^{\fp}} \left(j^{\fp}\right)^*\left(y \circ H\right),$$ for all $F:\bT \to \sC.$ Now, since $\Psh\left(\sC\right)$ is complete and cocomplete, we conclude that for all $$G:\left(\Alg_{\bT}\left(\Spc\right)^{\fp}\right)^{op} \to \Psh\left(\sC\right),$$ $G \to \Ran_{j^{\fp}} \left(j^{\fp}\right)^*G$ is an equivalence, by the above paragraph. Applying this to $G=y \circ H,$ for $$H:\left(\Alg_{\bT}\left(\Spc\right)^{\fp}\right)^{op} \to \sC$$ then finishes the proof.
\end{proof}

In view of the above theorem, we call $\Alg_{\bT}\left(\Spc\right)^{op}$ the \textbf{finite limit envelope} of the algebraic theory $\bT.$

\begin{rmk}
Unwinding the definitions, we see that the identity functor corresponds to $j^{\fp}$ under the equivalence
$$\Fun^{\lex}\left(\left(\Alg_{\bT}\left(\Spc\right)^{\fp}\right)^{op},\left(\Alg_{\bT}\left(\Spc\right)^{\fp}\right)^{op}\right) \stackrel{\sim}{\longlongrightarrow} \Alg_{\bT}\left(\left(\Alg_{\bT}\left(\Spc\right)^{\fp}\right)^{op}\right).$$
\end{rmk}

\begin{prop}
Let $\cE$ be an $\i$-topos. Then there is a canonical equivalence of $\i$-categories
$$\Alg_{\bT}\left(\cE\right) \simeq \mathbf{Geom}\left(\cE,\Psh\left(\left(\Alg_{\bT}\left(\Spc\right)^{\fp}\right)^{op}\right)\right),$$ between the $\i$-category of $\bT$-algebras in $\cE$ and the $\i$-category of geometric morphisms from $\cE$ to $\Psh\left(\left(\Alg_{\bT}\left(\Spc\right)^{\fp}\right)^{op}\right).$
\end{prop}

\begin{proof}
The $\i$-category $\mathbf{Geom}\left(\cE,\Psh\left(\left(\Alg_{\bT}\left(\Spc\right)^{\fp}\right)^{op}\right)\right)$ is equivalent to the full subcategory of $\Fun\left(\Psh\left(\left(\Alg_{\bT}\left(\Spc\right)^{\fp}\right)^{op}\right),\cE\right)$ on those functors which preserve small colimits and finite limits. By \cite[Proposition 6.1.5.2]{htt}, it follows that this $\i$-category is in turn equivalent to
$$\Fun^{\lex}\left(\left(\Alg_{\bT}\left(\Spc\right)^{\fp}\right)^{op},\cE\right),$$ so we are done by Theorem \ref{thm:finenv}.
\end{proof}

\begin{rmk}\label{rmk:t}
In view of the above proposition, we call the $\i$-topos $\Psh\left(\left(\Alg_{\bT}\left(\Spc\right)^{\fp}\right)^{op}\right)$ the \textbf{classifying $\i$-topos of $\bT$}, and denote by $\cB\bT.$
\end{rmk}

\begin{rmk}\label{rmk:ncatv}
Fix an integer $n \ge 0$ and suppose $\bT$ is an $\left(n+1\right)$-categorical algebraic theory. Let $\tau_{\le n}\Spc \simeq \mathbf{Gpd}_n$ be the $\left(n+1\right)$-category of $n$-truncated spaces / $n$-groupoids. Then $j^{\fp}$ factors through the canonical inclusion $$\Alg_{\bT}\left(\tau_{\le n}\Spc\right)^{op} \hookrightarrow \Alg_{\bT}\left(\Spc\right)^{op}.$$ Moreover, as any finitely generated free algebra is finitely presented, by abuse of notation we have $$j^{\fp}:\bT \hookrightarrow \left(\Alg_{\bT}\left(\tau_{\le n}\Spc\right)^{\fp}\right)^{op}.$$ Let $\sC$ be an $\left(n+1\right)$-category with finite limits. (Recall that any $\left(n+1\right)$-category with finite limits is automatically idempotent complete.) Then by a completely analogous proof to that of Theorem \ref{thm:finenv}, composition with $j^{\fp}$ induces an equivalence of $\left(n+1\right)$-categories
$$\Fun^{\lex}\left(\left(\Alg_{\bT}\left(\tau_{\le n}\Spc\right)^{\fp}\right)^{op},\sC\right) \stackrel{\sim}{\longlongrightarrow} \Fun^{\Pi\!\!}\left(\bT,\sC\right)=\Alg_{\bT}\left(\sC\right)$$
whose inverse is given by right Kan extension.
\end{rmk}

Fix an integer $n \ge 0.$ Consider the functor $\tau_{\le n}:\Spc \to \tau_{\le n}\Spc,$ which is left adjoint to the canonical inclusion. Since $\tau_{\le n}$ preserves finite products by \cite[Lemma 6.5.1.2]{htt}, it induces a left adjoint to the canonical inclusion
$$\Alg_{\bT}\left(\tau_{\le n}\Spc\right) \hookrightarrow \Alg_{\bT}\left(\Spc\right),$$ sending an algebra $A$ to $\tau_{\le n}\left(A\right).$ Note that in the special case that $n=0,$ this is the functor
$$\pi_0:\Alg_{\bT}\left(\Spc\right) \to \Alg_{\bT}\left(\Set\right).$$

\begin{prop}
If $\bT$ is an $\left(n+1\right)$-categorical algebraic theory, the functor $$\tau_{\le n}:\Alg_{\bT}\left(\Spc\right) \to \Alg_{\bT}\left(\tau_{\le n}\Spc\right)$$ restricts to a finite colimit preserving functor $$\tau_{\le n}:\Alg_{\bT}\left(\Spc\right)^{\fp} \to \Alg_{\bT}\left(\tau_{\le n} \Spc\right)^{\fp}.$$
\end{prop}

\begin{proof}
Let $A$ be in $\Alg_{\bT}\left(\Spc\right)^{\fp}.$ Since $A$ is finitely presented, it is a retract of a finite colimit of finitely generated free algebras. Moreover, since $\bT$ is a $\left(n+1\right)$-category, every finitely generated free $\bT$-algebra is in $\Alg_{\bT}\left(\tau_{\le n}\Spc\right).$ Since $\tau_{\le n}$ preserves colimits, this implies that $A$ is a retract of a finite colimit of finitely generated free algebras in $\Alg_{\bT}\left(\tau_{\le n}\Spc\right),$ and hence finitely presented. Moreover, as the category of finitely generated algebras is closed under finite colimits in all algebras, since $\tau_{\le n}$ preserves small colimits, its restriction preserves finite ones.
\end{proof}

\begin{warning}
A finitely presented $\bT$-algebra in $\tau_{\le n}\Spc$ need not be finitely presented when regarded as an object in $\Alg_{\bT}\left(\Spc\right).$ For this to be the case, it must be \emph{homotopically finitely presented}, which is a stronger condition. For example, if $n=1$ and $A$ is a finitely presented $k$-algebra (in $\Set$), for $k$ a field, if $A$ admits a singularity which is not a local complete intersection, then $A$ is not homotopically finitely presented \cite{Toenshv}.
\end{warning}

\begin{prop}\label{prop:nsame}
Let $\bT$ be an $\left(n+1\right)$-categorical algebraic theory and $\sC$ an $\left(n+1\right)$-category with finite limits. Then the induced functor $$\left(\tau_{\le n}\right)^*:\Fun^{\lex}\left(\left(\Alg_{\bT}\left(\Spc\right)^{\fp}\right)^{op},\sC\right) \to \Fun^{\lex}\left(\left(\Alg_{\bT}\left(\tau_{\le n}\Spc\right)^{\fp}\right)^{op},\sC\right)$$ is an equivalence of $\left(n+1\right)$-categories.
\end{prop}

\begin{proof}
Given any finite product preserving functor $F:\bT \to \sC,$ consider the induced left exact functors
$$\overline{F}_n:\left(\Alg_{\bT}\left(\tau_{\le n} \Spc\right)^{\fp}\right)^{op} \to \sC$$
and
$$\overline{F}:\left(\Alg_{\bT}\left(\Spc\right)^{\fp}\right)^{op} \to \sC.$$ By uniqueness, since $\tau_{\le n}$ preserves finite limits (since we have taken the opposite categories), $$\overline{F} \simeq \overline{F}_n \circ \tau_{\le n}.$$ The result now follows from Theorem \ref{thm:finenv} and Remark \ref{rmk:ncatv}
\end{proof}

\begin{rmk}\label{rmk:nsame}
Since $$\tau_{\le n}:\Alg_{\bT}\left(\Spc\right)^{op} \to \Alg_{\bT}\left(\tau_{\le n}\Spc\right)^{op}$$ preserves small limits, by an analogous proof to above, it follows using Theorem \ref{thm:5.5.8.15} that if $\bT$ is a $\left(n+1\right)$-categorical algebraic theory and $\sC$ an $\left(n+1\right)$-category with small limits, then the induced functor $$\left(\tau_{\le n}\right)^*:\Fun^{R}\left(\left(\Alg_{\bT}\left(\Spc\right)^{\fp}\right)^{op},\sC\right) \to \Fun^{R}\left(\left(\Alg_{\bT}\left(\tau_{\le n}\Spc\right)^{\fp}\right)^{op},\sC\right)$$ between functors preserving small limits, is an equivalence of $\left(n+1\right)$-categories.
\end{rmk}

\subsection{Unramified Transformations of Algebraic Theories}
Given a morphism of algebraic theories $f:\bT\rightarrow \bT'$, we have an induced functor $f^*:\Alg_{\bT'}\left(\Spc\right)\rightarrow \Alg_{\bT}\left(\Spc\right)$ which preserves small limits and small sifted colimits. A natural question that now arises is the following:
\begin{itemize}
    \item \emph{Under what conditions on the morphism $f:\bT\rightarrow \bT'$ does the induced functor $f^*$ preserve (a certain class of) pushouts?}
\end{itemize}
To answer this question, we propose the following definition:
\begin{dfn}\label{dfn:unramifiedtransformation}
Let $\bT$ and $\bT'$ be algebraic theories. A morphism of algebraic theories $f:\bT\rightarrow \bT'$ is \textbf{unramified} if for each morphism $g:X\rightarrow Y$ in $\bT'$ and each $Z\in \bT'$, the diagram
\[
\begin{tikzcd}
f^*j\left(X\times Y\right) \ar[d]\ar[r]& f^*j\left(X\times Y \times Z\right) \ar[d]\\
f^*j\left(X\right) \ar[r] & f^*j\left(X\times Z\right)
\end{tikzcd}
\]
is a pushout in $\Alg_{\bT}\left(\Spc\right)$.
\end{dfn}
The significance of this definition is explained by the following theorem.
\begin{thm}\label{unramifiedpreserveseffepi}
Let $f:\bT\rightarrow \bT'$ be an unramified transformation of algebraic theories. Then $f^*:\Alg_{\bT'}\left(\Spc\right)\rightarrow \Alg_{\bT}\left(\Spc\right)$ preserves pushouts along effective epimorphisms. 
\end{thm}

The proof of this theorem requires some preparation: the argument depends on the existence of a convenient resolution of a pushout diagram where one of the maps is an effective epimorphism as diagrams of the type appearing in Definition \ref{dfn:unramifiedtransformation}. Such resolutions are constructed \cite[Sections 2-3]{dagix} for Lurie's \emph{pregeometries} 
(these are algebraic theories with extra structure specifying a class of `geometrically correct' pullbacks), but we have no need of that generality, so we construct them only for algebraic theories.
\begin{dfn}\label{dfn:elementarydiagram}
A diagram $\tau:\Lambda^2_0\rightarrow \Alg_{\bT}\left(\Spc\right)$ in the $\infty$-category of $\bT$-algebras is \textbf{elementary} if it is equivalent to a diagram of the form 
\[
\begin{tikzcd}
j\left(X\times Y\right) \ar[d]\ar[r]& j\left(X\times Y \times Z\right) \\
j\left(X\right)
\end{tikzcd}
\]
for some morphism $X\rightarrow Y$ in $\bT$ and some $Z\in \bT$.
\end{dfn}
\begin{prop}\label{prop:resolveelementary}
Let $f:\bT\rightarrow \bT'$ be a transformation of algebraic theories, and let 
\[
\begin{tikzcd}
A\ar[d,"\alpha"]\ar[r,"\beta"]& C\\
B
\end{tikzcd}
\]
be a diagram $\sigma:\Lambda^2_0\rightarrow \Alg_{\bT}\left(\Spc\right)$. Suppose there exists a diagram $\overline{\sigma}:\mathcal{C}^{\rhd}\times \Lambda^{2}_0\rightarrow \Alg_{\bT}\left(\Spc\right)$ satisfying the following properties:
\begin{enumerate}
    \item For each vertex $v$ of $\Lambda^2_0$, the diagram $\mathcal{C}^{\rhd}\rightarrow \mathcal{C}^{\rhd}\times \{v\}\rightarrow  \Alg_{\bT}\left(\Spc\right)$ is a colimit diagram. 
    \item For each object $c\in \mathcal{C}$, $f^*$ preserves colimits of the diagram $$\Lambda^{2}_0\rightarrow \{c\}\times \Lambda^{2}_0\rightarrow \Alg_{\bT}\left(\Spc\right).$$
    \item On the cocone point, the diagram $\Lambda^{2}_0\rightarrow \{-\infty\}\times \Lambda^{2}_0\rightarrow \Alg_{\bT}\left(\Spc\right)$ is equivalent to the diagram $\sigma$.
    \item The $\infty$-category $\mathcal{C}$ is sifted. 
\end{enumerate}
Then $f^*$ preserves colimits of the diagram $\sigma$.
\end{prop}
\begin{proof}
We are given a diagram $\sigma:\Lambda^2_0\rightarrow \Alg_{\bT'}\left(\Spc\right)$ 
\[
\begin{tikzcd}
A\ar[d,"\alpha"]\ar[r,"\beta"]& C\\
B
\end{tikzcd}
\]
and a diagram $\overline{\sigma}:\mathcal{C}^{\rhd} \times \Lambda^2_0\rightarrow \Alg_{\bT'}\left(\Spc\right)$ satisfying conditions $\left(1\right)$ through $\left(4\right)$ in the statement of that proposition. Let $-\infty$ be the cocone vertex of $\left(\Lambda^2_0\right)^{\rhd}$, then by properties $\left(1\right)$ and $\left(3\right)$ we have $$\colim\, \sigma\simeq \underset{{D\in \mathcal{C}}}\colim\, \overline{\sigma}|_{\mathcal{C}\times \{-\infty\}}\left(D\right).$$ By property $\left(2\right)$ we have for each $C\in \mathcal{C}$ an equivalence $$\underset{{v\in \Lambda^2_0}}\colim\, f^*\left(\overline{\sigma}|_{\{C\}\times \Lambda^2_0}\left(v\right)\right)\simeq f^*\underset{{v\in \Lambda^2_0}}\colim\,\overline{\sigma}|_{\{C\}\times \Lambda^2_0}\left(v\right).$$ Any transformation of algebraic theories preserves sifted colimits, so by property $\left(4\right)$, we have for each vertex $v\in \Lambda^2_0$ an equivalence $$\underset{{D\in \mathcal{C}}}\colim \,f^*\left(\overline{\sigma}|_{\mathcal{C}\times \{v\}}\left(D\right)\right)\simeq f^*\underset{{D\in \mathcal{C}}}\colim \,\overline{\sigma}|_{\mathcal{C}\times \{v\}}\left(D\right).$$ 
Hence, we get a chain of equivalences
\begin{align*}f^* \colim\, \sigma &\simeq f^*\underset{{D\in \mathcal{C}}}\colim \,\underset{{v\in \Lambda^2_0}}\colim\, \overline{\sigma}|_{\mathcal{C}\times \Lambda^2_0}\left(D\times v\right)\\ &\simeq \underset{{D\in \mathcal{C}}}\colim \,\underset{{v\in \Lambda^2_0}}\colim\, f^*\overline{\sigma}|_{\mathcal{C}\times \Lambda^2_0}\left(D\times v\right) \\
&\simeq \underset{{v\in \Lambda^2_0}}\colim\,\underset{{D\in \mathcal{C}}}\colim \, f^*\overline{\sigma}|_{\mathcal{C}\times \Lambda^2_0}\left(D\times v\right) \\
&\simeq \underset{{v\in \Lambda^2_0}}\colim\,f^*\underset{{D\in \mathcal{C}}}\colim \, \overline{\sigma}|_{\mathcal{C}\times \Lambda^2_0}\left(D\times v\right) \simeq \underset{{v\in \Lambda^2_0}}\colim\,f^*\sigma\left(v\right)
\end{align*}
which proves the proposition.
\end{proof}
We now construct the desired resolution, and complete the proof of Theorem \ref{unramifiedpreserveseffepi} by applying Proposition \ref{prop:resolveelementary} repeatedly. 
\begin{dfn}\label{monadofalawvere}
Let $\bT$ be an algebraic theory. We have seen that evaluating on the generating object $r\in \bT$ induces a conservative functor
\[\theta_r:\Alg_{\bT}\left(\Spc\right)\longrightarrow \Spc \]
which commutes with limits and sifted colimits. It follows from the adjoint functor theorem that $\theta_r$ admits a left adjoint $F$, and from the Barr-Beck theorem \cite{LurieHA} that the adjunction $F\dashv \theta_r$ is monadic, with $\mathfrak{T}=\theta_r \circ F$ the associated endomorphism monad of $\theta_r$. The left action of $\mathfrak{T}$ on $\theta_r$ induces a functor $\theta':\Alg_{\bT}\left(\Spc\right)\rightarrow \mathbf{LMod}_{\mathfrak{T}}\left(\Spc\right)$ which is an equivalence, where $\mathbf{LMod}_{\mathfrak{T}}\left(\Spc\right)$ denotes the $\i$-category of left modules for the monad $\mathfrak{T}$. 
\end{dfn}
\begin{rmk}\label{rmk:barresolution}
In the situation of Definition \ref{monadofalawvere}, any left module $M$ for $\mathfrak{T}$ comes with a resolution given by the \textbf{Bar construction} $\mathrm{Bar}_{\mathfrak{T}}\left(\mathfrak{T},M\right)_{\bullet}$, the simplicial object given by 
\[
 \begin{tikzcd} \ldots\ar[r,shift left=3]\ar[r,shift left=1]\ar[r,shift left=-1]\ar[r,shift left=-3]&  \mathfrak{T}^3M \ar[r,shift left=2]\ar[r]\ar[r,shift left=-2]&  \mathfrak{T}^2M \ar[r,shift left=1]\ar[r,shift left=-1]& \mathfrak{T}M.
 \end{tikzcd} 
\]
The face maps are induced by the multiplicative structure of the monad $\mathfrak{T}$. Given a simplicial $\bT$-algebra $A$, the Bar construction of $\theta'\left(A\right)$ is obtained by applying $\theta'$ to the simplicial object 
\[
 \begin{tikzcd} \ldots\ar[r,shift left=3]\ar[r,shift left=1]\ar[r,shift left=-1]\ar[r,shift left=-3]&  \left(F\theta_r\right)^3A \ar[r,shift left=2]\ar[r]\ar[r,shift left=-2]&   \left(F\theta_r\right)^2A \ar[r,shift left=1]\ar[r,shift left=-1]& \left(F\theta_r\right)A.
 \end{tikzcd} 
\]
Since $\theta'$ is an equivalence, it follows that the geometric realization of the above diagram is equivalent to $A$. The simplicial resolution of $A$ described above is functorial in $A$, so given a diagram $\sigma:\Lambda^2_0\rightarrow \Alg_{\bT}\left(\Spc\right)$, we obtain a diagram $\overline{\sigma}:\Delta_+^{op}\times\Lambda^{2}_0\rightarrow \Alg_{\bT}\left(\Spc\right)$, where for each vertex $v\in \Lambda^2_0$, the diagram $\Delta^{op}_+\times\{v\}\rightarrow \Alg_{\bT}\left(\Spc\right)$ is the simplicial diagram above, augmented by $\sigma\left(v\right)$. 
\end{rmk}
\begin{lem}\label{lem:barresolutioneffepi}
Let $A\rightarrow B$ be an effective epimorphism of simplicial $\bT$-algebras. Then the induced map $\left(F\theta_r\right)A\rightarrow \left(F\theta_r\right)B$ is an effective epimorphism. 
\end{lem}
\begin{proof}
By assumption, the map of spaces $A\left(r\right)\rightarrow B\left(r\right)$ is surjective on connected components. Since we can also assume that the map $A\left(r\right)\rightarrow B\left(r\right)$ is a Kan fibration of simplicial sets, it follows that we may assume that the map $A\left(r\right)\rightarrow B\left(r\right)$ is a map of simplicial sets that is a surjection in each simplicial degree. Now the object $F\left(A\left(r\right)\right)$ is itself obtained as a geometric realization; it is the colimit $$\underset{{\Delta^{op}}} \colim A\left(r\right)_n\otimes j\left(r\right),$$ where $A\left(r\right)_n\otimes j\left(r\right)$ denotes the coproduct of $j\left(r\right)$ indexed by the set $A\left(r\right)_n$. It is also clear that the map $F\left(A\left(r\right)\right)\rightarrow F\left(B\left(r\right)\right)$ is induced by a map of simplicial objects $$A\left(r\right)_{\bullet}\otimes j\left(r\right)\rightarrow B\left(r\right)_{\bullet}\otimes j\left(r\right).$$ Moreover, for each simplicial degree $n$, this map is induced by the map of sets $$A\left(r\right)_n\rightarrow B\left(r\right)_n,$$ which, as we have just argued, may be assumed to be surjective. It follows that the map $\left(F\theta_r\right)A\rightarrow \left(F\theta_r\right)B$ is an effective epimorphism as well. 
\end{proof}
\begin{rmk}
Note that if $A\left(r\right)_n$ (and thus also $B\left(r\right)_n$) is a \emph{finite} set, then the induced map $A\left(r\right)_{n}\otimes j\left(r\right)\rightarrow B\left(r\right)_{n}\otimes j\left(r\right)$ is of the form required of the vertical map in an elementary pushout diagram. 
\end{rmk}
\begin{proof}[Proof of Theorem \ref{unramifiedpreserveseffepi}]
Let $f:\bT\rightarrow \bT'$ be an unramified transformation of algebraic theories, and let $\sigma:\Lambda^2_0\rightarrow \Alg_{\bT}\left(\Spc\right)$ be a diagram
\[
\begin{tikzcd}
A\ar[d,"\alpha"]\ar[r,"\beta"]& C\\
B
\end{tikzcd}
\]
where $\alpha$ is an effective epimorphism. Let $\Delta^{op}_+\times \Lambda^{2}_0\rightarrow \Alg_{\bT}\left(\Spc\right)$ be the diagram provided by Remark \ref{rmk:barresolution}. By repeatedly applying Lemma \ref{lem:barresolutioneffepi}, we find that for every totally ordered set $[n]$, the morphism $\Delta^1\rightarrow \{[n]\}\times \Lambda^2_0\rightarrow  \Alg_{\bT}\left(\Spc\right)$ corresponding to $\left(F\theta_r\right)^{n+1}A\rightarrow \left(F\theta_r\right)^{n+1}B$ is an effective epimorphism. Applying Proposition \ref{prop:resolveelementary} we are reduced to proving that $f^*$ preserves pushout diagrams of the form 
\[
\begin{tikzcd}
\underset{{n \in\Delta^{op}}}\colim A\left(r\right)_n\otimes j\left(r\right)\ar[d]\ar[r]& \underset{{n \in\Delta^{op}}}\colim C\left(r\right)_n\otimes j\left(r\right)\\
\underset{{n \in\Delta^{op}}}\colim B\left(r\right)_n\otimes j\left(r\right).
\end{tikzcd}
\]
The vertical map is induced by a map on simplicial objects $A\left(r\right)_{\bullet}\otimes j\left(r\right)\rightarrow B\left(r\right)_{\bullet}\otimes j\left(r\right)$, and the proof of Lemma \ref{lem:barresolutioneffepi} tells us that we may assume that in each simplicial degree, this map is given by a certain codiagonal determined by the surjective map $A\left(r\right)_n\rightarrow B_n\left(r\right)$. Also, we may replace the simplicial set $C\left(r\right)$ by an equivalent simplicial set so that the map $A\left(r\right)\rightarrow C\left(r\right)$ is a cofibration. Applying Proposition \ref{prop:resolveelementary} again, we only have to prove that $f^*$ preserves pushout diagrams of the form   
\[
\begin{tikzcd}
A'\otimes j\left(r\right)\ar[d]\ar[r]& C'\otimes j\left(r\right)\\
B'\otimes j\left(r\right)
\end{tikzcd}
\]
where $A'$, $B'$ and $C'$ are maps of sets such that $A'\rightarrow B'$ is surjective and $A'\rightarrow C'$ is injective. We note that this last diagram is a small coproduct of elementary pushout diagrams; applying Proposition \ref{prop:resolveelementary} one last time to the decomposition of this coproduct diagram as a filtered colimit of finite coproducts, we see that we should show that $f^*$ preserves colimits of elementary pushout diagrams, but this is the case by assumption.   
\end{proof}

\subsection{Characterizing the finite limit envelope}
Given a pair $\left(\sC,S\right)$ with $\sC$ an idempotent complete $\i$-category with finite limits, and $S$ a $\bT$-algebra in $\sC,$ i.e. $S:\bT \to \sC$ a functor which preserves finite products, we would like to be able to identify when this is equivalent to the finite limit envelope of $\bT$.

\begin{dfn}
Let $\left(\sC,S\right)$ be a pair with $\sC$ an $\i$-category, and $S$ a $\bT$-algebra in $\sC.$ Given an object $C$ of $\sC,$ the functor $$\Map_{\sC}\left(C,S\left(\blank\right)\right):\bT \to \Spc$$ preserves finite products, and hence is a $\bT$-algebra. This assembles into a functor $$\Gamma_S:\sC\to \Alg_{\bT}\left(\Spc\right)^{op}.$$
\end{dfn}

With the hypothesis that $\sC$ additionally is idempotent complete and has finite limits, by Theorem \ref{thm:finenv}, $$\Ran_{j^{\fp}} S:\left(\Alg_{\bT}\left(\Spc\right)^{\fp}\right)^{op} \to \sC$$ exists and is left exact. Moreover, there is a canonical equivalence $$S \simeq \left(j^{\fp}\right)^*\Ran_{j^{\fp}} S.$$
The functor $\Ran_{j^{\fp}} S$ canonically induces a natural transformation
$$\eta:\omega \Rightarrow \Gamma_S \circ \Ran_{j^{\fp}} S,$$ where $$\left(\Alg_{\bT}\left(\Spc\right)^{\fp}\right)^{op} \stackrel{\omega}{\longhookrightarrow} \left(\Alg_{\bT}\left(\Spc\right)\right)^{op},$$ is the canonical inclusion.
Namely, by the Yoneda lemma, for all finitely presented $\bT$-algebras $A,$ $$A\simeq \Map_{\Alg_{\bT}\left(\Spc\right)}\left(j\left(\blank\right),A\right)=\Map_{\left(\Alg_{\bT}\left(\Spc\right)^{\fp}\right)^{op}}\left(A,j^{\fp}\left(\blank\right)\right),$$ and $$\Gamma_S \circ \Ran_{j^{\fp}} S\left(A\right)\simeq\Map_{\sC}\left(\Ran_{j^{\fp}}\left(S\right)\left(A\right),\Ran_{j^{\fp}}\left(S\right)\left(j^{\fp}\left(\blank\right)\right)\right),$$ and under these equivalences, $\eta$ is just the map on mapping spaces induced by $\Ran_{j^{\fp}}\left(S\right).$

\begin{dfn}\label{dfn:universal}
A $\bT$-algebra $S$ in $\sC$ is \textbf{versal} if 
\begin{itemize}
\item[i)] $\Gamma_S \circ \Ran_{j^{\fp}} S$ is left exact
\item[ii)] The component $\eta_S\left(r\right):\omega\left(r\right) \to \Gamma_S S\left(r\right)$ is an equivalence, where $r$ is the generator of $\bT.$
\end{itemize}
\end{dfn}

\begin{dfn}
A $\bT$-algebra $S$ in $\sC$ is \textbf{generating} if every object in $\sC$ is a retract of a finite limit of $S\left(r\right).$
\end{dfn}

\begin{dfn}
A $\bT$-algebra $S$ in $\sC$ is \textbf{universal} if and only if it is both versal and generating.
\end{dfn}

\begin{prop}
The following are equivalent for a $\bT$-algebra $S$ in $\sC$:
\begin{enumerate}
 \item $S$ is versal
 \item $\eta_S:\omega \Rightarrow \Gamma_S \circ \Ran_{j^{\fp}}S$ is an equivalence
\end{enumerate}

\end{prop}

\begin{proof}
 $\left(2\right) \Rightarrow \left(1\right)$ is clear, since $\omega$ is left exact. For $\left(1\right) \Rightarrow \left(2\right),$ notice that if $\Gamma_S \circ \Ran_{j^{\fp}}$ is left exact, then by Theorem \ref{thm:finenv}, $\left(2\right)$ holds if and only if $$\left(j^{\fp}\right)^*\eta_S:\left(j^{\fp}\right)^*\omega=j \Rightarrow \left(j^{\fp}\right)^*\Gamma_S \circ \Ran_{j^{\fp}}S$$ is an equivalence. Notice that since $j^{\fp}$ is fully faithful
 $$\left(j^{\fp}\right)^*\Gamma_S \circ \Ran_{j^{\fp}}S=\Gamma_S \circ \left(\Ran_{j^{\fp}}S\right)\circ j^{\fp}\simeq \Gamma_S \circ S.$$ Moreover, the left exactness of $\Gamma_S \circ \Ran_{j^{\fp}} S$ implies $\Gamma_S \circ S$ preserves finite products. Since $j^{\fp}$ does as well, it suffices to check that the component of  $\left(j^{\fp}\right)^*\eta_S$ along $r$ is an equivalence, which is precisely condition $ii)$ of Definition \ref{dfn:universal}.
\end{proof}

\begin{thm}\label{thm:universal1}
Let $S$ be a $\bT$-algebra in $\sC.$ Then $\Ran_{j^{\fp}} S:\left(\Alg_{\bT}\left(\Spc\right)^\fp\right)^{op} \to \sC$ is fully faithful if and only if $S$ is versal. Moreover, $\Ran_{j^{\fp}}$ is an equivalence if and only if $S$ is both versal and generating, i.e. universal.
\end{thm}

\begin{proof}
Unwinding definitions, one sees that $\eta$ is an equivalence, if and only if for all $t$ in $\bT$ and all $A$ in $\left(\Alg_{\bT}\left(\Spc\right)^\fp\right)^{op},$
$$\Map_{\left(\Alg_{\bT}\left(\Spc\right)^{\fp}\right)^{op}}\left(A,j\left(t\right)\right) \to \Map_{\sC}\left(\Ran_{j^{\fp}} S\left(A\right),\Ran_{j^{\fp}} S\left(j^{\fp}\left(t\right)\right)\right)$$ is an equivalence. So clearly, if $\Ran_{j^{\fp}} S$ is fully faithful, $\eta$ is an equivalence. Conversely, suppose $\eta$ is an equivalence. Fix a finitely presented $\bT$-algebra $A.$ Then the full subcategory of $\left(\Alg_{\bT}\left(\Spc\right)^\fp\right)^{op}$ spanned by those $B$ for which $$\Map_{\left(\Alg_{\bT}\left(\Spc\right)^{\fp}\right)^{op}}\left(A,B\right) \to \Map_{\sC}\left(\Ran_{j^{\fp}} S\left(A\right),\Ran_{j^{\fp}} S\left(B\right)\right)$$ is an equivalence is closed under finite limits and retracts, and also contains $j\left(r\right),$ hence is the whole $\i$-category by Lemma \ref{lem:retcol}.

Notice that if $\Ran_{j^{\fp}}$ is fully faithful, its essential image is the smallest subcategory of $\sC$ containing $S\left(r\right)$ and closed under retracts and finite limits. If $S$ is generating, this smallest subcategory is all of $\sC$ and hence $\Ran_{j^{\fp}}$ is also essentially surjective.

Conversely, suppose that $\Ran_{j^{\fp}}$ is an equivalence. Then in particular it is fully faithful, and hence $S$ is versal. Also, the essential image of $ \Ran_{j^{\fp}}$, on one hand is the smallest subcategory of $\sC$ containing $S\left(r\right)$ and closed under retracts and finite limits, and on the other hand is all of $\sC,$ hence $S$ is generating.
\end{proof}

\begin{rmk}\label{rmk:universalgamma}
Suppose that $S$ is universal. Then it follows, in particular, that $\Gamma_S$ restricts to a functor
$$\Gamma_S:\sC \stackrel{\sim}{\longrightarrow} \left(\Alg_{\bT}\left(\Spc\right)^\fp\right)^{op},$$ inverse to $\Ran_{j^{\fp}} S.$
\end{rmk}

Notice that the essential image of $\Ran_{j^{\fp}}S$ is contained within the smallest subcategory of $\sC$ containing $S\left(r\right)$ and closed under finite limits and retracts, so is in particular essentially small. Therefore, to check if $S$ is versal, we can without loss of generality assume that $\sC$ is essentially small. Since the Yoneda embedding $$y:\sC \hookrightarrow \Psh\left(\sC\right)$$ is fully faithful and preserves limits, one also has that $S$ is versal if and only if $y \circ S$ is. Moreover, $\Gamma_S$ has the following interpretation: The functor $y \circ S$ is in the subcategory $\Fun^{\Pi}\left(\bT,\Psh\left(\sC\right)\right)$ on those functors which preserve finite products, and there are canonical equivalences of $\i$-categories
\begin{eqnarray*}
 \Fun^{\Pi}\left(\bT,\Psh\left(\sC\right)\right) &\simeq& \Fun^{R}\left(\left(\Alg_{\bT}\left(\Spc\right)\right)^{op},\Psh\left(\sC\right)\right)\\
 &\simeq& \Fun^{R}\left(\Psh\left(\sC\right)^{op},\Alg_{\bT}\left(\Spc\right)\right)\\
 &\simeq& \Fun^{L}\left(\Psh\left(\sC\right),\left(\Alg_{\bT}\left(\Spc\right)\right)^{op}\right)\\
 &\simeq& \Fun\left(\sC,\left(\Alg_{\bT}\left(\Spc\right)\right)^{op}\right).
\end{eqnarray*}
Under these equivalences, one has that $y \circ S$ corresponds to $\Gamma_S.$ Indeed, unwinding the definitions, one has to check that $\Ran_j\left(y \circ S\right)$ is right adjoint to $\Lan_y \left(\Gamma_S\right).$ By the Yoneda lemma, it follows that the right adjoint $R_S$ to $\Lan_y \left(\Gamma_S\right)$ satisfies $$R_S\left(A\right)\left(C\right)\simeq \Map_{\left(\Alg_{\bT}\left(\Spc\right)\right)^{op}}\left(\Gamma_S\left(C\right),A\right).$$ Since the evaluation functor $$ev_C:\Psh\left(\sC\right) \to \Spc$$ preserves all limits, it preserves pointwise right Kan extensions, and hence 
\begin{eqnarray*}
\Ran_j\left(y\circ S\right)\left(A\right)\left(C\right)&\simeq& \Ran_j\left(ev_C \circ S\right)\left(A\right)\\
&\simeq& \Ran_j\left(\Map\left(C,S\left(\blank\right)\right)\right)\left(A\right)\\
&\simeq& \Map_{\left(\Alg_{\bT}\left(\Spc\right)\right)^{op}}\left(\Gamma_S\left(C\right),A\right).
\end{eqnarray*}
The last equivalence follows since for any $\bT$-algebra $B$, $\Map_{\left(\Alg_{\bT}\left(\Spc\right)\right)^{op}}\left(B,\blank\right)$ preserves limits and one moreover has that $$\Map_{\left(\Alg_{\bT}\left(\Spc\right)\right)^{op}}\left(B,j\left(\blank\right)\right)\simeq B.$$

Moreover, note that by Remark \ref{rmk:t}, we have
$$\Fun^{\Pi}\left(\bT,\Psh\left(\sC\right)\right)=\Alg_{\bT}\left(\Psh\left(\sC\right)\right)\simeq \mathbf{Geom}\left(\Psh\left(\sC\right),\cB \bT\right).$$
So both $y \circ S$ and $\Gamma_S$ are different ways of encoding the data of a geometric morphism $$\O_S:\Psh\left(\sC\right) \to \cB \bT.$$ We will now show that there is a simple topos-theoretic meaning of $S$ being versal. To explain this, we will first describe $\O_S$ explicitly. To ease notation, denote by $$\rho_S:=\Ran_{j^{\fp}} S:\left(\Alg_{\bT}\left(\Spc\right)^\fp\right)^{op} \to \sC.$$ This functor produces three adjoint functors $\left(\rho_S\right)_! \vdash \left(\rho_S\right)^* \vdash \left(\rho_S\right)_*$
$$\xymatrix{\Psh\left(\left(\Alg_{\bT}\left(\Spc\right)^\fp\right)^{op}\right) \ar@<-0.8ex>[r] \ar@<0.8ex>[r]  & \Psh\left(\sC\right) \ar[l]},$$ where $\left(\rho_S\right)_!$ and $\left(\rho_S\right)_*$ are global left and global right Kan extension respectively. In particular, the pair $\left(\left(\rho_S\right)_*,\left(\rho_S\right)^*\right)$ constitute a geometric morphism $$P_S:\cB \bT \to \Psh\left(\sC\right).$$ Notice moreover that since $\rho_S$ is left exact, by \cite[Proposition 6.1.5.2]{htt}, $\left(\rho_S\right)_!=\Lan_y \left( y \circ \rho_S\right)$ is also left exact. Hence the pair $\left(\left(\rho_S\right)_!,\left(\rho_S\right)^*\right)$ constitutes a geometric morphism $$\cO_S:\Psh\left(\sC\right) \to \cB \bT,$$ that is $\left(\cO_S\right)_*=\left(\rho_S\right)^*$ and $\left(\cO_S\right)^*=\left(\rho_S\right)_!.$

\begin{thm}
The following are equivalent for a $\bT$-algebra $S$ in $\sC$:
\begin{enumerate}
 \item $S$ is versal
 \item $\rho_S$ is fully faithful
 \item $\left(\rho_S\right)_*$ is fully faithful
 \item $\left(\rho_S\right)_!$ is fully faithful
 \item $\left(\cO_S\right)^*$ is fully faithful
 \item the canonical morphism $id_{\cB \bT} \to \cO_S \circ \rho_S$ in $\mathbf{Geom}\left(\cB \bT,\cB \bT\right)$ is an equivalence.
\end{enumerate}
\end{thm}
\begin{proof}
$\left(1\right) \iff \left(2\right)$ follows from Theorem \ref{thm:universal1}. Notice that $\left(\rho_S\right)_*$ is fully faithful if and only if the counit $$\left(\rho_S\right)^*\left(\rho_S\right)_* \to id_{\cB \bT}$$ is an equivalence, and it is standard that this is equivalent to the unit $$id_{\cB \bT} \to \left(\rho_S\right)^*\left(\rho_S\right)_!$$ being an equivalence, which in turn is equivalent to the fully faithfulness of $\left(\rho_S\right)_!=\left(\cO_S\right)^*.$ Notice that the above unit is a morphism $$id_{\cB \bT} \to \left(\rho_S\right)^*\left(\cO_S\right)^*,$$ i.e. a morphism in the $\i$-category of geometric morphisms from $\cB \bT$ to itself, and it is the morphism being referred to in $\left(6\right).$
\end{proof}

\section{Properties of $C^\infty$-rings}\label{sec:prop}
As the category $\Alg_{\Cart}\left(  \Set\right)$ of $C^\infty$-rings is presentable, in particular, it has binary coproducts, which we denote by $A \oinfty B,$ and analogous to the case of commutative rings, we will write pushouts as $$A \underset{C} \oinfty B.$$ One of the main reasons that $C^\i$-rings are so useful for geometry is that the functor
$$\sC^\i:\Mfd \to \Alg_{\Cart}\left(  \Set\right)$$ is fully faithful and preserves transverse pullbacks \cite{MSIA}. In particular, for smooth manifolds $M$ and $N,$ even though there is a proper containment of commutative $\bR$-algebras $$\sC^\i\left(  M\right) \underset{\bR}\otimes \sC^\i\left(  N\right) \subsetneq \sC^\i\left(  M \times N\right),$$ there is an isomorphism of $C^\i$-rings
$$\sC^\i\left(  M\right) \oinfty\sC^\i\left(  N\right) \cong \sC^\i\left(  M \times N\right).$$ We will eventually prove the analogous statement for $\Alg_{\Cart}\left(  \Spc\right).$ For now we start with the following:

\begin{lem}\label{lem:ff}
The functor $$\sC^\i:\Mfd \to \Alg_{\Cart}\left(  \Spc\right)$$ is fully faithful.
\end{lem}

\begin{proof}
The functor $\sC^\i$ takes $M$ to the ordinary $C^{\infty}$-ring of smooth functions on $M$, which is a fully faithful functor (see \cite{MSIA}), followed by the fully faithful inclusion of discrete objects into $\Alg_{\Cart}\left( \Spc\right)$.
\end{proof}

\subsection{The Unramified Transformation}
Using the results of the last section, we find that the theory of simplicial $C^{\infty}$-rings is in fact controlled in large part by the underlying algebraic model; in this case given by the transformation of algebraic theories $\ComR\rightarrow \Cart$. We write $\left(  \blank\right)^{alg}$ for the functor induced by this transformation; it takes values in $\Alg_{\ComR}\left( \Spc\right)$, the $\infty$-category of simplicial commutative $\bR$-algebras, and is clearly conservative. For $M$ a manifold, we will usually not write $\sC^{\i}\left( M\right)^{alg}$, to avoid cluttering up notation; it will be clear from the context when we think of $\sC^{\i}\left( M\right)$ as an $\bR$-algebra. The following remark is the main source of computational power when we deal with pushouts of simplicial rings. 
\begin{rmk}\label{torsionspectralsequence}
Recall that for a pushout diagram
\begin{equation*}
\begin{tikzcd}
A\ar[d]\ar[r] & B\ar[d] \\
C\ar[r] &D
\end{tikzcd}    
\end{equation*}
of simplicial commutative algebras (over any ring), there is a convergent spectral sequence
\begin{align}\label{torsionspecseq}
E_2^{p,q} = \mathrm{Tor}^{\pi_*\left( A\right)}_p\left( \pi_*B,\pi_*C\right)_q\Rightarrow \pi_{p+q}\left( D\right),
\end{align}
which we refer to as the \textbf{torsion spectral sequence}. See for instance \cite[Proposition 7.2.1.19]{LurieHA} and \cite[Corollary 4.1.14]{dagv}.
\end{rmk}
In the sequel, we will often take derived global sections of sheaves on manifolds, or sheaves on more general smooth spaces. In such situations, the following spectral sequence may be applied.
\begin{rmk}\label{rmk:hypercohomologyspectralsequence}
Recall that for a sheaf of complexes $\mathcal{F}$ valued in an abelian category $\mathbf{A}$ on a topological space $X$, there is a convergent \emph{hypercohomology spectral sequence}    
\begin{align}\label{hypercohspecseq}
E_2^{p,q} = H^p\left( X,H^q\left(\mathcal{F}\right)\right)\Rightarrow H^{p+q}\left( \Gamma(\mathcal{F})\right),
\end{align}
where $H^q\left(\mathcal{F}\right)$ denotes the sheaf of the $q$'th cohomology of the complex $\mathcal{F}$. 
\end{rmk}
In the classical theory of (discrete) $C^{\infty}$-rings, the weak Nulstellensatz for ideals of the $C^{\infty}$-rings $\sC^{\i}\left( \bR^n\right)$ does not hold in general. We identify two classes of ideals for which the weak Nulstellensatz is true.
\begin{dfn}\label{defn:germpointdetermined}
Let $M$ be a manifold and let $I$ be an ideal of the commutative algebra $\sC^{\infty}\left( M\right)$. Write $Z\left( I\right)$ for the common zero locus of the functions in $I$. 
\begin{enumerate}
    \item $I$ is \textbf{point determined} iff for all $f\in \sC^{\infty}\left( M\right)$, $f\in I$ iff $f\left( x\right)=0$ for all $x\in Z\left( I\right)$.
    \item $I$ is \textbf{germ determined} iff for all $f\in \sC^{\infty}\left( M\right)$, $f\in I$ iff $f_x\in I_x$ for all $x\in Z\left( I\right)$.
\end{enumerate}
\end{dfn}

It is not hard to prove that if the functions $\{f_1,\ldots,f_n\}\subset \sC^{\i}\left( M\right)$ are \textbf{independent}, that is, the common zero locus of these functions consists of regular points for the induced function $\sC^{\i}\left( M\right)\rightarrow \bR^n$, then these functions generate a point determined ideal. The following lemma is a derived analogue of this fact. 
\begin{lem}\label{projresolutiontransversal}
Let $M$ be an $m$-dimensional manifold and let $\{f_1,\ldots,f_n\}$ be independent functions on $M$. Then the Koszul algebra $\sC^{\i}\left( M\right)\left[y_1,\ldots,y_n\right]$ with $|y_i|=-1$ for $1\leq i\leq n$ and $\partial y_i=f_i$, is a projective resolution of $\sC^{\i}\left( Z\left( f_1,\ldots,f_n\right)\right)$ in the category of differentially graded $\sC^{\i}\left( M\right)$-modules. 
\end{lem}
\begin{proof}
Clearly, the Koszul complex is a complex of projective $\sC^{\i}\left( M\right)$-modules, so we should show that the complex is a resolution. Consider the sheaf of dg $\sC^{\i}\left( M\right)$-modules on $M$ given by 
\[ \mathcal{F}:U\mapsto \sC^{\i}\left( U\right)\left[y_1,\ldots,y_n\right],\,\quad \partial y_i=f_i|_U,\,1\leq i\leq n,\]
whose complex of global sections is the Koszul algebra $\sC^{\i}\left( M\right)\left[y_1,\ldots,y_n\right]$. The cohomology sheaves are sheaves of $\sC^{\i}\left( M\right)$-modules which are fine, so the hypercohomology spectral sequence collapses at the second page and the cohomology of the global sections coincides with the global sections of the cohomology sheaves. Thus, to show that the higher cohomology of the Koszul complex vanishes, it suffices to give for each point $p\in M$ a neighborhood basis $\{V_{\beta}\}$ such that $\sC^{\i}\left( V_{\beta}\right)\left[y_1,\ldots,y_n\right]$ has vanishing cohomology in degrees $<0$. The function 
$$\left( f_1,\ldots,f_n\right):M\rightarrow \bR^n$$
has full rank at $Z\left( f_1,\ldots,f_n\right)$, so it has full rank in some open neighborhood 
$$Z\left( f_1,\ldots,f_n\right)\subset V.$$ By the constant rank theorem, there is an open cover $\{U_{\alpha}\}$ of $V$ such that $U_{\alpha}\cong \bR^{m}$ and in these coordinates, the function $\left( f_1,\ldots,f_n\right)$ is the projection $$\left( x_1,\ldots,x_n\right):\bR^m\rightarrow \bR^n$$ onto the first $n$ coordinates. We have a cover $\{U_{\alpha}\}\coprod \{M\setminus Z\left( f_1,\ldots,f_n\right)\}$ of $M$ so each point in $M$ has a neighborhood basis on which $\mathcal{F}$ evaluates as either a complex of the form $$\sC^{\i}\left( V\right)\left[y_1,\ldots,y_n\right],$$ with $V\subset M\setminus Z\left( f_1,\ldots,f_n\right)$, which is acyclic because all $f_i|_{M\setminus Z\left( f_1,\ldots,f_n\right)}$ are invertible, or we have $\sC^{\i}\left( U\right)\left[y_1,\ldots,y_n\right]$, where $U\subset \bR^m$ is an open subset and $\partial y_i=x_i$, the projection onto the $i^{th}$ coordinate. Applying Hadamard's lemma repeatedly, one finds that $$\sC^{\i}\left( U\right)/\left( x_1,\ldots,x_i\right)\cong \sC^{\i}\left( U\cap \{0\}\times \bR^{m-i}\right)$$ for $1\leq i\leq n$. In particular, the function $x_{i+1}$ is a non-zero divisor of $\sC^{\i}\left( U\right)/\left( x_1,\ldots,x_i\right)$. Thus, the sequence $\left( x_1,\ldots,x_n\right)$ is a regular sequence on $\sC^{\i}\left( U\right)$ showing that the cohomology of the complex $\sC^{\i}\left( U\right)\left[y_1,\ldots,y_n\right]$ is $\sC^{\i}\left( U\cap \{0\}\times \bR^{m-n}\right)$ concentrated in degree 0. We are left to show that the zero'th cohomology of the Koszul complex is $\sC^{\i}\left( Z\left( f_1,\ldots,f_n\right)\right)$. This follows from the previous computation: the presheaf of zero'th cohomology is already a sheaf and the global sections are clearly given by $\sC^{\i}\left( Z\left( f_1,\ldots,f_n\right)\right)$.
\end{proof}

\begin{lem}\cite[Lemma 8.1]{spivak},\cite[Lemma 11.10]{dagix}\label{diffdiscunramified}
The transformation of algebraic theories $\ComR\rightarrow \Cart$ is unramified. 
\end{lem}
\begin{proof}
We should prove that for any smooth map $f:\bR^n\rightarrow \bR^m$ and any $k\geq 0$, the diagram
\[
\begin{tikzcd} 
\sC^{\i}\left( \bR^{n+m}\right)\ar[d]\ar[r] & \sC^{\i}\left( \bR^{n+m+k}\right)\ar[d] \\
\sC^{\i}\left( \bR^{n}\right) \ar[r] & \sC^{\i}\left( \bR^{n+k}\right)
\end{tikzcd}
\]
is a pushout in $\Alg_{\ComR}\left(  \Spc\right)$. We proceed by induction on $m$; for $m=0$, there is nothing to prove. For $m=1$, $f:\bR^n\rightarrow \bR$ is some smooth function. As we work with discrete objects, the torsion spectral sequence collapses at the second page, so we should show that \[\mathrm{Tor}^{\sC^{\i}\left( \bR^{n+1}\right)}_p\left( \sC^{\i}\left( \bR^{n}\right),\sC^{\i}\left( \bR^{n+1+k}\right)\right)=0,\quad p\geq 1,\]
and that 
\[\mathrm{Tor}^{\sC^{\i}\left( \bR^{n+1}\right)}_0\left( \sC^{\i}\left( \bR^{n}\right),\sC^{\i}\left( \bR^{n+1+k}\right)\right)\cong \sC^{\i}\left( \bR^{n}\right)\underset{{\sC^{\i}\left( \bR^{n+1}\right)}}\otimes\sC^{\i}\left( \bR^{n+1+k}\right)\cong \sC^{\i}\left( \bR^{n+k}\right).\]
Denote the first $n$ coordinates on $\bR^{n+1}$ collectively by $\mathbf{x}$ and the last coordinate by $y$. The function $y-f\left( \mathbf{x}\right)$ is a submersion and its zero locus is $\mathrm{Graph}\left( f\right)\cong\bR^n$, so the ring $\sC^{\i}\left( \bR^n\right)$ admits a projective resolution as an $\sC^{\i}\left( \bR^{n+1}\right)$-module of the form 

$$\left(\sC^{\i}\left( \bR^{n+1}\right)\left[z\right], \mspace{3mu} \partial z=y-f\left( \mathbf{x}\right)\right),$$ by Lemma \ref{projresolutiontransversal}. The torsion groups are computed as the cohomology of 
$$\left(\sC^{\i}\left( \bR^{n+1}\right)\left[z\right]\underset{{\sC^{\i}\left( \bR^{n+1}\right)}}\otimes\sC^{\i}\left( \bR^{n+1+k}\right) \cong  \sC^{\i}\left( \bR^{n+1+k}\right)\left[z\right],\quad \partial z=y-f\left( \mathbf{x}\right)\right).$$
By Lemma \ref{projresolutiontransversal} again, the complex on the right hand side is a resolution of $$\sC^{\i}\left( \bR^{n+1+k}\right)/\left( y-f\left( \mathbf{x}\right)\right),$$ since the map $\sC^{\i}\left( \bR^{n+1+k}\right)\rightarrow \sC^{\i}\left( \bR^{n+k}\right)$ given by restricting to the graph of $y-f\left( \mathbf{x}\right)$ induces an isomorphism $\sC^{\i}\left( \bR^{n+1+k}\right)/\left( y-f\left( \mathbf{x}\right)\right)\rightarrow \sC^{\i}\left( \bR^{n+k}\right)$.  
\\
For $m>1$, we consider the diagram 
\[
\begin{tikzcd}
\sC^{\i}\left( \bR^{n+m+1}\right)\ar[d]\ar[r] & \sC^{\i}\left( \bR^{n+m+1+k}\right)\ar[d] \\
\sC^{\i}\left( \bR^{n+1}\right)\ar[d]\ar[r] & \sC^{\i}\left( \bR^{n+1+k}\right)\ar[d] \\
\sC^{\i}\left( \bR^{n}\right) \ar[r] & \sC^{\i}\left( \bR^{n+k}\right)
\end{tikzcd}
\]
where the upper square is a pushout by the induction hypothesis applied to $\bR^{n+1}$. The large rectangle is a pushout if and only if the lower square is a pushout, so we reduce to the case $m=1$, and we are done.
\end{proof}
\begin{cor}
$\left(  \blank\right)^{alg}$ preserves pushouts along effective epimorphisms.
\end{cor}
\begin{proof}
Apply Theorem \ref{unramifiedpreserveseffepi} to the unramified transformation $\ComR\rightarrow \Cart$.
\end{proof}
Proving results by `unramifiedness' using the corollary above unlocks the powerful techniques available in the algebraic setting, and we will appeal to it several times in this article. 
\begin{rmk}
Using the resolution of effective epimorphisms between simplicial $C^{\infty}$-rings, and Lemma \ref{projresolutiontransversal}, other useful comparison results between simplicial $C^{\infty}$-rings and simplicial $\bR$-algebras can be proven. For instance, let $A,B$ be finitely generated simplicial $C^{\infty}$-rings, so that we have effective epimorphisms $\sC^{\i}\left( \bR^n\right)\rightarrow A$ and $\sC^{\i}\left( \bR^m\right)\rightarrow B$. Then the natural diagram 
\[
\begin{tikzcd}
\sC^{\i}\left( \bR^n\right)\otimes \sC^{\i}\left( \bR^n\right)\ar[d]\ar[r]& A^{alg}\otimes B^{alg}\ar[d]\\
\sC^{\i}\left( \bR^{n+m}\right)\ar[r]&\left( A\oinfty B\right)^{alg}
\end{tikzcd}
\]
is a pushout in $\Alg_{\ComR}\left( \Spc\right)$.
\end{rmk}

\begin{prop}\label{prop:fiberopen}
Let $U$ be an open subset of $\bR^n.$ Then $U$ is the fiber $f^{-1}\left(  0\right)$ of a smooth map $$f:\bR^{n+1} \to \bR^n,$$ which is transverse to $0.$
\end{prop}

\begin{proof}
Let $C$ be the compliment of $U.$ Then, there is a smooth function $\chi_C:\bR^n \to \bR$ such that $C=\chi_C^{-1}\left(  0\right).$ Define a function
\begin{eqnarray*}
\bR^{n+1} &\to& \bR\\
\left(  x_1,\ldots,x_n,t\right) & \mapsto& t\cdot \chi_C\left(  x_1,\ldots,x_n\right)-1.
\end{eqnarray*}
Notice that $f\left(  x_1,\ldots,x_n,t\right)=0$ if and only if $\chi_C\left(  x_1,\ldots,x_n\right) \ne 0,$ and $$t=\frac{1}{\chi_C\left(  x_1,\ldots,x_n\right)}.$$ Hence, $U \cong f^{-1}\left(  0\right).$ Moreover, for $\left(  x_1,\ldots,x_n,t\right) \in f^{-1}\left(  0\right),$
$$\frac{d}{dt}=f_*\left(  \frac{1}{\chi_C\left(  x_1,\ldots,x_n\right)}\frac{\partial}{\partial t}\right),$$ so $f$ is transverse to $0.$
\end{proof}

\begin{cor}\label{cor:retractrans}
Every manifold is a retract of a transverse pullback of Cartesian manifolds ( i.e. ones in $\Cart$).
\end{cor}

\begin{proof}
By Whitney's embedding theorem, every manifold $M$ can be embedded into $\bR^n.$ Let $U$ be a tubular neighborhood for $M.$ Then $M$ is a retract of $U.$ Hence we are done by Proposition \ref{prop:fiberopen}.
\end{proof}

\begin{lem}\label{preservecoproduct}
Let $M,N$ be manifolds, then the natural map $$\sC^{\i}\left( M\right)\oinfty \sC^{\i}\left( N\right)\rightarrow \sC^{\i}\left( M\times N\right)$$ is an equivalence.
\end{lem}
\begin{proof}
Take an open submanifold $U\subset \bR^n$, and let $\chi_U$ be a characteristic function for $U$. Denote by $y$ the last coordinate on $\bR^{n+1}$, let $$f\left( \mathbf{x},y\right)=\chi_U\left( x\right)y-1$$ and consider the pushout of the diagram
\[
\begin{tikzcd}
\sC^{\i}\left( \bR\right)\ar[d,"\mathrm{ev}_0"']\ar[r,"f^*"]& \sC^{\i}\left( \bR^{n+1}\right)\\
\bR 
\end{tikzcd}
\]
The left vertical map is an effective epimorphism, so by unramifiedness, we can compute this pushout in $\Alg_{\ComR}\left(  \Spc\right)$. Using the spectral sequence of Remark \ref{torsionspectralsequence}, we see that the homotopy groups of the pushout are computed as the torsion groups $\mathrm{Tor}_n^{\sC^{\i}\left( \bR\right)}\left( \bR,\sC^{\i}\left( \bR^{n+1}\right)\right)$. Using the projective resolution 
$$\left(\sC^{\i}\left( \bR\right)\left[z\right], \mspace{3mu} \partial z=x\right)$$ of $\bR$ as a $\sC^{\i}\left( \bR\right)$-module, we find that the homotopy groups are given by the cohomology of the complex $$\left(\sC^{\i}\left( \bR^{n+1}\right)\left[z\right],\mspace{3mu} \partial z=\chi_U y-1\right).$$ Lemma \ref{projresolutiontransversal} and Proposition \ref{prop:fiberopen} imply that this complex has cohomology $$\sC^{\i}\left( \bR^{n+1}\right)/\left( f\right)\cong \sC^{\i}\left( U\right)$$ concentrated in degree 0. Now for $U,V$ open submanifolds of Euclidean spaces, with presentations $\sC^{\i}\left( U\right)\cong \sC^{\i}\left( \bR^{n+1}\right)/\left( f\right)$ and $\sC^{\i}\left( V\right)\cong \sC^{\i}\left( \bR^{m+1}\right)/\left( g\right)$, the coproduct $\sC^{\i}\left( U\right)\oinfty \sC^{\i}\left( V\right)$ is the colimit of the diagram
\[
\begin{tikzcd}
\sC^{\i}\left( \bR^2\right)\ar[d,"\mathrm{ev}_0"'] \ar[r,"\left( f\times g\right)^*"]& \sC^{\i}\left( \bR^{n+1+m+1}\right)\\
\bR
\end{tikzcd}
\]
Using unramifiedness, the torsion spectral sequence, and Lemma \ref{projresolutiontransversal} again, we find that the pushout above is the discrete $C^{\infty}$-ring $\sC^{\i}\left( U\times V\right)$. \\
Now we treat the case of general manifolds $M,N$. Let $\mathbf{Dom}$ be the category of open submanifolds of Cartesian manifolds, then by Corollary \ref{cor:retractrans}, we may realize $M$ and $N$ as retracts of some $U$ respectively $V$ in $\mathbf{Dom}$. Then $M\times N$ is a retract of $U\times V$, $\sC^{\i}\left( M\right)\oinfty \sC^{\i}\left( N\right)$ is a retract of $\sC^{\i}\left( U\right)\oinfty \sC^{\i}\left( V\right)$ and $\sC^{\i}\left( M\times N\right)$ is a retract of $\sC^{\i}\left( U\times V\right)$. But as the natural map $\sC^{\i}\left( U\right)\oinfty \sC^{\i}\left( V\right)\rightarrow \sC^{\i}\left( U\times V\right)$ is an equivalence, $\sC^{\i}\left( M\right)\oinfty \sC^{\i}\left( N\right)$ and $\sC^{\i}\left( M\times N\right)$ split equivalent idempotents, so the natural map $\sC^{\i}\left( M\right)\oinfty \sC^{\i}\left( N\right)\rightarrow \sC^{\i}\left( M\times N\right)$ must be an equivalence. 
\end{proof}
\begin{rmk}\label{cor:manfp}
Notice that the proof of Lemma \ref{preservecoproduct} shows that the functor $$\sC^{\i}:\Mfd\rightarrow \Alg_{\Cart}\left( \Spc\right)$$ creates retracts of pushouts of compact projective objects of $\Alg_{\Cart}\left( \Spc\right)$, and its essential image therefore consists of finitely presented objects.
\end{rmk}
\begin{lem}\label{opentransversal}
The functor $\sC^{\i}:\Mfd\rightarrow \Alg_{\Cart}\left( \Spc\right)^{op}$ sending a manifold $M$ to the discrete simplicial $C^{\infty}$-ring of smooth functions on $M$ preserves transverse pullbacks of the form 
\[
\begin{tikzcd}
N\times_U M\ar[r]\ar[d] & N\ar[d]\\    
M\ar[r,] & U 
\end{tikzcd}
\]
where $U$ is an open submanifold of $\bR^n$ for some $n\geq 0$.
\end{lem}
\begin{proof}
We note that the pullback $N\times_U M$ is equivalent to the pullback 
\[
\begin{tikzcd}
\left( M\times N\right)\times_{U\times U}U\ar[d]\ar[r] & N\times M\ar[d,"g "]\\
U \ar[r] & U\times U
\end{tikzcd}
\] 
and, since the functor $\sC^{\i}:\Mfd\rightarrow \Alg_{\Cart}\left( \Spc\right)^{op}$ preserves binary products by Lemma \ref{preservecoproduct}, we only have to deal with pullback diagrams of the form above. Because the map $\sC^{\i}\left( U\times U\right)\rightarrow \sC^{\i}\left( U\right)$ induced by the diagonal $U\rightarrow U\times U$ is an (effective) epimorphism and the fact that the transformation of algebraic theories $\ComR\rightarrow \Cart$ is unramified, there is a natural equivalence
\[\left( \sC^{\i}\left( U\right)\underset{\sC^{\i}\left( U\times U\right)}\oinfty\sC^{\i}\left( N\times M\right)\right)^{alg}\simeq \sC^{\i}\left( U\right)\underset{{\sC^{\i}\left( U\times U\right)}}\otimes\sC^{\i}\left( N\times M\right).\]
As $\left(  \blank\right)^{alg}$ is conservative, it suffices to show that $\sC^{\i}\left( U\right)\underset{{\sC^{\i}\left( U\times U\right)}}\otimes\sC^{\i}\left( N\times M\right)$ is $0$-truncated and the natural map $$\tau_{\leq 0}\left( \sC^{\i}\left( U\right)\underset{{\sC^{\i}\left( U\times U\right)}}\otimes\sC^{\i}\left( N\times M\right)\right)\rightarrow \sC^{\i}\left( \left( M\times N\right)\times_{U\times U}U\right)$$ is an equivalence. To see this, we note that we work with discrete objects, so the torsion spectral sequence \eqref{torsionspecseq} collapses at the second page and we have natural isomorphisms
\[ \mathrm{Tor}^{\sC^{\i}\left( U\times U\right)}_n\left( \sC^{\i}\left( U\right),\sC^{\i}\left( N\times M\right)\right)\cong \pi_{n}\left( \sC^{\i}\left( U\right)\underset{{\sC^{\i}\left( U\times U\right)}}\otimes\sC^{\i}\left( N\times M\right)\right).\]
Since $U\subset \bR^n$ is open, the diagonal embedding $U\rightarrow U\times U$ is cut out by $n$ independent functions $\{f_1,\ldots,f_n\}$, so Lemma \ref{projresolutiontransversal} provides us with a projective resolution $\sC^{\i}\left( U\times U\right)\left[y_1,\ldots,y_n\right]$ of $\sC^{\i}\left( U\right)$ as a $\sC^{\i}\left( U\times U\right)$-module. The torsion groups are computed as the cohomology of 
$$\resizebox{6in}{!}{$\left(\sC^{\i}\left( U\times U\right)\left[y_1,\ldots,y_n\right]\underset{\sC^{\i}\left( U\times U\right)}\otimes\sC^{\i}\left( N\times M\right)\cong  \sC^{\i}\left( N\times M\right)\left[y_1,\ldots,y_n\right],\quad \partial y_i=f_i\circ g,\,1\leq i\leq n\right).$}$$
Because $g:N\times M\rightarrow U\times U$ is transverse to $U\rightarrow U\times U$, the functions $f_i\circ g$ are independent, so, again by Lemma \ref{projresolutiontransversal}, this complex is a projective resolution of $\sC^{\i}\left( Z\left( f_1\circ g,\ldots,f_n\circ g\right)\right)$. But $Z\left( f_1\circ g,\ldots,f_n\circ g\right)$ is the closed submanifold $\left( M\times N\right)_{U\times U}U\rightarrow N\times M$.
\end{proof}

We will show later (Theorem \ref{manifoldssmoothring}) that the functor $\sC^{\i}$ preserves \emph{all} transverse pullbacks using a local-to-global argument.

\begin{rmk}
Combing the above lemma with Corollary \ref{cor:retractrans} gives another proof that for any manifold $M,$ $\sC^\i\left(M\right)$ is in $\Alg_{\Cart}\left(\Spc\right)^{\fp}.$
\end{rmk}

\subsection{Localizations of $C^\infty$-rings}

\begin{dfn}
Let $A$ be a simplicial $C^{\infty}$-ring and let $a\in \pi_0\left( A\right)$. We say that a map $f:A\rightarrow B$ such that $f\left( a\right)\in \pi_0\left( B\right)$ is invertible is a \textbf{localization of $A$ with respect to $a$} if for each $C\in \Alg_{\Cart}\left( \Spc\right)$, the map $\Map_{\Alg_{\Cart}\left( \Spc\right)}\left( B,C\right)\rightarrow \Map_{\Alg_{\Cart}\left( \Spc\right)}\left( A,C\right)$ given by composition with $f$ induces a homotopy equivalence of Kan complexes
\[ \Map_{\Alg_{\Cart}\left( \Spc\right)}\left( B,C\right)\overset{\simeq}{\longrightarrow} \Map^0_{\Alg_{\Cart}\left( \Spc\right)}\left( A,C\right),\]
where $\Map^0_{\Alg_{\Cart}\left( \Spc\right)}\left( A,C\right)$ is the union of those connected components of $\Map_{\Alg_{\Cart}\left( \Spc\right)}\left( A,C\right)$ spanned by those maps $g$ such that $g\left( a\right)$ is invertible in $\pi_0\left( C\right)$.
\end{dfn}
In the case of an ordinary $C^{\infty}$-ring $A$ and some $a\in A$, the above definition reduces to the usual $C^{\infty}$ localization $A\left[1/a\right]$ given up to equivalence by the pushout
\begin{equation*}
\begin{tikzcd}
\sC^{\i}\left( \bR\right)\ar[r,"q_a"]\ar[d]& A\ar[d]\\
\sC^{\i}\left( \bR\setminus \{0\}\right)\ar[r] & A\left[1/a\right]
\end{tikzcd}    
\end{equation*}
of $C^{\infty}$-rings. The localization of a simplicial $C^{\infty}$-ring admits a similar characterization, for which we will need the following definition.
\begin{dfn}\label{strongmap}
\begin{enumerate}
    \item A map $f:A\rightarrow B$ of simplicial commutative rings is \textbf{strong} (in the sense of \cite[Definition 2.2.2.1]{ToeVez}) if the natural map
\[ \pi_n\left( A\right)\underset{\pi_0\left( A\right)}\otimes\pi_0\left( B\right)\rightarrow \pi_n\left( B\right)\]
is an isomorphism for all $n\geq 0$.
\item A map $f:A\rightarrow B$ of simplicial $C^{\infty}$-rings is \textbf{strong} if $f^{alg}:A^{alg}\rightarrow B^{alg}$ is strong.
\end{enumerate}
\end{dfn}
\begin{prop}\label{localization}
Let $A$ be a simplicial $C^{\infty}$ ring and let $a\in \pi_0\left( A\right)$, and let $f:A\rightarrow B$ a map of simplicial $C^{\infty}$-rings. The following are equivalent:
\begin{enumerate}
    \item The map $f:A\rightarrow B$ exhibits $B$ as a localization with respect to $a$.
    \item For every $n\geq 0$, the induced map 
    \[\pi_n\left( A^{alg}\right)\underset{\pi_0\left( A^{alg}\right)}\otimes\left( \pi_0\left( A\right)\left[1/a\right]\right)^{alg}\rightarrow \pi_n\left( B^{alg}\right) \]
    is an equivalence; that is, $f$ is strong and the map of $C^{\infty}$-schemes corresponding to $\pi_0\left( A\right)\rightarrow \pi_0\left( B\right)$ is an open inclusion. 
    \item $B$ fits into a pushout diagram
\begin{equation*}
\begin{tikzcd}
\sC^{\i}\left( \bR\right)\ar[r,"q_a"]\ar[d]& A\ar[d,"f"]\\
\sC^{\i}\left( \bR\setminus \{0\}\right)\ar[r] & B
\end{tikzcd}    
\end{equation*}
where $q_a$ is the unique up to homotopy map associated to $a\in \pi_0\left( A\right)$ (note that as a consequence, localizations always exist).
\end{enumerate}
\end{prop}
\begin{proof}
First, we show that $\left( 1\right)$ is equivalent to $\left( 3\right)$. Let $A$ be a simplicial $C^{\infty}$-ring, and choose some $a\in \pi_0\left( A\right)$. By an elementary cofinality argument, we can write $A$ as a directed colimit of finitely generated subrings $$\underset{{i\in \mathcal{J}}}\colim A_i\simeq A$$ such that $a\in \pi_0\left( A_i\right)$ for all $i\in \mathcal{J}$. We claim that the map $$\varphi:A\rightarrow \underset{{i\in \mathcal{J}}}\colim \left( A_i\left[a^{-1}\right]\right)$$ is a localization. To see this, we let $C$ be any simplicial $C^{\infty}$-ring and $f\in \Map_{\Alg_{\Cart}\left( \Spc\right)}\left( A,C\right)$ and we consider the homotopy pullback
\begin{equation*}
\begin{tikzcd}
K_f \ar[d]\ar[r] & \lim_{i\in J}\Map_{\Alg_{\Cart}\left( \Spc\right)}\left( A_i\left[a^{-1}\right],C\right)\ar[d]\\
\{f\}\ar[r] & \Map_{\Alg_{\Cart}\left( \Spc\right)}\left( A,C\right)
\end{tikzcd}    
\end{equation*}
of Kan complexes. The map $\varphi$ is a localization if and only if for each $C$ and each $f$, $K_f$ is weakly contractible if $f\left( a\right)$ is invertible in $\pi_0\left( C\right)$ and empty if $f\left( a\right)$ is not invertible. The map $f$ induces maps $f_i:A_i\rightarrow C$, and Kan complexes $$K_{f_i}:=\{f_i\}\times^h_{\Map_{\Alg_{\Cart}\left( \Spc\right)}\left( A_i,C\right)}\Map_{\Alg_{\Cart}\left( \Spc\right)}\left( A_i\left[a^{-1}\right],C\right),$$ and we have an equivalence $\lim_{i\in \mathcal{J}}K_{f_i}\simeq K_f$. If $f\left( a\right)$ (and therefore $f_i\left( a\right)$) is not invertible, $K_{f}$ is a limit of empty simplicial sets and also empty. If $f\left( a\right)$ (and therefore $f_i\left( a\right)$) is invertible, $K_{f}$ is a limit of weakly contractible Kan complexes and also weakly contractible. Now if $A_i\rightarrow A_i\left[a^{-1}\right]$ satisfies $\left( 3\right)$ for all $i\in \mathcal{J}$, then $$A\rightarrow \underset{{i\in \mathcal{J}}}\colim \left( A_i\left[a^{-1}\right]\right)$$ satisfies $\left( 3\right)$ and vice versa, so we reduce to the case of finitely generated simplicial $C^{\infty}$-rings. If $A$ is finitely generated, we have an effective epimorphism $p:\sC^{\i}\left( \bR^n\right)\rightarrow A$ for some $n$ by Proposition \ref{prop:fgeffectiveepi}, so the map $q_a:\sC^{\i}\left( \bR\right)\rightarrow A$ defining $a\in \pi_0\left( A\right)$ factors up to homotopy through $p$, which defines some $\hat{a}\in \sC^{\i}\left( \bR^n\right)$. Consider the diagram 
\begin{equation*}
\begin{tikzcd}
\sC^{\i}\left( \bR\right)\ar[r]\ar[d]&\sC^{\i}\left( \bR^n\right)\ar[d]\ar[r,"p"] & A \ar[d]\\
\sC^{\i}\left( \bR\setminus \{0\}\right)\ar[r]& \sC^{\i}\left( \bR^n\right)\left[\hat{a}^{-1}\right]\ar[r] & A\left[a^{-1}\right].
\end{tikzcd}    
\end{equation*}
The left square is a pushout, so the right square is a pushout if and only if the outer rectangle is a pushout, and we reduce to the case of free simplicial $C^{\infty}$-rings, for which we already know that the localization is given by the pushout $\left( 3\right)$ in the truncated $1$-category of $C^{\infty}$-rings, which coincides with the pushout in $\Alg_{\Cart}\left( \Spc\right)$ by Lemma \ref{opentransversal}. \\
Now we show that $\left( 3\right)$ and $\left( 2\right)$ are equivalent. First, we show that $\left( 3\right)$ implies $\left( 2\right)$. Since taking homotopy groups and tensor products commutes with filtered colimits, we may assume that we are dealing with finitely generated objects. The localization of a finitely generated object $A$ is given by the pushout diagram above for some effective epimorphism $p:\sC^{\i}\left( \bR^n\right)\rightarrow A$. Let $U\subset \bR^n$ be the open set where the function $\hat{a}$ is nonzero. By unramifiedness, $A\left[a^{-1}\right]$ is given by the pushout $$A\underset{\sC^{\i}\left( \bR^n\right)}\otimes \sC^{\i}\left( U\right)$$ of simplicial commutative rings. Moreover, since $U\rightarrow \bR^n$ is an open inclusion, the map on smooth functions is flat so applying the torsion spectral sequence we have an equivalence 
\[\pi_n\left( A\right)\underset{\sC^{\i}\left( \bR^n\right)}\otimes \sC^{\i}\left( U\right)\simeq \pi_n\left( A\underset{\sC^{\i}\left( \bR^n\right)}\otimes \sC^{\i}\left( U\right)\right) =   \pi_n\left( A\left[a^{-1}\right]\right),\]
for all $n\geq 0$, so we have equivalences
\begin{eqnarray*}
\pi_n\left( A\left[a^{-1}\right]\right) &\simeq& \pi_n\left( A\right)\underset{\sC^{\i}\left( \bR^n\right)}\otimes \sC^{\i}\left( U\right)\\
&\simeq& \pi_n\left( A\right)\underset{\pi_0\left( A\right)}\otimes \pi_0\left( A\right)\underset{\sC^{\i}\left( \bR\right)}\otimes \sC^{\i}\left( U\right)\\
&\simeq& \pi_n\left( A\right)\underset{\pi_0\left( A\right)}\otimes \pi_0\left( A\left[a^{-1}\right]\right).
\end{eqnarray*}
What remains to be shown is that $\left( 2\right)$ implies $\left( 3\right)$. If $f:A\rightarrow B$ satisfies $\left( 2\right)$, then there is an induced map $A\left[a^{-1}\right]\rightarrow B$ where $A\left[a^{-1}\right]$ is the pushout of $\left( 3\right)$. As we have just verified, this map induces an isomorphism on all homotopy groups so it is an equivalence as the functors taking homotopy groups are jointly conservative.
\end{proof} 

\begin{rmk}\label{rmk:locfp}
Combining Corollary \ref{cor:manfp} with Proposition \ref{localization} shows that the localization of a finitely presented simplicial $C^{\infty}$-ring with respect to any $a\in \pi_0\left( A\right)$ is again finitely presented.
\end{rmk}

\subsection{Spectra of $C^\i$-rings}

\subsubsection{The functor $\Speci$}
Notice that $\bR$ carries the structure of topological $C^\i$-ring. More formally, this is realized as the canonical forgetful functor $$u:\Cart \to \Top$$ to topological spaces. By Theorem \ref{thm:finenv}, the right Kan extension $$Sp:=\Ran_{j^{\fp}} u:\left(  \Alg_{\Cart}\left(  \Spc\right)^{fp}\right)^{op} \to \Top$$ preserves finite limits. Moreover, by Proposition \ref{prop:nsame}, $$Sp \simeq Sp_0\circ \pi_0$$ (see the proof of the proposition for the notation). Therefore, to describe the functor $Sp,$ it suffices to describe the analogous functor $Sp_0$ for finitely presented $\Cart$-rings in $\Set.$ Notice that forgetful functor $$\underline{\left(  \blank\right)}:\Top \to \Set$$ preserves limits, hence $$\underline{Sp_0}=\Ran_{j^{\fp}} \underline{u}.$$ 
Furthermore, the functor $$\underline{u}:\Cart \to \Set$$ preserve finite products and hence is a $\Cart$-ring; it is precisely the $\Cart$-ring $C^\i\left(  \bR^0\right)=\bR.$ 

It follows from the above that 
$$\underline{Sp_0} \left(  \blank\right) \simeq \Map_{\Alg_{\Cart}\left(  \Spc\right)^{\fp}}\left(  \blank,\bR\right)\simeq \Hom_{\Alg_{\Cart}\left(  \Set\right)^{\fp}}\left(  \pi_0\left(  \blank\right),\bR\right).$$
Hence, if $A$ is a finitely presented $C^\i$-algebra in spaces, the underlying set of $Sp\left(  A\right)$ is $$\Hom_{\Alg_{\Cart}\left(  \Set\right)^{\fp}}\left(  \pi_0\left(  A\right),\bR\right).$$ Since $Sp=\Ran_{j^{\fp}} u$ can be computed by a pointwise Kan extension formula, it follows that $Sp\left(  A\right)$ is the set $$\Hom_{\Alg_{\Cart}\left(  \Set\right)^{\fp}}\left(  \pi_0\left(  A\right),\bR\right)$$ equipped with the initial topology with respect to the maps $$\Hom_{\Alg_{\Cart}\left(  \Set\right)^{\fp}}\left(  \pi_0\left(  A\right),\bR\right) \to \Hom_{\Alg_{\Cart}\left(  \Set\right)^{\fp}}\left(  \pi_0\left(  C^\i\left(  \bR^n\right)\right),\bR\right)=\bR^n,$$ induced by all maps $$C^\i\left(  \bR^n\right) \to A,$$ and where $\bR^n$ is given the standard topology. This is the coarsest topology making each map $$\Hom_{\Alg_{\Cart}\left(  \Set\right)^{\fp}}\left(  \pi_0\left(  A\right),\bR\right) \to \Hom_{\Alg_{\Cart}\left(  \Set\right)^{\fp}}\left(  \pi_0\left(  C^\i\left(  \bR\right)\right),\bR\right)=\bR,$$ continuous, for each map $C^\i\left(  \bR\right) \to \pi_0\left(  A\right).$ However, maps $$C^\i\left(  \bR\right) \to \pi_0\left(  A\right)$$ are in natural bijection with elements of $\pi_0\left(  A\right),$ as $C^\i\left(  \bR\right)$ is the free $C^\i$-ring on one generator. Given an element $a \in \pi_0\left(  A\right),$ the induced map $$\Hom_{\Alg_{\Cart}\left(  \Set\right)^{\fp}}\left(  \pi_0\left(  A\right),\bR\right) \to \bR$$ is given simply by evaluation at $a.$ Summarizing, we have that $$Sp\left(  A\right)=\Hom_{\Alg_{\Cart}\left(  \Set\right)^{\fp}}\left(  \pi_0\left(  A\right),\bR\right)$$ equipped with the coarsest topology making each evaluation map $$ev_a:\Hom_{\Alg_{\Cart}\left(  \Set\right)^{\fp}}\left(  \pi_0\left(  A\right),\bR\right) \to \bR,$$ continuous, for each $a \in \pi_0\left(  A\right).$ More explicitly, this means that the subsets of $$\Hom_{\Alg_{\Cart}\left(  \Set\right)^{\fp}}\left(  \pi_0\left(  A\right),\bR\right)$$ of the form $ev_a^{-1}\left(  U\right)$ form a subbasis for the topology.

\begin{lem}
Each open subset of the form $ev_a^{-1}\left(  U\right)$ is also of the form $W_b:=ev_b^{-1}\left(  \mathbb{R}\setminus 0\right).$ Moreover, these open subsets are closed under finite intersection. In particular, they not only form a subbasis, but also a basis.
\end{lem}

\begin{proof}
Let $U \subseteq \mathbb{R}$ be open. Then there exists a smooth function $$\chi_{U}:\mathbb{R} \to \mathbb{R},$$ such that $U=\chi_U^{-1}\left(  \mathbb{R}\setminus 0\right).$ But this implies that 
\begin{eqnarray*}
ev_a^{-1}\left(  U\right) &=& \left(  \chi_U \circ ev_a\right)^{-1}\left(  \mathbb{R}\setminus 0\right)\\
&=& ev_{\chi_U\left(  a\right)}^{-1}\left(  \mathbb{R}\setminus 0\right).
\end{eqnarray*}
This implies the first statement. For the second, note that since $\mathbb{R}$ has no zero-divisors, it follows that $W_a \cap W_b=W_{ab}.$
\end{proof}

\begin{prop}
Every open subset of $Sp\left(  A\right)$ is of the form $W_a.$
\end{prop}

\begin{proof}
Since $\pi_0\left(  A\right)$ is finitely presented, it is in particular finitely generated, so we may find a surjective homomorphism  $$\pi:C^\i\left(  \bR^n\right) \to \pi_0\left(  A\right),$$ and $Sp\left(  A\right)=Sp\left(  \pi_0\left(  A\right)\right)$ is naturally a closed subset of $Sp\left(  C^\i\left(  \bR^n\right)\right)=\bR^n.$ Let $U$ be any open subset of $Sp\left(  A\right),$ and denote its compliment by $C.$ Then $C$ is also a closed subset of $\bR^n,$ so we may find a smooth function $g:\bR^n \to \bR$ such that $g^{-1}\left(  0\right)=C.$ It follows that $C=ev_{\pi\left(  g\right)}^{-1}\left(  0\right),$ and hence $U=W_{\pi\left(  g\right)}.$
\end{proof}

\begin{dfn}
Consider the functor
$$\Sh\left(  \blank,\Alg_{\sC^\i}\left(  \Spc\right)\right):\Top \to \widehat{\mathbf{Cat}}_\i$$
induced by the push-forward functor, and the induced coCartesian fibration
$$\int_{\Top} \Sh\left(  \blank,\Alg_{\sC^\i}\left(  \Spc\right)\right) \to \Top.$$
Informally, the objects are pairs $\left(  X,\cO_X\right)$ with $X$ a topological space and $\cO_X$ a sheaf of simplicial $C^\i$-rings on $X,$ and the morphisms are pairs $$\left(  f,\alpha\right):\left(  X,\cO_X\right) \to \left(  Y,\cO_Y\right),$$ with $f$ a continuous map and $$\alpha:\cO_Y \to f_{*}\cO_X.$$ Denote by $\Loc$ the full subcategory on those objects such that each stalk of $\pi_0\cO_X$ is a local $C^\i$-ring with residue field $\bR.$ The objects of $\Loc$ will be called \textbf{homotopically locally $C^\i$-ringed spaces}.
\end{dfn}

\begin{rmk}\label{rmk:local}
Since all the residue fields are $\bR,$ it follows that all maps between homotopically $C^\i$-ringed spaces are local, in the sense that the induced map on stalks by are maps of local rings. 
\end{rmk}

\begin{rmk}
Note that for each continuous map $f:X \to Y$ in $\Top,$ the push-forward functor
$$f_*:\Sh\left(  X,\Alg_{\sC^\i}\left(  \Spc\right)\right) \to \Sh\left(  Y,\Alg_{\sC^\i}\left(  \Spc\right)\right)$$
has a left adjoint $f^*.$ It follows from \cite[Corollary 5.2.2.5]{htt} that 
$$\int_{\Top} \Sh\left(  \blank,\Alg_{\sC^\i}\left(  \Spc\right)\right) \to \Top$$ is also a Cartesian fibration. Unwinding the definitions, we may identify it as arising from the functor
$$\Sh\left(  \blank,\Alg_{\sC^\i}\left(  \Spc\right)\right)^*:\Top^{op} \to \widehat{\mathbf{Cat}}_\i$$ induced by the pullback functor.
\end{rmk}

\begin{lem}\label{lem:finlim}
The $\i$-category $$\int_{\Top} \Sh\left(  \blank,\Alg_{\sC^\i}\left(  \Spc\right)\right)$$ has finite limits and the functor to $\Top$ preserves them.
\end{lem}

\begin{proof}
Denote by $F$ the composite
$$\Top^{op} \stackrel{\Sh\left(  \blank,\Alg_{\sC^\i}\left(  \Spc\right)\right)^*}{\longlonglongrightarrow} \widehat{\mathbf{Cat}}_\i \stackrel{\left(  \blank\right)^{op}}{\longrightarrow} \widehat{\mathbf{Cat}}_\i.$$
Then there is a canonical equivalence
$$\int_{\Top} \Sh\left(  \blank,\Alg_{\sC^\i}\left(  \Spc\right)\right) \simeq \left(  \int_{\Top^{op}} F\right)^{op}$$
over $\Top.$ So it suffices to prove that $\int_{\Top^{op}} F$ has finite colimits and the functor to $\Top^{op}$ preserves them.
Consider a continuous map $f:X \to Y$ and the restriction functor.
$$f^*:\Sh\left(  Y,\Alg_{\sC^\i}\left(  \Spc\right)\right) \to \Sh\left(  X,\Alg_{\sC^\i}\left(  \Spc\right)\right).$$
We have a canonical equivalence $$\Sh\left(  X,\Alg_{\sC^\i}\left(  \Spc\right)\right) \simeq \Alg_{\sC^\i}\left(  \Sh\left(  X\right)\right)$$ and similarly for $Y.$ Since $\Cart$ is an algebraic theory, the forgetful functor
$$\Alg_{\sC^\i}\left(  \Sh\left(  X\right)\right) \to \Sh\left(  X\right)$$ preserves limits, and hence $f^*$ preserves finite limits, since the functor $$f^*:\Sh\left(  Y\right) \to \Sh\left(  X\right)$$ does. The result now follows from \cite[Corollary 4.3.1.11]{htt}.
\end{proof}

\begin{rmk}\label{rmk:limitconc}
Unwinding the definitions, we see that moreover, given a functor from a finite $\i$-category $J$ $$F:J \to \int_{\Top} \Sh\left(  \blank,\Alg_{\sC^\i}\left(  \Spc\right)\right),$$ the limit may be computed as follows:
First, consider the fibration $$p:\int_{\Top} \Sh\left(  \blank,\Alg_{\sC^\i}\left(  \Spc\right)\right) \to \Top$$ and denote by $X$ the limit of $p \circ F.$ Let $\rho:\Delta_{X} \Rightarrow pF$ be a limiting cone for $X.$ Consider the functor
\begin{eqnarray*}
F_X:J &\to& \Sh\left(  X,\Alg_{\sC^\i}\left(  \Spc\right)\right)\\
j &\mapsto& \rho_j^*F\left(  j\right),
\end{eqnarray*}
and let $\cO_X:=\underset{j \in J} \colim F_X\left(  j\right).$ Then $\left(  X,\cO_X\right)$ is a limit of $F.$
\end{rmk}

\begin{prop}\label{prop:idemok}
The $\i$-category $$\int_{\Top} \Sh\left(  \blank,\Alg_{\sC^\i}\left(  \Spc\right)\right)$$ is idempotent complete.
\end{prop}

\begin{proof}
The proof is analogous to the proof of Lemma \ref{lem:finlim}, by replacing limits of finite shape with limits of diagrams of the form $$F:\mathbf{Idem} \to \int_{\Top} \Sh\left(  \blank,\Alg_{\sC^\i}\left(  \Spc\right)\right),$$ since retracts are preserved by all functors.
\end{proof}

\begin{rmk}\label{rmk:mappingsp}
Notice that for all $f:X \to Y$ continuous, that there is a pullback diagram in $\Spc$
$$\xymatrix{\Map\left(  \cO_Y,f_*\cO_X\right) \ar[r] \ar[d] & \Map\left(  \left(  X,\cO_X\right),\left(  Y,\cO_Y\right)\right) \ar[d]\\
\ast \ar[r]^-{f} & \Hom\left(  X,Y\right).}$$ It follows that there is an canonical equivalence
\begin{eqnarray*}
\Map\left(  \left(  X,\cO_X\right),\left(  Y,\cO_Y\right)\right) &\simeq& \underset{f \in \Hom\left(  X,Y\right)} \coprod \Map\left(  \cO_Y,f_*\cO_X\right)\\
&\simeq& \underset{f \in \Hom\left(  X,Y\right)} \coprod \Map\left(  f^*\cO_Y,\cO_X\right).
\end{eqnarray*}
Comparing this to the simplicial mapping spaces of Spivak \cite[Definition 6.3]{spivak}, in light of Remark \ref{rmk:local}, we conclude that (the homotopy coherent nerve of) Spivak's $\mathbf{LRS}$ is equivalent to the full subcategory of $\Loc$ on those objects $\left(X,\cO_X\right)$ for which $\cO_X$ is a hypersheaf.
\end{rmk}

\begin{rmk}
Since each $f^*$ is left exact and colimit preserving, it follows from \cite[Proposition 5.5.6.28]{htt}, that the $0$-truncation functors induce a canonical natural transformation
$$\pi_0:\Sh\left(  \blank,\Alg_{\sC^\i}\left(  \Spc\right)\right)^* \Rightarrow \Sh\left(  \blank,\Alg_{\sC^\i}\left(  \Set\right)\right)^*.$$ Therefore, there is an induced functor
$$t_0:\int_{\Top} \Sh\left(  \blank,\Alg_{\sC^\i}\left(  \Spc\right)\right)^* \to \int_{\Top} \Sh\left(  \blank,\Alg_{\sC^\i}\left(  \Set\right)\right)^*$$ over $\Top.$ Concretely, $$t_0\left(  X,\cO_X\right)=\left(  X,\pi_0\cO_X\right).$$
Since each $f_*$ is also left exact, the natural inclusions 
$$i_0:\Sh\left(  X,\Alg_{\sC^\i}\left(  \Set\right)\right) \hookrightarrow \Sh\left(  \blank,\Alg_{\sC^\i}\left(  \Spc\right)\right),$$ which are right adjoint to each $\pi_0,$ by \cite[Proposition 5.5.6.16]{htt}, induce a natural transformation
$$i_0:\Sh\left(  \blank,\Alg_{\sC^\i}\left(  \Set\right)\right) \Rightarrow \Sh\left(  \blank,\Alg_{\sC^\i}\left(  \Spc\right)\right),$$ and under the equivalence
$$\int_{\Top} \Sh\left(  \blank,\Alg_{\sC^\i}\left(  \Spc\right)\right)^* \simeq \int_{\Top} \Sh\left(  \blank,\Alg_{\sC^\i}\left(  \Spc\right)\right),$$ we have an adjunction $i_o \dashv t_0,$ and $i_0$ is fully faithful. 
\end{rmk}


\begin{lem}
The $\i$-category $\Loc$ has finite limits and the inclusion $$\Loc \hookrightarrow \int_{\Top} \Sh\left(  \blank,\Alg_{\sC^\i}\left(  \Spc\right)\right)$$ preserves them. Moreover $\Loc$ is idempotent complete.
\end{lem}

\begin{proof}
Notice that $\left(  X,\cO_X\right)$ is a homotopically locally $C^\i$-ringed space if and only if $t_0\left(  X,\cO_X\right)$ is a locally $C^\i$-ringed space. First we show that $\Loc$ has finite limits. Since $t_0$ preserves limits, the result follows from \cite[Proposition 7]{cinfsch}. However, since limits of shape $\mathbf{Idem}$ in any $1$-category become equalizer diagrams, we can again use \cite[Proposition 7]{cinfsch} to conclude that $\Loc$ is is closed under retracts in $$\int_{\Top} \Sh\left(  \blank,\Alg_{\sC^\i}\left(  \Spc\right)\right).$$
\end{proof}

\begin{prop}\label{prop:prodok}
Let $\left(  \bR,\sC^\i_\bR\right)$ be the smooth manifold $\bR$ regarded as on object of $\Loc.$ Then, for all $n \ge 0,$ $\left(  \bR,\sC^\i_\bR\right)^n \simeq \left(  \bR^n,\sC^\i_{\bR^n}\right).$
\end{prop}

\begin{proof}
For $n=0$ this is obvious, and for all non-zero $n,$ the case is analogous to the case $n=2,$ which we present for simplicity. According to Remark \ref{rmk:limitconc}, 
$$\left(  \bR,\sC^\i_\bR\right) \times \left(  \bR,\sC^\i_\bR\right) \simeq \left(  \bR^2,pr_1^*\sC^\i_\bR \oinfty pr_2^*\sC^\i_\bR\right).$$ The sheaf $pr_1^*\sC^\i_\bR \oinfty pr_2^*\sC^\i_\bR$ over $\bR^2$ is the sheafification of the presheaf $$W \mapsto pr_1^*\sC^\i_\bR\left(  W\right) \oinfty pr_2^*\sC^\i_\bR\left(  W\right).$$ We claim the sheafification is equivalent to $\sC^\i_{\bR^2}.$ For this, it suffices to show that for all $U,V \subseteq \bR,$ open 
$pr_1^*\sC^\i_\bR\left(  U \times V\right) \oinfty pr_2^*\sC^\i_\bR\left(  U \times V\right) \simeq \sC^\i\left(  U \times V\right).$ Since $pr_1$ and $pr_2$ are open maps, the above $C^\i$-ring is equivalent to $$\sC^\i\left(  U\right) \oinfty \sC^\i\left(  V\right),$$ so the result now follows from Lemma \ref{preservecoproduct}.
\end{proof}

\begin{cor}
$\left(  \bR,\sC^\i_\bR\right)$ has the canonical structure of a $C^\i$-ring object in $\Loc.$ 
\end{cor}

\begin{dfn}
Denote by $\underline{\bR}:\sC^\i \to \Loc$ the $C^\i$-ring from the above corollary. Define the functor
$$\Spec_{\Cart}:=\Ran_{j^{\fp}}\left(  \underline{\bR}\right):\left(  \Alg_{\Cart}\left(  \Spc\right)^{\fp}\right)^{op} \to \Loc.$$
For $A$ a finitely presented $C^\i$-ring in $\Spc,$ $\Speci\left(  A\right)$ is the $C^\i$-spectrum of $A.$

Similarly, by abuse of notation, $$\underline{\bR}:\sC^\i \to \Loc^0,$$ with $\Loc^0$ the category of spaces locally ringed in $C^\i$-rings (equivalently, the full subcategory of $\Loc$ on those objects $\left(  X,\cO_X\right)$ for which $\cO_X$ is $0$-truncated) is a $C^\i$-ring object, and we can define 
$$\Spec_{\Cart}^0:=\Ran_{j^{\fp}}\left(  \underline{\bR}\right):\left(  \Alg_{\Cart}\left(  \Set\right)^{\fp}\right)^{op} \to \Loc^0.$$
\end{dfn}

\begin{rmk}\label{rmk:to}
Since $t_0:\Loc \to \Loc^0$ preserves finite limits, and $\Loc^0$ is a $1$-category, it follows from Proposition \ref{prop:nsame} that for any $A$ a finitely presented $C^\i$-ring in $\Spc,$ $$t_0\left(  \Speci\left(  A\right)\right) \simeq \Speci^0\left(  \pi_0\left(  A\right)\right).$$ 
\end{rmk}

\begin{rmk}
Recall that the forgetful functor $$p:\Loc \to \Top$$ preserves finite limits, so it follows that $$p \circ \Speci \simeq Sp.$$
Hence we can write $\Speci\left(  A\right)=\left(  Sp\left(  A\right),\cO_A\right),$ for some sheaf $\cO_A.$
\end{rmk}

\begin{rmk}\label{rmk:hyper}
Let $A$ be an object of $\Alg_{\Cart}\left(\Spc\right)^{\fp}.$ Then since $\pi_0\left(A\right)$ is finitely presented, and $Sp\left(A\right) \cong Sp\left(\pi_0\left(A\right)\right),$ $Sp\left(A\right)$ is a closed subset of some $\mathbb{R}^n,$ and hence paracompact and of finite covering dimension. As a corollary, $\Sh\left(Sp\left(A\right)\right)$ has finite homotopy dimension and is therefore hypercomplete by \cite[Theorem 7.2.3.6 and Corollary 7.2.1.12]{htt}. In particular, $\cO_A$ is a hypersheaf.
\end{rmk}

\begin{thm}\cite[Theorem 16]{cinfsch} \label{thm:mfdgoes}
The natural functor $\Mfd \to \Loc^0$ factors as $$\Mfd \stackrel{\sC^\i}{\longlongrightarrow} \left(  \Alg_{\Cart}\left(  \Set\right)^{\fp}\right)^{op} \stackrel{\Spec_{\Cart}^0}{\longlonglongrightarrow} \Loc^0,$$ and is fully faithful.
\end{thm}

\begin{thm}\cite[Theorem 14]{cinfsch}
The functor $\Spec_{\Cart}^0$ is fully faithful, and for any finitely presented $C^\i$-algebra in $\Set,$ $$A \simeq \Gamma\left(  \cO_A\right).$$
\end{thm}

\begin{cor}\label{cor:sheafex}
For $A$ a finitely presented $C^\i$-algebra in $\Set,$ and $W_a \subseteq Sp\left(  A\right)$ an arbitrary open subset, with $a \in A,$ then $$\cO_A\left(  W_a\right)\cong A\left[a^{-1}\right].$$
\end{cor}

\begin{proof}
Notice that there is a pullback diagram
$$\xymatrix{\left(  W_a,\cO_A|_{W_a}\right) \ar[r] \ar[d] & \left(  \bR\setminus 0,\sC^\i_{\bR\setminus 0}\right) \ar[d]\\
\left(  Sp\left(  A\right),\cO_A\right) \ar[r] & \left(  \bR,\sC^\i_\bR\right)}$$
in $\Loc^0.$ As the left most vertical arrow can be identified with $\Speci^0\left(  \sC^\i\left(  \bR\right) \to \sC^\i\left(  \bR\setminus 0\right)\right)$ by Theorem \ref{thm:mfdgoes}, and since $\Speci^0$ preserves finite limits, it follows by  Proposition \ref{localization} that $$\left(  W_a,\cO_A|_{W_a}\right)\simeq \Speci\left(  A\left[a^{-1}\right]\right).$$
Hence
\begin{eqnarray*}
\cO_A\left(  W_a\right) &\cong& \Gamma\left(  \cO_A|_{W_a}\right)\\
&\cong& \Gamma\left(  \cO_{A\left[a^{-1}\right]}\right)\\
&\cong& A\left[a^{-1}\right],
\end{eqnarray*}
since any localization of a finitely presented algebra is finitely presented by Remark \ref{rmk:locfp}
\end{proof}

Later, we will give a similar description of $\cO_A$ for $A$ a finitely presented $C^\i$-algebra in $\Spc$ (Corollary \ref{cor:tildeok}).



\subsubsection{Modules for $C^\i$-rings}

\begin{dfn}
Let $A$ be in $\Alg_{\Cart}\left(  \Set\right).$ An \textbf{$A$-module} is a module for the underlying commutative $\bR$-algebra $A^{alg}$ of $A.$ Similarly, if $\cO_X$ is a sheaf of $C^\i$-rings on a space $X,$ an \textbf{$\cO_X$-module} is a module sheaf for the underlying sheaf $\cO_X^{alg}$ of commutative $\bR$-algebras.
\end{dfn}

Let $A$ be finitely presented. Note that there is global sections functor

$$\Gamma:\Mod_{\cO_A} \to \Mod_{A}.$$

\begin{prop}\cite[Theorem 5.19]{joycesch}
The functor $\Gamma$ has a left adjoint $\mathrm{M}\Speci$. Explicitly, $\mathrm{M}\Speci\left(  M\right)$ is the sheafification of the presheaf
\begin{equation}\label{eq:tildeM}
\widetilde{M}:W_a \mapsto M \underset{A^{alg}}\otimes \left(  A\left[a^{-1}\right]\right)^{alg}.
\end{equation}
\end{prop}

\begin{dfn}
An $A$-module $M$ is \textbf{complete} if the unit $$M \to \Gamma \mathrm{M}\Speci\left(  M\right)$$ is an isomorphism.
\end{dfn}
\begin{rmk}\label{rmk:completemoduleprops}
Let $A$ be a $C^{\i}$-ring. Joyce proves the following facts about the class of complete $A$-modules:
\begin{enumerate}
    \item An $A$-module $M$ is complete if and only if $M$ arises as the global sections of a sheaf of $\mathcal{O}_{Sp(A)}$-modules on $Sp\left(A\right)$. 
    \item If $A$ is finitely presented, then $M$ is a complete $A$-module if $M$ is finitely presented.
\end{enumerate}
\end{rmk}
We will also have need of the following useful property of complete modules.

\begin{prop}\label{completenesspreservation}
Let $f:A\rightarrow B$ be a surjective map of finitely generated $C^{\infty}$-rings, and let $M$ be a $B$-module. If $M$ is complete as an $A$-module (via $f$), then $M$ is complete as a $B$-module. 
\end{prop}
\begin{proof}
Considering $M$ as an $A$-module, $\mathrm{M}\Speci^A\left(  M\right)$ is the sheaf $\cF_M$ associated to the presheaf
\[W_a\mapsto M \underset{A^{alg}}\otimes \left(  A\left[a^{-1}\right]\right)^{alg} \cong M \underset{B^{alg}}\otimes \left(  B\left[f\left(  a\right)^{-1}\right]\right)^{alg}.\]
Meanwhile, $\mathrm{M}\Speci^B\left(  M\right)$ is the sheaf $\cF'_M$ associated to the presheaf 
\[W_b\mapsto  M\underset{B^{alg}} \otimes \left(  B\left[b^{-1}\right]\right)^{alg}.\]
Using that $f$ is surjective, it follows easily that for each point of $\Speci^0\left(  B\right)$, that is, for each $\phi:B\rightarrow \bR$, the map of filtered posets 
$$\left\{a \in A;\ \phi \left(f\left(a\right)\right) \neq 0 \right\} \longrightarrow \left\{b\in B;\,\phi\left(b\right)\neq 0\right\}$$
is left cofinal. Using this fact, it follows by checking on stalks that $\cF_M$ is simply the direct image sheaf of $\cF'_M$ along the map $Sp\left(  f\right):Sp\left(  B\right) \rightarrow Sp\left(  A\right)$, so the global sections of $\cF_M$ and $\cF'_M$ coincide. Thus, if $M$ is complete as an $A$-module, then $M$ is complete as a $B$-module.
\end{proof}

\begin{thm}\label{thm:complpin}
Let $A$ be a finitely presented $C^\i$-algebra in $\Spc.$ Then for all $n \ge 0,$ $\pi_n\left(  A\right)$ is a complete $\pi_0\left(  A\right)$-module.
\end{thm}

We prove this theorem at the end of this section. 
\begin{rmk}\label{rmk:finesheaves}
For the proof of the next corollary, we will use that the structure sheaf of the spectrum $\Speci^0\left(A\right)$ of a finitely presented $C^{\i}$-ring $A$ is a \emph{fine} sheaf, which is proven as \cite[Corollary 4.42]{joycesch}. It follows that sheaves of $\mathcal{O}_{Sp(A)}$-modules are also fine; in particular, the derived global section of such sheaves of modules vanish in degrees other than $0$. 
\end{rmk}

\begin{cor}\label{cor:tildeshf}
Let $A$ be a finitely presented $C^\i$-algebra in $\Spc.$ Then the presheaf
$$\widetilde{\cO_A}:W_a \mapsto A\left[a^{-1}\right]$$ is a sheaf.
\end{cor}

\begin{proof}
Notice that $$\widetilde{\cO_A}|_{W_a} \cong \widetilde{\cO_{A\left[a^{-1}\right]}}.$$ Since $W_a$ is an open subset of $Sp\left(  A\right),$ and $A\left[a^{-1}\right]$ is also finitely presented, it suffices to show that for any finitely presented $A,$ global sections of the sheafification of $\widetilde{\cO_A}$ is equivalent to $A.$ Denote this sheafification by $\overline{\cO_A}.$

Since taking homotopy groups of sheaves commutes with sheafification, we find that each $\pi_n\left(  \overline{\cO_A}\right)$ is the sheafification of the presheaf
\[  W_a\mapsto \pi_n\left(  A\left[a^{-1}\right]\right)=\pi_n\left(  A\right)\underset{\pi_0\left(  A\right)}\otimes \pi_0\left(  A\right)\left[a^{-1}\right], \]
where the last equality follows from Proposition \ref{localization}.
Now since $\pi_0\left(  A\right)$ is finitely presented as a $\Cart$-algebra in $\Set,$ the presheaf 
$$W_a\rightarrow \pi_0\left(  A\right)\left[a^{-1}\right]$$ is already a sheaf by Corollary \ref{cor:sheafex}. Also, each $\pi_n\left(  A\right)$ is a complete $\pi_0\left(  A\right)$-module by Theorem \ref{thm:complpin}, so the module of global sections of the sheaves of higher homotopy groups coincide with the homotopy groups $\pi_n\left(  A\right)$, by definition of completeness. To relate these sheaves of homotopy groups to the homotopy groups of $\Gamma\left(  \overline{\cO_A}\right)$, we have the hypercohomology spectral sequence 
\[ E_2^{p,q}=H^p\left(  Sp\left(  A\right),\pi_q\left(  \overline{\cO_A}\right)\right) \Rightarrow \pi_{q-p}\left(  \Gamma\left(  \overline{\cO_A}\right)\right).  \]
All the sheaves $\pi_q\left(  \overline{\cO_A}\right)$ are fine sheaves as they are sheaves of $\pi_0\left(  \overline{\cO_A}\right)$-modules, and $\pi_0\left(  \overline{\cO_A}\right)$ is a fine sheaf, so this spectral sequence collapses at the second page and we see that $\pi_{q}\left(  \Gamma\left(  \overline{\cO_A}\right)\right)=\pi_q\left(  A\right)$. In summary, we have found that the canonical map $ \Gamma\left(  \overline{\cO_A}\right) \rightarrow A$ induces isomorphism on all homotopy groups, so, in light of Remark \ref{rmk:hyper}, we conclude it is an equivalence.
\end{proof}

\subsubsection{More on $\Speci$}

\begin{lem}\label{lem:Ww}
Let $\left(  X,\cO_X\right)$ be a homotopically $C^\i$-ringed space, and let $f \in \Gamma\left(  \cO_X\right)$ be a global section. Denote by
$$W_f:=\left\{x \in X\mspace{3mu}|\mspace{3mu}\left[f\right]_x \notin \mathfrak{m}_x \subset \left(  \pi_0\cO_X\right)_x\right\},$$ where $\left[f\right]$ denotes the image of $f$ in $\pi_0\Gamma\left(  \cO_X\right).$ Then $W_f$ is an open subset of $X$ and $\left[f\right]|_{W_f}$ is invertible.
\end{lem}

\begin{proof}
First we will show that $W_f$ is open. Let $x \in W_f.$ Then, there exists an open neighborhood $U$ of $x$ such that $\left[f\right]|_U$ is invertible in $\pi_0\cO_X\left(  U\right).$ But for all $y \in U,$ since $\left[f\right]|_U$ is invertible, its stalk at $y$ must also be invertible, and hence be outside the maximal ideal $\mathfrak{m}_y,$ and hence we conclude that $y \in W_f.$ Now we will show that $\left[f\right]|_{W_f}$ is invertible. However, since $\pi_0\cO_X$ is a sheaf, and since the image of an invertible element under a unital ring homomorphism is always invertible, it suffices to show that there exists a cover of $$W_f=\underset{\alpha} \bigcup U_\alpha$$ such that $\left[f\right]|_{U_\alpha}$ is invertible for all $\alpha,$ but we may find such an open neighborhood of $x \in W_f,$ for all $x,$ so we are done.
\end{proof}

\begin{prop}\label{prop:specadj}
Let $\left(  X,\cO_X\right)$ be a homotopically $C^\i$-ringed space and let $A \in \Alg_{\Cart}\left(  \Spc\right)^{\fp}.$ Then taking global sections yields an equivalence
$$\Map_{\Loc}\left(  \left(  X,\cO_X\right),\left(  Sp\left(  A\right),\widetilde{\cO_A}\right)\right) \simeq \Map_{\Alg_{\Cart}\left(  \Spc\right)}\left(  A,\Gamma\left(  \cO_X\right)\right),$$
where $\widetilde{\cO_A}$ is the sheaf in Corollary \ref{cor:tildeshf}.
\end{prop}

\begin{proof}
Let $\varphi:A \to \Gamma\left(  \cO_X\right)$ be map of $\Cart$-algebras. Define a map
\begin{eqnarray*}
\theta\left(  \varphi\right):X &\to& Sp\left(  A\right)\\
x &\mapsto& \left(  A \stackrel{\varphi}{\to} \Gamma\left(  \cO_X\right) \to \left(  \cO_X\right)_x \to \left(  \pi_0\cO_X\right)_x \to \left(  \pi_0\cO_X\right)_x/\mathfrak{m}_x=\bR\right). 
\end{eqnarray*}
Notice that $\theta\left(  \varphi\right)^{-1}\left(  W_a\right) = W_{\varphi\left(  a\right)},$ so $\theta\left(  \varphi\right)$ is continuous.
Suppose now that $$\left(  f,\alpha\right):\left(  X,\cO_X\right) \to \left(  Sp\left(  A,\right),\widetilde{\cO_A}\right).$$ Then since we have a commutative diagram
$$\xymatrix{A \ar[d] \ar[r] & \widetilde{\cO_{A,f\left(  x\right)}} \ar[d] \ar[r] & \pi_0\widetilde{\cO_{A,f\left(  x\right)}} \ar[d]\\
\Gamma\left(  \cO_X\right) \ar[r] & \cO_{X,x} \ar[r] & \pi_0\cO_{X,x}}$$
the pre-image of $\mathfrak{m}_x$ is $\mathfrak{m}_{f\left(  x\right)},$ and it follows that $f=\theta\left(  \Gamma\left(  \alpha\right)\right).$
To show that $$\Map_{\Loc}\left(  \left(  X,\cO_X\right),\left(  Sp\left(  A,\right),\widetilde{\cO_A}\right)\right)  \to \Map_{\Alg_{\Cart}\left(  \Spc\right)}\left(  A,\Gamma\left(  \cO_X\right)\right)$$ is an equivalence, it suffices to prove that the homotopy fiber over any $\varphi:A \to \Gamma\left(  \cO_X\right)$ is contractible. By the above observation together with Remark \ref{rmk:mappingsp}, we may identify this homotopy fiber with the homotopy fiber of the map
$$\Map\left(  \widetilde{\cO_A},\theta\left(  \varphi\right)_*\cO_X\right)  \to \Map_{\Alg_{\Cart}\left(  \Spc\right)}\left(  A,\Gamma\left(  \cO_X\right)\right).$$ It therefore suffices to show that $\varphi$ determines $\alpha$ up to a contractible space of choices. Let $a \in A.$ Consider the naturality square for $\alpha$
$$\xymatrix@C=2.5cm{\widetilde{\cO_A}\left(  Sp\left(  A\right)\right)=A \ar[r]^-{\alpha_{Sp\left(  A\right)}=\varphi} \ar[d]& \theta\left(  \varphi\right)_*\cO_X\left(  Sp\left(  A\right)\right)=\Gamma\left(  \cO_X\right) \ar[d]\\
\widetilde{\cO_A}\left(  W_a\right)=A\left[a^{-1}\right] \ar[r]^-{\alpha_{W_a}} & \theta_*\left(  \varphi\right)_*\left(  W_a\right)=\cO_X\left(  W_{\varphi\left(  a\right)}\right)}.$$
By the universal property of $A \to A\left[a^{-1}\right],$ it follows by Lemma \ref{lem:Ww} that $\alpha_{W_a}$ is determined up to a contractible space of choices by $\varphi.$ The result now follows.
\end{proof}

\begin{cor}\label{cor:tildeok}
For any $A \in \Alg_{\Cart}^{\fp},$ $$\cO_A \simeq \widetilde{\cO_A}.$$
\end{cor}

\begin{proof}
By Proposition \ref{prop:specadj}, we conclude that $\left(  Sp\left(  A\right),\cO_A\right)$ represents the functor 
$$\Map_{\Alg_{\Cart}\left(  \Spc\right)}\left(  A,\Gamma\left(  \blank\right)\right):\Loc^{op} \to \Alg_{\Cart}\left(  \Spc\right).$$ Since the Yoneda embedding is fully faithful, we can conclude that the functor
$$A \mapsto \Map_{\Alg_{\Cart}\left(  \Spc\right)}\left(  A,\Gamma\left(  \blank\right)\right)$$ induces a functor
$$\widetilde{\Speci}:\left(  \Alg_{\Cart}\left(  \Spc\right)^{\fp}\right)^{op} \to \Loc$$ such that
$$\widetilde{\Speci}\left(  A\right)=\left(  Sp\left(  A\right),\cO_A\right).$$ Moreover, by Lemma \ref{lem:retcol} this functor manifestly preserves finite limits, since the original functor did and the Yoneda embedding reflects limits. Notice that by Proposition \ref{localization}, $$\widetilde{\Speci}\left(  \sC^\i\left(  \bR\right)\right) \simeq \left(  \bR,\sC^\i_{\bR}\right) \simeq \Speci\left(  \sC^\i\left(  \bR\right)\right).$$ Since both $\Speci$ and $\widetilde{\Speci}$ preserve finite limits, we conclude by Theorem \ref{thm:finenv} that they are equivalent. The result now follows.
\end{proof}

\begin{thm}\label{thm:speciff}
The functor $$\Speci:\left(  \Alg_{\Cart}\left(  \Spc\right)^{\fp}\right)^{op} \to \Loc$$ is fully faithful.
\end{thm}

\begin{proof}
Let $A$ and $B$ be in $\Alg_{\Cart}\left(  \Spc\right)^{\fp},$ then
\begin{eqnarray*}
\Map\left(  A,B\right) &\simeq& \Map\left(  A,\Gamma\left(  \cO_B\right)\right)\\
&\simeq& \Map\left(  \left(  Sp\left(  B\right),\cO_B\right),\left(  Sp\left(  A\right),\cO_A\right)\right).
\end{eqnarray*}
\end{proof}

\begin{prop}\label{prop:mansame}
Let $M$ be a manifold. Then
$$\Speci\left(  \sC^\i\left(  M\right)\right) \simeq \left(  M,\sC^\i_M\right).$$
\end{prop}

\begin{proof}
It follows from Proposition \ref{prop:prodok} combined with the proof of Corollary \ref{cor:sheafex} that the result holds for open subsets of $\bR^n.$ Denote by $\mathbf{Dom}$ the full subcategory of $\Mfd$ on those manifolds diffeomorphic to an open subset of $\bR^n$ for some $n.$ Then, as $\Mfd$ is the Karoubi envelope of $\mathbf{Dom},$ we have that restriction functor
$$\Fun\left(  \Mfd,\Loc\right) \to \Fun\left(  \mathbf{Dom},\Loc\right)$$ is an equivalence of $\i$-categories. Consider the functor
\begin{eqnarray*}
F:\Mfd &\mapsto& \Loc\\
M &\mapsto& \left(  M,\sC^\i_M\right)
\end{eqnarray*}
Notice that $F$ and $\Speci \circ \sC^\i\left(  \blank\right)$ both restrict to the same functor on $\mathbf{Dom}.$ It follows that they must agree on $\Mfd$ as well.
\end{proof}
 
\begin{thm}\label{manifoldssmoothring}
The functor $\sC^\i:\Mfd \hookrightarrow \Alg_{\Cart}\left(  \Spc\right)^{op}$ sending a manifold $M$ to the discrete simplicial $C^{\infty}$-ring of smooth functions on $M$ is fully faithful, and preserves transverse pullbacks. 
\end{thm}
\begin{proof}
Fully faithfulness is the content of Lemma \ref{lem:ff}. Suppose we have a transverse pullback diagram
$$\xymatrix{P \ar[d]_-{pr_M} \ar[r]^-{pr_N} \ar[rd]^-{\psi} & N \ar[d]^-{g}\\
M \ar[r]^-{f} & L}.$$
Note that by \cite[Chapter 1, Theorem 2.8]{MSIA}, $$\pi_0\left(  \sC^\i\left(  M\right) \underset{\sC^\i\left(  L\right)} \oinfty \sC^\i\left(  N\right)\right) \cong \sC^\i\left(  P\right).$$ It therefore suffices to show that for all $n \ge 1,$ $$\pi_n\left(  \sC^\i\left(  M\right) \underset{\sC^\i\left(  L\right)} \oinfty  \sC^\i\left(  N\right)\right)=0.$$ Note that $$\Pi_n:=\pi_n\left(  \sC^\i\left(  M\right) \underset{\sC^\i\left(  L\right)} \oinfty  \sC^\i\left(  N\right)\right)$$ has the canonical structure of a $\sC^\i\left(  P\right)$-module, and is moreover complete by Theorem \ref{thm:complpin}. Hence the global sections of the sheaf associated to the presheaf on $P$ defined by
$$U \mapsto \sC^\i\left(  U\right) \underset{\sC^\i\left(  P\right)} \otimes \Pi_n$$ coincide with $$\pi_n\left(  \sC^\i\left(  M\right) \underset{\sC^\i\left(  L\right)} \oinfty  \sC^\i\left(  N\right)\right)$$ for all $n\ge 1$. Choose a cover of $L$ by open subsets $U_\alpha$ such that each $$U_\alpha \cong \bR^l,$$ where $l= \dim L.$ By Proposition \ref{localization}, it follows that the presheaf defined above may equivalently be described by
$$W_a \mapsto \pi_n\left(  \left(  \sC^\i\left(  M\right) \underset{\sC^\i\left(  L\right)} \oinfty \sC^\i\left(  N\right)\right)\left[a^{-1}\right]\right).$$
For each $U_\alpha,$ we may choose smooth functions $\varphi_\alpha:L \to \bR$ such that $U_\alpha=\varphi^{-1}\left(  U_\alpha\right).$ 
Again by Proposition \ref{localization}, setting $$a=\varphi_\alpha \circ \psi \in \sC^\i\left(  P\right) = \pi_0\left(  \left(  \sC^\i\left(  M\right) \underset{\sC^\i\left(  L\right)} \oinfty \sC^\i\left(  N\right)\right)\right),$$ we have
$$\left(  \sC^\i\left(  M\right) \underset{\sC^\i\left(  L\right)} \oinfty \sC^\i\left(  N\right)\right)\left[a^{-1}\right] \simeq \left(  \sC^\i\left(  M\right) \underset{\sC^\i\left(  L\right)} \oinfty \sC^\i\left(  N\right)\right) \underset{\sC^\i\left(  \bR\right)} \oinfty \sC^\i\left(  \bR\setminus 0 \right),$$ and since we have a factorization
$$\sC^\i\left(  \bR\right) \to \sC^\i\left(  L\right) \to \sC^\i\left(  M\right) \underset{\sC^\i\left(  L\right)} \oinfty \sC^\i\left(  N\right),$$ it follows that 
\begin{eqnarray*}
\left(  \sC^\i\left(  M\right) \underset{\sC^\i\left(  L\right)} \oinfty \sC^\i\left(  N\right)\right)\left[a^{-1}\right] &\simeq& \left(  \sC^\i\left(  M\right) \underset{\sC^\i\left(  L\right)} \oinfty \sC^\i\left(  N\right)\right) \underset{\sC^\i\left(  L\right)} \oinfty \sC^\i\left(  U_\alpha\right)\\
&\simeq& \left(  \sC^\i\left(  M\right) \underset{\sC^\i\left(  L\right)} \oinfty \sC^\i\left(  U_\alpha\right)\right) \underset{\sC^\i\left(  L\right)} \oinfty \sC^\i\left(  N\right).
\end{eqnarray*}
Observe that
\begin{eqnarray*}
\sC^\i\left(  M\right) \underset{\sC^\i\left(  L\right)} \oinfty \sC^\i\left(  U_\alpha\right) &\simeq& \Gamma\left(   \left(  M,\sC^\i_M\right) \times_{\left(  L,\sC^\i_L\right)} \left(  U_\alpha,\sC^\i_{U_\alpha}\right)\right)\\
&\simeq& \Gamma\left(  \left(  f^{-1}\left(  U_\alpha\right),\sC^\i_M|_{f^{-1}\left(  U_\alpha\right)}\right)\right)\\
&\simeq& \sC^\i\left(  f^{-1}\left(  U_\alpha\right)\right),
\end{eqnarray*}
by Proposition \ref{prop:mansame}, so we have $$\left(  \sC^\i\left(  M\right) \underset{\sC^\i\left(  L\right)} \oinfty \sC^\i\left(  N\right)\right)\left[a^{-1}\right] \simeq \sC^\i\left(  f^{-1}\left(  U_\alpha\right)\right)   \underset{\sC^\i\left(  L\right)} \oinfty \sC^\i\left(  N\right).$$
Notice, by Lemma \ref{opentransversal},
$$\sC^\i\left(  f^{-1}\left(  U_\alpha\right)\right) \underset{\sC^\i\left(  U_\alpha\right)} \oinfty \sC^\i\left(  g^{-1}\left(  U_\alpha\right)\right) \simeq \sC^\i\left(  f^{-1}\left(  U_\alpha\right)\times_{U_\alpha} g^{-1}\left(  U_\alpha\right)\right),$$ so we have
\begin{eqnarray*}
\left(  \sC^\i\left(  M\right) \underset{\sC^\i\left(  L\right)} \oinfty \sC^\i\left(  N\right)\right)\left[a^{-1}\right] &\simeq& \sC^\i\left(  f^{-1}\left(  U_\alpha\right)\right)   \underset{\sC^\i\left(  L\right)} \oinfty \sC^\i\left(  N\right)\\
&\simeq& \sC^\i\left(  f^{-1}\left(  U_\alpha\right)\right)   \underset{\sC^\i\left(  U_\alpha\right)} \oinfty \sC^\i\left(  U_\alpha\right) \underset{\sC^\i\left(  L\right)} \oinfty \sC^\i\left(  N\right)\\
&\simeq& \sC^\i\left(  f^{-1}\left(  U_\alpha\right)\right)   \underset{\sC^\i\left(  U_\alpha\right)} \oinfty \sC^\i\left(  g^{-1}\left(  U_\alpha\right)\right)\\
&\simeq& \sC^\i\left(  f^{-1}\left(  U_\alpha\right)\times_{U_\alpha} g^{-1}\left(  U_\alpha\right)\right).
\end{eqnarray*}
So it follows that $$\pi_n\left(  \left(  \sC^\i\left(  M\right) \underset{\sC^\i\left(  L\right)} \oinfty \sC^\i\left(  N\right)\right)\left[a^{-1}\right]\right)=0.$$ This implies that $$\mathrm{M}\Speci\left(  \Pi_n\right)\left(  \psi^{-1}\left(  U_\alpha\right)\right)=0.$$ Since the open subsets $\psi^{-1}\left(  U_\alpha\right)$ form a cover, this implies that the sheaf itself is zero, and since $\Pi_n$ is a complete $\sC^\i\left(  P\right)$-module, this implies
$$\Pi_n \simeq \Gamma \mathrm{M}\Speci\left(  \Pi_n\right) =0.$$
\end{proof}
\subsection{The proof of Theorem \ref{thm:complpin}}
We are left to prove Theorem \ref{thm:complpin}. We do this by fixing strict models for finitely presented simplicial $C^{\i}$-rings in the model category $\mathbf{cdga}_{\bR}^{\leq 0}$ of differentially graded commutative $\bR$-algebras in non-positive degrees. 
\begin{rmk}
Let $V$ be a real vector space, possibly of infinite dimension. We write $$\sC^{\i}\left( V^{\vee}\right):=\underset{{V'\subset V\,\mathrm{dim}V'<\infty}}\colim \sC^{\i}\left( \left( V'\right)^{\vee}\right)$$ for the free $C^{\i}$-ring on $V$. Evaluation at $0\in V^{\vee}$ yields a map $\sC^{\i}\left( V^{\vee}\right)\rightarrow \bR$ of simplicial $C^{\i}$-rings, so $\sC^{\i}\left( V^{\vee}\right)$ is augmented over the initial object in $\Alg_{\Cart}\left( \Spc\right)$, and we may consider the $n$-fold suspension $\Sigma^n\sC^{\i}\left( V^{\vee}\right)$ with respect to the augmentation. 
\end{rmk}
\begin{dfn}\label{defn:cellobject}
A simplicial $C^{\i}$-ring $A$ is a \textbf{cell object} if $A$ is obtained as a directed colimit 
\[\bR\overset{\phi_{-1}}{\longrightarrow} A_0\overset{\phi_0}{\longrightarrow} A_1 \longrightarrow \ldots,\]
where $\phi_{-1}$ is a pushout along a map of the form $\bR\rightarrow \sC^{\i}\left( V\right)$ for $V$ a possibly infinite dimensional vector space, and $\phi_n$ for $n\geq0$ is a pushouts along a map of the form $\Sigma^n\sC^{\i}\left( V\right)\rightarrow \bR$ for $V$ a possibly infinite dimensional vector space. A cell object is \textbf{finite} if the directed colimit in the definition is indexed by a finite set and there are only finitely many cells in each degree. \end{dfn}
\begin{prop}\label{afpcell}
Let $A$ be a simplicial $C^{\i}$-ring. 
\begin{enumerate}
    \item $A$ is equivalent to a cell object. 
    \item If $A$ is finitely presented, then $A$ is equivalent to a retract of a finite cell object. 
\end{enumerate}
\end{prop}
We prove Proposition \ref{afpcell} at the end of this subsection. 
\begin{proof}[Proof of Theorem \ref{thm:complpin}]
Let $A$ be a finitely presented simplicial $C^{\infty}$-ring. Using Proposition \ref{afpcell} and the fact that the property described in Theorem \ref{thm:complpin} is stable under retracts, we may assume that $A$ has a presentation as a finite cell object. Such a cell object is inductively obtained by pushouts of the from 
\[
\begin{tikzcd}
\Sigma^{n-1}\sC^{\i}\left( \bR^n\right)\ar[r]\ar[d] & A_{n-1}\ar[d] \\
\bR \ar[r] & A_n
\end{tikzcd}
\]
where $A_0=\sC^{\i}\left( \bR^m\right)$, for some finite $m$. By unramifiedness, $A^{alg}$ is given by the colimit of maps obtained by the same sequence of pushout diagrams in $\Alg_{\ComR}\left( \Spc\right)$. Recall the left proper combinatorial model category structure on $\mathbf{cdga}^{\leq 0}_{\bR}$ which presents the $\infty$-category $\Alg_{\ComR}\left( \Spc\right)$. Lemma \ref{projresolutiontransversal} implies that in the model category $\mathbf{cdga}^{\leq 0}_{\bR}$, the morphism $\sC^{\i}\left( \bR^n\right)\rightarrow \bR$ has a cofibrant replacement as $\sC^{\i}\left( \bR^n\right)\rightarrow \sC^{\i}\left( \bR^n\right)\left[y_1,\ldots,y_n\right]$, with $y_i$ in degree $-1$ and differential $\partial y_i=x_i$, the $i^{th}$ coordinate function on $\bR^n$. Since $$\Sigma^{n-1}\sC^{\i}\left( \bR^k\right)^{alg}\simeq \bR\left[\epsilon_1,\ldots,\epsilon_k\right]$$ with $|\epsilon_i|=n-1$ for $n>1$, the map $\Sigma^{n-1}\sC^{\i}\left( \bR^k\right)^{alg}\rightarrow \bR$ can be replaced by a finite coproduct of copies of the generating cofibration $\bR\left[\epsilon^i\right]\rightarrow \bR\left[\epsilon^i,\epsilon^{i+1}\right]$. As the model category $\mathbf{cdga}^{\leq 0}_{\bR}$ is left proper, it follows that $A^{alg}$ is given by the (ordinary) colimit over a sequence of maps obtained by pushouts along the cofibrations we have just described, so $A^{alg}$ has a presentation in $\mathbf{cdga}^{\leq 0}_{\bR}$ by a quasi-free object of the form \[\tilde{A}=\sC^{\i}\left( \bR^m\right)\left[\epsilon^1_{1},\ldots,\epsilon^1_{l_1},\epsilon^2_{1},\ldots,\epsilon^2_{l_2},\ldots,\epsilon^k_{1},\ldots,\epsilon^k_{l_k}\right]\]
where $|\epsilon^i_{j_i}|=-i$ for $1\leq i\leq k$ (we have suppressed the (nontrivial!) differentials). Fix $n>0$, and consider the truncated dga $\tau_{\geq -\left( n+1\right)}\tilde{A}$, so that we have 
$$H^{-n}\left( \tau_{\geq -\left( n+1\right)}\tilde{A}\right)\cong H^{-n}\left( \tilde{A}\right)=\pi_n\left( A\right).$$ 
As $\tilde{A}$ is a \emph{finite} cell dga, $\tilde{A}$ is a finitely generated free $\sC^{\i}\left( \bR^m\right)$-module in each degree, so $\tau_{\geq -\left( n+1\right)}\tilde{A}$ is a finitely presented $\sC^{\i}\left( \bR^m\right)$-module. Now consider the presheaf of dg $\sC^{\i}\left( \bR^m\right)$-modules on $\bR^m$ given by 
\[ \mathcal{F}:=U\mapsto U\mapsto \tau_{\geq -\left( n+1\right)}\left( \sC^{\i}\left( U\right)\left[\epsilon^1_{1},\ldots,\epsilon^1_{l_1},\epsilon^2_{1},\ldots,\epsilon^2_{l_2},\ldots,\epsilon^k_{1},\ldots,\epsilon^k_{l_k}\right]\right), \]
whose module of global sections is $\tau_{\geq -\left( n+1\right)}\tilde{A}$. This presheaf is a sheaf, precisely because $\tau_{\geq -\left( n+1\right)}\tilde{A}$ is a finitely presented and thus complete $\sC^{\i}\left( \bR^m\right)$-module. By fineness and an appeal to the hypercohomology spectral sequence, the cohomology groups of $\tau_{\geq -\left( n+1\right)}\tilde{A}$ are given by the global sections of the sheaves of cohomology groups of $\mathcal{F}$. This implies in particular that $H^{-n}\left( \tilde{A}\right)$ is a complete $\sC^{\i}\left( \bR^n\right)$-module, by Remark \ref{rmk:completemoduleprops}. As the map $\sC^{\i}\left( \bR^m\right)\rightarrow \pi_0\left( A\right)$ is surjective, the module 
$$H^{-n}\left( \tilde{A}\right)\underset{\sC^{\i}\left( \bR^m\right)}\otimes\pi_0\left( A\right)\cong H^{-n}\left( \tilde{A}\right)\cong \pi_n\left( A\right)$$
is a complete $\pi_0\left( A\right)$-module by Proposition \ref{completenesspreservation}.
\end{proof}
The rest of this subsection is devoted to the proof of Proposition \ref{afpcell}. The following lemmas are adapted from \cite[Lemmas 12.18 and 12.19]{dagix}.
\begin{rmk}
The free $C^{\i}$-ring functor $F$ preserves colimits, so we have $$\Sigma^n\sC^{\i}\left( V^{\vee}\right)\simeq F\left( \mathrm{Sym}^{\bullet}\left( V\left[n\right]\right)\right).$$ The forgetful-free adjunction between $\mathbb{E}_{\infty}$-algebras and simplicial $C^{\i}$-rings now establishes the equivalence
\[ \Map_{\Alg_{\Cart}\left( \Spc\right)}\left( \Sigma^n\sC^{\i}\left( V^{\vee}\right),A\right)\simeq \Map_{\mathsf{Mod}_{\bR}}\left( V\left[n\right],A^{alg}\right)  \]
for all $A\in \Alg_{\Cart}\left( \Spc\right)$.
\end{rmk}

\begin{lem}\label{looping1}
The map $$V\left[n\right]\rightarrow \Sigma^n\sC^{\i}\left( V^{\vee}\right)^{alg}$$ corresponding to the identity $\Sigma^n\sC^{\i}\left( V^{\vee}\right)\rightarrow \Sigma^n\sC^{\i}\left( V^{\vee}\right)$ via the equivalences above induces an equivalence $\mathrm{Sym}^{\bullet}\left( V\left[n\right]\right)\rightarrow \Sigma^n\sC^{\i}\left( V^{\vee}\right)^{alg}$ of $\mathbb{E}_{\infty}$-algebras over $\bR$ for $n>0$.
\end{lem}
\begin{proof}
Since all forgetful and free functors involved commute with filtered colimits, we may write $$V=\underset{{V'\subset V,\,\mathrm{dim}\,V<\infty}}\colim V'$$ and suppose that $V$ is finite dimensional. We work by induction on $n$. For $n=1$, we are asked to prove that the natural map \[\bR\underset{\mathrm{Sym}^{\bullet}\left( V\right)}\otimes\bR\rightarrow \bR\underset{\sC^{\i}\left( V^{\vee}\right)^{alg}}\oinfty\bR\simeq \bR\underset{\sC^{\i}\left( V^{\vee}\right)^{alg}}\otimes\bR\]
is an equivalence (the last equivalence follows by unramifiedness). Suppose that $V$ is 1-dimensional, then $\mathrm{Sym}^{\bullet}\left( V\right)\simeq \bR\left[x\right]$ and we have a map of projective resolutions
\[
\begin{tikzcd}
0\ar[r] & \bR[x] \ar[r,"x"] \ar[d] &  \bR[x] \ar[r] \ar[d]& \bR\ar[d,"\mathrm{id}"] \\
0\ar[r] & \sC^{\i}\left( \bR\right)\ar[r,"x"]  & \sC^{\i}\left( \bR\right) \ar[r] & \bR
\end{tikzcd}
\]
where $x$ denotes multiplication by the function $x\mapsto x$ on $\bR$
which shows that $$\mathrm{Tor}_i^{\bR\left[x\right]}\left( \bR,\bR\right)\cong \mathrm{Tor}_i^{\sC^{\i}\left( \bR\right)}\left( \bR,\bR\right)$$ for all $i\geq0$, so we are done for $n=1$ and $\mathrm{dim}\,V=1$. For $V$ $k$-dimensional, the map $ \mathrm{Sym}^{\bullet}\left( V\left[1\right]\right)\rightarrow \Sigma \sC^{\i}\left( V^{\vee}\right)^{alg}$ is simply the $k$-fold tensor product of the equivalence we have just established. The induction step for $n\geq 1$ follows at once from unramifiedness.
\end{proof}

\begin{lem}\label{looping2}
Let $A$ be a simplicial $C^{\i}$-ring and let $V$ be a vector space. Let $n>0$ and $V\left[n\right]\rightarrow A^{alg}$ be a map of $\bR$-modules adjoint to a map $\varphi:V\underset{\bR}\otimes A^{alg}\left[n\right]\rightarrow A^{alg}$ of $A^{alg}$-modules. By taking the symmetric algebra and the free simplicial $C^{\i}$-ring, $V\left[n\right]\rightarrow A^{alg}$ is adjoint to a map $\Sigma^n\sC^{\i}\left( V^{\vee}\right)\rightarrow A$. Consider the pushout diagram
\[
\begin{tikzcd}
\Sigma^n\sC^{\i}\left( V^{\vee}\right)\ar[d]\ar[r]& A\ar[d]\\
\bR\ar[r] & B
\end{tikzcd}
\]
Then there is a natural map $\mathrm{cofib}\left( \varphi\right)\rightarrow B^{alg}$ of $A^{alg}$-modules which has $\left( 2n+2\right)$-connective cofiber. 
\end{lem}
\begin{proof}
By unramifiedness and Lemma \ref{looping1}, we have $B^{alg}\simeq \bR \underset{\mathrm{Sym}^{\bullet}\left( V\left[n\right]\right)}\otimes A^{alg}$. The composition
\[ V\underset{\bR}\otimes A^{alg}\left[n\right] \overset{\varphi}{\longrightarrow} A^{alg} \longrightarrow B \]
of $A^{alg}$-modules is homotopic to the composition 
\[ V\underset{\bR}\otimes A^{alg}\left[n\right] {\longrightarrow} \mathrm{Sym}^{\bullet}\left( V\left[n\right]\right) \longrightarrow B  \]
which is null-homotopic for degree reasons, yielding the desired map $\mathrm{cofib}\left( \varphi\right)\rightarrow B$. Since taking cofibers commutes with tensor products, we have an equivalence 
$$\resizebox{6in}{!}{$\mathrm{cofib}\left( V\left[n\right]\underset{\bR}\otimes\mathrm{Sym}^{\bullet}\left( V\left[n\right]\right)\rightarrow \mathrm{Sym}^{\bullet}\left( V\left[n\right]\right)\right)\underset{\mathrm{Sym}^{\bullet}\left( V\left[n\right]\right)}\otimes A^{alg}\simeq \mathrm{cofib}\left( V\underset{\bR}\otimes A^{alg}\left[n\right]\rightarrow A^{alg}\right)=\mathrm{cofib}\left( \varphi\right).$}$$
One readily verifies that  $\mathrm{cofib}\left( V\left[n\right]\underset{\bR}\otimes\mathrm{Sym}^{\bullet}\left( V\left[n\right]\right)\rightarrow \mathrm{Sym}^{\bullet}\left( V\left[n\right]\right)\right)$ has vanishing homotopy groups in degrees $0<i\leq 2n$, so the map $$\mathrm{cofib}\left( V\left[n\right]\underset{\bR}\otimes\mathrm{Sym}^{\bullet}\left( V\left[n\right]\right)\rightarrow \mathrm{Sym}^{\bullet}\left( V\left[n\right]\right)\right)\rightarrow \bR$$ has $\left( 2n+2\right)$-connective cofiber, showing that the map
$$\resizebox{6in}{!}{$\mathrm{cofib}\left( V\left[n\right]\underset{\bR}\otimes\mathrm{Sym}^{\bullet}\left( V\left[n\right]\right)\rightarrow \mathrm{Sym}^{\bullet}\left( V\left[n\right]\right)\right)\underset{\mathrm{Sym}^{\bullet}\left( V\left[n\right]\right)}\otimes A^{alg}\simeq \mathrm{cofib}\left( \varphi\right) \rightarrow  B^{alg}\simeq \bR\underset{\mathrm{Sym}^{\bullet}\left( V\left[n\right]\right)}\otimes A^{alg}$}$$
has $\left( 2n+2\right)$-connective cofiber as well. 
\end{proof}

\begin{proof}[Proof of Proposition \ref{afpcell}]
\begin{enumerate}
\item Let $A$ be a simplicial $C^{\i}$-ring. We will inductively define a sequence of $n$-connective maps $\psi_n:A_n\rightarrow A$ as follows. For the base step of the induction, choose an effective epimorphism $\sC^{\i}\left( \bR^{J_0}\right)\rightarrow A$; for instance, $J_0$ may be the set underlying $\pi_0\left( A\right)$. Now let $n>0$. Assuming we have constructed an $\left( n-1\right)$-connective map $\psi_{n-1}:A_{n-1}\rightarrow A$, we construct $\psi_n$. We have $\pi_j\left( A_{n-1}\right)\simeq \pi_j\left( A\right)$ for $j<\left( n-1\right)$. The algebraic fiber $\mathrm{fib}\left( \psi_{n-1}^{\mathrm{alg}}\right)$ of the map $\psi_{n-1}^{\mathrm{alg}}:A_{n-1}^{\mathrm{alg}}\rightarrow A^{alg}$ of connective $\mathbb{E}_{\infty}$-algebras over $\bR$ fits into a long exact sequence
\[\ldots\rightarrow \pi_{n}\left( A^{alg}\right)\rightarrow \pi_{n-1}\left( \mathrm{fib}\left( \psi_{n-1}^{\mathrm{alg}}\right)\right)\rightarrow \pi_{n-1}\left( A^{alg}_{n-1}\right)\rightarrow \pi_{n-1}\left( A^{alg}\right)\rightarrow\ldots \]
Choose a set $J_n$ and a map $\bR^{J_n}\underset{\bR}\otimes A_{n-1}^{alg}\left[n-1\right]\rightarrow \mathrm{fib}\left( \psi_{n-1}^{\mathrm{alg}}\right)$ of $A_{n-1}^{alg}$-modules that induces a surjective map $\bR^{J_n}\left[n-1\right]\underset{\bR}\otimes \pi_0\left( A_{n-1}^{alg}\right)\rightarrow \pi_{n-1}\left( \mathrm{fib}\left( \psi_{n-1}^{\mathrm{alg}}\right)\right)$. The composition \[\varphi:\bR^{J_n}\underset{\bR}\otimes A_{n-1}^{alg}\left[n-1\right]\longrightarrow \mathrm{fib}\left( \psi_{n-1}^{\mathrm{alg}}\right)\longrightarrow A_{n-1}^{alg}\]
in the $\infty$-category of $A_{n-1}^{alg}$-modules is adjoint to a map 
\[ \bR^{J_n}\left[n-1\right]\longrightarrow   A_{n-1}^{alg}\]
of $\bR$-modules. This map yields a map $\mathrm{Sym}^{\bullet}\left( \bR^{J_n}\left[n-1\right]\right)\rightarrow A_{n-1}^{alg}$ in $\Alg_{\ComR}\left( \Spc\right)$, which is in turn adjoint to a map $f:\Sigma^{n-1} \sC^{\i}\left( \left( \bR^{J_n}\right)^{\vee}\right)\rightarrow A_{n-1}$ of simplicial $C^{\i}$-rings, with $\Sigma^{n-1} \sC^{\i}\left( \left( \bR^{J_n}\right)^{\vee}\right)$ the $\left( n-1\right)^{th}$ suspension of $\sC^{\i}\left( \left( \bR^{J_n}\right)^{\vee}\right)$ at the basepoint $0\in \left( \bR^{J_n}\right)^{\vee}$. Now we define $A_{n}$ as the pushout 
\begin{equation*}
\begin{tikzcd}
\Sigma^{n-1} \sC^{\i}\left( \left( \bR^{J_n}\right)^{\vee}\right)\ar[r,"f"]\ar[d] &A_{n-1}\ar[d]\\
\bR \ar[r] & A_{n}
\end{tikzcd}    
\end{equation*}
 The canonical null-homotopy of the map $$\bR^{J_n}\underset{\bR}\otimes A_{n-1}^{alg}\left[n-1\right]\rightarrow \mathrm{fib}\left( \psi_{n-1}^{\mathrm{alg}}\right)\rightarrow A^{alg}$$ yields a homotopy between $\psi_{n-1}\circ f$ and $$\Sigma^{n-1} \sC^{\i}\left( \left( \bR^{J_n}\right)^{\vee}\right)\rightarrow \bR\rightarrow A,$$ so we get a map $\psi_n:A_n\rightarrow A$. We check that $\psi_n$ is $n$-connective: notice that the left vertical map in the diagram above induces a surjection on connected components, so by unramifiedness, we have an equivalence $$A_n^{alg}\simeq \Sigma^{n-1}\sC^{\i}\left( \left( \bR^{J_n}\right)^{\vee}\right)^{alg}\underset{\bR}\otimes A_{n-1}^{alg}.$$ For $n=1$, we observe that $$\pi_0\left( A_1\right)\simeq \pi_0\left( \sC^{\i}\left( \left( \bR^{J_0}\right)^{\vee}\right)/\pi_0\left( \mathrm{fib}\left( \psi_0^{alg}\right)\right)\right)\simeq \pi_0\left( A\right).$$ For $n>1$, Lemma \ref{looping2} provides us with a map $\mathrm{cofib}\left( \varphi\right)\rightarrow A_n$ with $\left( 2n\right)$-connective cofiber. Comparing the $\pi_{n-1}$-terms in the long exact sequence associated with the fiber sequence $$\mathrm{fib}\left( \psi^{alg}_{n-1}\right)\rightarrow A_{n-1}^{alg}\rightarrow A$$ with those of the long exact sequence associated to the cofiber sequence of $\varphi$ yields the desired connectivity estimate.
\item Let $A$ be a simplicial $C^{\i}$-ring with a cell decomposition provided by part 1 of the proof. We will show that if $B$ is a finitely presented simplicial $C^{\i}$-ring, then any morphism $B\rightarrow A$ factors through a finite cell complex. The desired statement then follows by applying this to the identity morphism $A\rightarrow A$. Choose some morphism $f:B\rightarrow A$. We have $$A\simeq \underset{{i\in \bZ_{\geq 0}}}\colim A_i,$$ so $f$ factors through some $A_i$. We prove by reverse induction that $f$ factors through a cell complex with finitely many cells in degrees greater than $j$ for every $j\leq i$. For $j=i$, we use that 
\[A_i=\bR\underset{\Sigma^{i-1}\sC^{\i}\left( \left( \bR^{J_i}\right)^{\vee}\right)}\oinfty A_{i-1}\simeq \underset{{S\subset J_i,\,|S|<\infty}}\colim \bR\underset{\Sigma^{i-1}\sC^{\i}\left( \bR^S\right)}\oinfty A_{i-1},  \] 
to deduce that the map $B\rightarrow A_i$ factors through some $$\bR\underset{\Sigma^{i-1}\sC^{\i}\left( \bR^S\right)}\oinfty A_{i-1}$$ where $S$ is a finite set. Now assume that $B\rightarrow A_i$ factors through a cell complex $\tilde{A}$ that is obtained from the object $A_j$, $j<i$, by attaching finitely many cells (in degrees $>j$). $A_j$ is itself obtained as 
$$\bR  \underset{\Sigma^{j-1}\sC^{\i}\left( \left( \bR^{J_j}\right)^{\vee}\right)}\oinfty A_{j-1},$$
where $J_{j}$ may be an infinite set. Just as in the case $i=j$, we have $$A_j\simeq \underset{{S'\subset J_j,\,|S'|<\infty}}\colim C_{S'},$$ where we write $$C_{S'}:=\bR\underset{\Sigma^{j-1}\sC^{\i}\left( \bR^{S'}\right)}\oinfty A_{j-1}.$$
By assumption on $\tilde{A}$, we attach only finitely many cells in degree $j$, given by a pushout $$\bR\underset {\Sigma^j\sC^{\i}\left( \bR^n\right)}\oinfty \underset{{S'\subset J_j,\,|S'|<\infty}}\colim C_{S'}.  $$
Because $\Sigma^j\sC^{\i}\left( \bR^n\right)$ is finitely presented, the map  $$\Sigma^j\sC^{\i}\left( \bR^n\right)\rightarrow \underset{{S'\subset J_j,\,|S'|<\infty}}\colim C_{S'}$$ factors through some $C_{S''}$, so we can write the pushout above as the colimit $$\underset{{S''\supset S',\,|S''|<\infty}}\colim \bR \underset{\Sigma^j\sC^{\i}\left( \bR^n\right)}\oinfty C_{S''}.$$ Now we repeat this argument for all cells of higher degrees, using finite presentation as there are only a finite number of cells left in each degree. We find that $\tilde{A}$ can be written as some filtered colimit $$\underset{{k\in\mathcal{J}}}\colim\tilde{A}_k,$$ where each $\tilde{A}_k$ is a relative cell complex obtained by attaching a finite number of cells to the object $A_{j-1}$. Using compactness of $B$, we see that $B\rightarrow\tilde{A}$ factors through some $\tilde{A}_k$. This completes the induction step.   
\end{enumerate}
\end{proof}

\section{The Universal Property Revisited}\label{sec:univ2}

\subsection{Derived Manifolds and $C^\i$-rings}
Recall that $\DMfd$ was defined (Definition \ref{dfn:DMfd}) as the unique idempotent complete $\i$-category with finite limits equipped with a functor $$i:\Mfd \to \DMfd$$ which preserves transverse pullbacks and the terminal object, such that for any other idempotent complete $\i$-category $\sC$ with finite limits, composition with $i$ induces an equivalence of $\i$-categories
$$\Fun^{\lex}\left(\DMfd,\sC\right) \stackrel{\sim}{\longlongrightarrow} \Fun^{\pitchfork}\left(\Mfd,\sC\right)$$
between functors from derived manifolds to $\sC$ which preserve finite limits, to functors from manifolds which preserve transverse pullbacks and the terminal object.

\begin{lem}\label{lem:rkn}
Let $q:\Cart \hookrightarrow \Mfd$ be the fully faithful inclusion. Let $\sC$ be an idempotent complete $\i$-category with finite limits. Then the following conditions are equivalent for a functor $F:\Cart \to \sC$:
\begin{enumerate}
\item $F$ preserves finite products
\item The right Kan extension $\Ran_qF$ exists and preserves transverse pullbacks and the terminal object.
\end{enumerate}
\end{lem}

\begin{proof}
$(2) \Rightarrow (1)$ since $q^*\Ran_q \simeq id,$ since $q$ is fully faithful. Conversely, suppose that $(1)$ holds. First, assume that $\sC$ is both complete and cocomplete. We have a commutative diagram
$$\xymatrix@C=2.5cm{\Cart \ar[d]_-{q} \ar[rd]^-{j^\fp} & \\ \Mfd \ar[r]_-{\sC^\i} & \left(\Alg_{\Cart}\left(\Spc\right)^{\fp}\right)^{op}.}$$ By Theorem \ref{thm:finenv}, since $\sC^\i$ is fully faithful by Lemma \ref{lem:ff}, the right Kan extension $$\Ran_{j^{\fp}} F:\left(\Alg_{\Cart}\left(\Spc\right)^{\fp}\right)^{op} \to \sC$$ exists and preserves finite limits. 
By Lemma \ref{lem:crucial}, it follows that the right Kan extension $\Ran_q F$ exists and is given by $$\left(\Ran_{j^{\fp}} F\right) \circ \sC^\i.$$ By Theorem \ref{manifoldssmoothring}, it follows that $\Ran_q F$ preserves transverse pullbacks and the terminal object. 

Now relax the (co)completeness assumptions on $\sC$. Consider the composite $$\Cart \stackrel{F}{\longrightarrow} \sC \stackrel{y}{\longhookrightarrow} \Psh\left(\sC\right).$$ Then, as it preserves finite products, by the above paragraph, the right Kan extension $\Ran_q \left(y \circ F\right)$ exists and preserves transverse pullbacks and the terminal object. The $\i$-category $\Psh\left(C\right)$ is complete and therefore this right Kan extension can be computed with the standard pointwise formula. Since the Yoneda embedding preserves all limits, it suffices to show that $\Ran_q \left(y \circ F\right)$ takes values in representables, since this would imply that the limits needed for the pointwise formula for $\Ran_q \left(\circ F\right)$ exist in $\sC,$ and that $\Ran_q \left(y\circ F\right)$ preserves transverse pullbacks and the terminal object. Let $M$ be a manifold. Then $M$ is a retract of a transverse pullback of Cartesian manifolds, by Corollary \ref{cor:retractrans}. For any Cartesian manifold, $\bR^n,$ $$\Ran_q \left(y \circ F\right)\left(\bR^n\right)\simeq y\left(F\left(\bR^n\right)\right)$$ is in the essential image of $y,$ and since $\Ran_q \left(y \circ F\right)$ preserves finite limits (and retracts), $\Ran_q \left(y \circ F\right)\left(M\right)$ is a retract of a pullback of representables, hence representable, as $\sC$ has retracts and the Yoneda embedding preserves all small limits.
\end{proof}

\begin{cor}\label{cor:eqtrans}
Let $\sC$ be an idempotent complete $\i$-category with finite limits. Then $$q^*:\Fun\left(\Mfd,\sC\right) \to \Fun\left(\Cart,\sC\right)$$ restricts to an equivalence of $\i$-categories
$$q^*:\Fun^\pitchfork\left(\Mfd,\sC\right) \to \Fun^{\pi}\left(\Cart,\sC\right)$$ between functors which preserves transverse pullbacks and the terminal object and functors which preserve finite products.
\end{cor}

\begin{proof}
By Lemma \ref{lem:rkn}, there exists a global right Kan extension functor
$$\Ran_q:\Fun^{\pi}\left(\Cart,\sC\right) \to \Fun^\pitchfork\left(\Mfd,\sC\right)$$ which, by the universal property of Kan extensions, is right adjoint to $q^*$. Since $q$ is fully faithful, the counit $$q^*\Ran_q \to id$$ is an equivalence. It suffices to show that the unit $$id \to \Ran_q q^*$$ is an equivalence as well. Let $F:\Mfd \to \sC$ be in $\Fun^\pitchfork\left(\Mfd,\sC\right).$ Then since both $F$ and $\Ran_q q^*F$ preserve transverse pullbacks (and retracts), and by Corollary \ref{cor:retractrans}, every manifold is a retract of a transverse pullback of Cartesian manifolds, it follows that the unit must be an equivalence. 
\end{proof}

\begin{thm}\label{thm:univ2}
For all idempotent complete $\i$-categories $\sC$ with finite limits, composition with $$i\circ q:\Cart \to \DMfd$$ induces an equivalence of $\i$-categories
$$\Fun^{\lex}\left(\DMfd,\sC\right) \stackrel{\sim}{\longlongrightarrow} \Fun^{\pi}\left(\Cart,\sC\right)=\Alg_{\Cart}\left(\sC\right)$$
between the $\i$-category of left exact functors from derived manifolds to $\sC,$ and the $\i$-category of $C^\i$-rings in $\sC.$
\end{thm}

\begin{proof}
By definition $$i^*:\Fun^{\lex}\left(\DMfd,\sC\right) \stackrel{\sim}{\longlongrightarrow} \Fun^{\pitchfork}\left(\Mfd,\sC\right)$$ is an equivalence. And by Corollary \ref{cor:eqtrans} $$q^*:\Fun^{\pitchfork}\left(\Mfd,\sC\right) \to \Fun^{\pi}\left(\Cart,\sC\right)=\Alg_{\Cart}\left(\sC\right)$$ is an equivalence.
\end{proof}

\begin{cor}\label{cor:fparedms}
There is a canonical equivalence of $\i$-categories $$\DMfd \simeq \left(\Alg_{\Cart}\left(\Spc\right)^{\fp}\right)^{op}.$$
In fact, this equivalence is the restriction of the functor $\cO_{\DMfd}$ from Example \ref{ex:sheaf}.
\end{cor}

\begin{proof}
The existence of an equivalence follows immediately since the pair $\left(\DMfd,i \circ q\right)$ satisfies the same universal property as $\left(\left(\Alg_{\Cart}\left(\Spc\right)^{\fp}\right)^{op},j^{\fp}\right).$ In particular, $i \circ q$ is a universal $C^\i$-ring, and, by Remark \ref{rmk:universalgamma}, the equivalence can be concretely realized as
$$\Gamma_{i \circ q} = \cO_{\DMfd}^{op}:\DMfd \to \left(\Alg_{\Cart}\left(\Spc\right)^{\fp}\right)^{op},$$
with inverse given by $\Ran{j^{\fp}}\left(i \circ q\right).$
\end{proof}

\begin{prop}\label{prop:subcan}
The Grothendieck topology $J_{\DMfd}$ on $\DMfd$ of Definition \ref{dfn:JDMfd} is subcanonical.
\end{prop}

\begin{proof}
Note that the following diagram commutes
$$\xymatrix{ & \Mfd \ar[ld]_-{i} \ar[rd]^-{\sC^\i} &\\
\DMfd \ar[d]^-{U} & & \left(\Alg_{\Cart}\left(\Spc\right)^{\fp}\right)^{op} \ar[ll]^-{\Ran_{j^{\fp}}\left(i \circ q\right)} \ar[lld]^-{Sp}\\
\Top & & }$$
since both $U$ and $Sp$ are Kan extended from $\Mfd.$ So under the equivalence $\Ran{j^{\fp}}\left(i \circ q\right),$ $J_{\DMfd}$ corresponds to the Grothendieck topology on $\left(\Alg_{\Cart}\left(\Spc\right)^{\fp}\right)^{op}$ induced from $\Top$ via $Sp.$ Lets call this Grothendieck topology $J.$ It suffices to show that $J$ is subcanonical. But this follows immediately from the fact that $\Speci$ is fully faithful and the open cover topology on $\Loc$ is subcanonical.
\end{proof}

\begin{prop}\label{prop:Osheaf}
The functor $$\cO_{\DMfd}:\DMfd^{op} \to \Alg_{\Cart}\left(\Spc\right)$$ is a $J_{\DMfd}$-sheaf.
\end{prop}

\begin{proof}
Under the equivalence $$\DMfd \simeq \left(\Alg_{\Cart}\left(\Spc\right)\right)^{op},$$ the universal $C^\i$-ring object $i \circ q$ corresponds to $j^{\fp}.$ Note that as a finite product preserving functor $$\Cart \to \Spc$$ for a derived manifold $\cM,$
$$\cO_{\DMfd}\left(\cM\right)=\Map\left(\cM,i\circ q\right).$$ Notice that in $\left(\Alg_{\Cart}\left(\Spc\right)\right)^{op},$ for $A$ a finitely presented algebra,
$$\Map_{\left(\Alg_{\Cart}\left(\Spc\right)\right)^{op}}\left(A,j^{\fp}\right) \simeq A.$$ So the result now follows since the Grothendieck topology $J$ is subcanonical.
\end{proof}

\begin{prop}
For any idempotent complete $\i$-category $\sC$ with finite limits, under the equivalence $$\Fun^{\lex}\left(\DMfd,\sC\right) \stackrel{\sim}{\longlongrightarrow} \Fun^{\pi}\left(\Cart,\sC\right)=\Alg_{\Cart}\left(\sC\right)$$ a $C^\i$-ring $S$ in $\sC$ corresponds to a fully faithful left exact functor $\DMfd \hookrightarrow \sC$ if and only if it is versal and corresponds to an equivalence, if and only if it is universal.
\end{prop}

\begin{proof}
This follows immediately form Corollary \ref{cor:fparedms} and Theorem \ref{thm:universal1}.
\end{proof}

\begin{cor}
Let $\sC$ be an idempotent complete $\i$-category with finite limits. Recall that $\mathbf{Dom}$ is the full subcategory of $\Mfd$ on open domains. The following are equivalent for a functor $F:\Mfd \to \sC$
\begin{itemize}
\item[1.] $F$ preserve transverse pullbacks and the terminal object
\item[2.] The restriction of $F$ to $\mathbf{Dom}$ preserves transverse pullbacks and the terminal object.
\end{itemize}
\end{cor}

\begin{proof}
Clearly $1) \Rightarrow 2).$ Suppose that $F:\Mfd \to \sC$ and $F|_{\mathbf{Dom}}$ preserves transverse pullbacks. Since $\Mfd$ is the Karoubi envelope of $\mathbf{Dom},$ the restriction functor
$$\varphi^*:\Fun\left(\Mfd,\sC\right) \to \Fun\left(\mathbf{Dom},\sC\right)$$ is an equivalence. Since the Yoneda embedding preserves and reflects limits, we may assume that $\sC$ is complete. Hence, $\varphi^*$ has a right adjoint given by global right Kan extension $\Ran_{\varphi}.$ This implies that $$F \simeq \Ran_{\varphi} F|_{\mathbf{Dom}}.$$ Notice however that all of the results and their proofs carry over completely analogously by replacing the category $\Mfd$ with $\mathbf{Dom}.$ It follows that we can identify $\Ran_{\varphi} F|_{\mathbf{Dom}}$ with $\Ran_q\left(F|_{\Cart}\right),$ and hence conclude that $F$ preserves transverse pullbacks and the terminal object.
\end{proof}

\subsection{Comparison with Spivak's Model}

In \cite[Section 6]{spivak}, Spivak defines a simplicial category of quasi-smooth derived manifolds as a subcategory of a simplicial category $\mathbf{LRS}$ of spaces locally ringed in homotopical $C^\i$-rings. Spivak defines a pair $\left(X,\cO_X\right)$ to be a (quasi-smooth) derived manifold if it is locally equivalent to a pullback
$$\xymatrix{Rf^{-1}\left(0\right) \ar[r] \ar[d] & \bR^{0} \ar[d]^-{0}\\
\bR^n \ar[r]^{f} & \bR.}$$
In particular, by \cite[Theorem 8.15]{spivak}, any pullback of the form
$$\xymatrix{M \times^{\i}_L N \ar[r] \ar[d] & N \ar[d]^-{0}\\
M \ar[r] & L}$$ is a quasi-smooth derived manifold in Spivak's sense. On one hand, Spivak's objects are very general as he imposes no separation conditions on the underlying space $X$ (such as being Hausdorff or paracompact). On the other hand, by \cite[Remark 8.16]{spivak}, quasi-smooth derived manifolds are not closed under fibered products.

We offer the following variant of Spivak's model:

\begin{dfn}
Denote by $\mathbf{dMan}_{\mathit{Spivak}}$ the smallest full simplicial subcategory of $\mathbf{LRS}$ closed under finite homotopy limits and retracts containing the essential image of $\Mfd$.
\end{dfn}

\begin{thm}
There is a canonical equivalence of $\i$-categories between the homotopy coherent nerve of $\mathbf{dMan}_{\mathit{Spivak}}$ and $\DMfd.$
\end{thm}

\begin{proof}
For simplicity of notation, let us not distinguish notationally between $\mathbf{LRS}$ and its homotopy coherent nerve, and similarly for $\mathbf{dMan}_{\mathit{Spivak}}$. Recall from Remark \ref{rmk:mappingsp} that we can identify (the homotopy coherent nerve) of $\mathbf{LRS}$ with the full subcategory of the $\i$-category $\Loc$ on those objects $\left(X,\cO_X\right)$ for which $\cO_X$ is a hypersheaf.  Since for any space $X,$ hypersheaves of spaces are a reflective subcategory of (\v{C}ech) sheaves of spaces, and since for any algebraic theory $\bT,$ the forgetful functor $$\Alg_{\bT}\left(\Spc\right) \to \Spc$$ preserves all limits, it follows that hypersheaves of $\Cart$-algebras are closed under limits. So, by Remark \ref{rmk:limitconc}, it follows that $\mathbf{LRS}$ is stable under finite limits and retracts in $\Loc.$ Hence  of $\mathbf{dMan}_{\mathit{Spivak}}$ can be identified with the smallest subcategory of $\Loc$ closed under finite limits and retracts containing the essential image of $\Mfd.$ Consider the canonical $C^\i$-ring object $\left(\bR,\sC^\i_\bR\right).$ Then, by Theorem \ref{thm:univ2}, there is a canonically induced functor $$\varphi:\DMfd \to \Loc$$ given by $\varphi=\Ran_{i \circ q}\left(\underline{\bR}\right).$ Unwinding definitions, under the equivalence $$\DMfd \simeq \left(\Alg_{\Cart}\left(\Spc\right)\right)^{op},$$ $\varphi$ corresponds to the functor $\Speci,$ which is fully faithful by Theorem \ref{thm:speciff}. By Proposition \ref{lem:retcol}, it follows that the essential image of $\varphi$ is the smallest subcategory containing $\left(\bR,\sC^\i_\bR\right)$ which is closed under finite limits and retracts. But by Theorem \ref{manifoldssmoothring} and Corollary \ref{cor:retractrans}, we conclude that this essential image by be identified with $\mathbf{dMan}_{\mathit{Spivak}}.$
\end{proof}

Finally, to justify this variant, we offer the following:

\begin{prop}\label{prop:ft1}
Suppose that $\left(X,\cO_X\right) \in \Loc$ is such that $X$ is paracompact Hausdorff, and there is a cover $\left(U_\alpha \hookrightarrow X\right)$ of $X$ such that for all $\alpha$  $\left(U_\alpha,\cO_X|_{U_\alpha}\right)$ is equivalent to $\Speci\left(A_\alpha\right)$ for some finitely presented $A_\alpha$ in $\Alg_{\Cart}\left(\Spc\right).$ If $\Gamma\left(\cO_X\right)$ is finitely presented, then $\left(X,\cO_X\right) \simeq \Speci\left(\Gamma\left(\cO_X\right)\right).$ In particular, $\left(X,\cO_X\right)$ is in $\mathbf{dMan}_{\mathit{Spivak}}.$
\end{prop}

\begin{proof}
For $A$ finitely presented, $\pi_0\cO_A$ is fine, and hence soft. Since softness is a local property on paracompact Hausdorff spaces, we conclude that $\pi_0\cO_X$  is also soft. Also, since hyper-completeness is a local property, Remark \ref{rmk:hyper} implies that $\cO_X$ is a hypersheaf. Since $\Gamma\left(\cO_X\right)$ is finitely presented, by Lemma \ref{lem:Ww}, we have a canonical morphism
$$\left(X,\cO_X\right) \to \Speci\left(\Gamma\left(\cO_X\right)\right).$$ By \cite{cinfsch} and \cite[Proposition 2]{derivjustden} it follows that this morphism induces an isomorphism
$$\left(X,\pi_0\cO_X\right) \to \Speci\left(\Gamma\left(\pi_0\cO_X\right)\right),$$ and by softness, we can identify 
$\Gamma\left(\pi_0\cO_X\right)$ with $\pi_0\Gamma\left(\cO_X\right).$ Denote by $\alpha:\cO_X \to \cO_{\Gamma\left(\cO_X\right)}$ the canonical map. By hypercompleteness, it suffices to show that the map on sheaves of homotopy groups are isomorphisms. We have just seen that this map induces an isomorphism on $\pi_0.$  The result now follows  from the softness of the sheaves involved, the collapse of the hypercohomology spectral sequence, and \cite[Proposition 5.20]{joycesch}.
\end{proof}

\begin{cor}\label{cor:ft2}
If $\left(X,\cO_X\right) \in \mathbf{LRS}$ is a quasi-smooth derived manifold in the original sense of Spivak, if $X$ is paracompact Hausdorff and $\Gamma\left(\cO_X\right)$ is finitely presented, then $$\left(X,\cO_X\right) \in \mathbf{dMan}_{\mathit{Spivak}}.$$
\end{cor}

\begin{rmk}
In light of Proposition \ref{prop:ft1} and Corollary \ref{cor:ft2}, one may consider objects $\left(X,\cO_X\right)$ satisfying the conditions of Proposition \ref{prop:ft1} to be derived manifolds \emph{of finite type}, and allow any object locally isomorphic to $\Speci\left(A\right)$ for varying finitely presented $A$ to be a derived manifold. Note, any derived manifold would then be locally of finite type. Note furthermore however that any derived manifold locally of finite type is faithfully represented by its functor of points on finite type derived manifolds c.f. \cite[Theorem 2.2]{higherme}.
\end{rmk}

\bibliographystyle{hplain}
\bibliography{research}

\end{document}